\documentclass[british]{amsart}

\pdfoutput=1

\usepackage{babel}
\usepackage[T1]{fontenc}
\usepackage[utf8]{inputenc}
\usepackage[varg]{pxfonts}

\usepackage{needspace}
\usepackage{setspace}
\usepackage{aliascnt}
\usepackage{varioref}
\labelformat{equation}{\textup{(#1)}}
\labelformat{enumi}{\textup{(#1)}}

\usepackage{amssymb,mathrsfs}
\usepackage{subfigure}

\usepackage[all,arc]{xy}

\usepackage{ifpdf}
\ifpdf
\IfFileExists{xypdf}{
\xyoption{pdf}
}{}
\usepackage{microtype}
\fi

\usepackage[pdfusetitle, unicode=true, bookmarks=true, pdfborder={0 0 0}, colorlinks=false, breaklinks=true, linktoc=all]{hyperref}

\hypersetup{breaklinks=false, colorlinks=true, linkcolor=blue, citecolor=blue, urlcolor=blue, linktoc=page} 

\title{Moduli spaces of Klein surfaces and related operads}

\author{Christopher Braun}

\theoremstyle{theorem}
\newtheorem{theorem}{Theorem}[section]
\newcommand{\newautoreftheorem}[2]{\newaliascnt{#1}{theorem}\newtheorem{#1}[#1]{#2}\aliascntresetthe{#1}\expandafter\def\csname #1autorefname\endcsname{#2}}

\newautoreftheorem{proposition}{Proposition}
\newautoreftheorem{lemma}{Lemma}
\newautoreftheorem{corollary}{Corollary}
\newtheorem*{theorem*}{Theorem}

\theoremstyle{definition}
\newautoreftheorem{definition}{Definition}
\newautoreftheorem{remark}{Remark}
\newautoreftheorem{example}{Example}
\newautoreftheorem{examples}{Examples}
\newautoreftheorem{notation}{Notation}
\newautoreftheorem{convention}{Convention}
\newtheorem*{notation*}{Notation}

\numberwithin{equation}{section} 
\numberwithin{figure}{section}   

\newcommand{\cstart}{\xybox{*\xycircle(1,3){-}}}
\newcommand{\cstop}{\xybox{\ellipse(1,3){.}\ellipse(1,3)_,=:a(180){-}}}

\newcommand{\handle}{\xybox{(4,0.9)="A";"A"-(8,0.5)**\crv{"A"-(4,3.5)},(3,0.3)="A";"A"-(6,0.4)**\crv{"A"+(-4.3,1.5)}}}
\newcommand{\puncture}{\xybox{*\xycircle(2,2){-}}}
\newcommand{\crosscap}{\xybox{{\ellipse(2.5,2){.}}*\xycircle(0,2){-}*\xycircle(2.5,0){-}}}
\newcommand{\cobord}[8][(10,0)]{\xybox{@=#6@@{@i@={#6}@@{*{\handle}}@i}
@=#7@@{@i@={#7}@@{*{\puncture}}@i}@=#8@@{@i@={#8}@@{*{\crosscap}}@i},
(0,0)="d"@=#2@@{@i-(0,3)="B1"="T1"@={#2}
@@{;"T1";*{#4}-(0,3)="b"**\crv{"T1"+"d"&"b"+"d"}+(0,6)="T1",(5,0)="d"}@i},
(0,0)="d"@=#3@@{@i-(0,3)="B2"="T2"@={#3}
@@{;"T2";*{#5}-(0,3)="b"**\crv{"T2"-"d"&"b"-"d"}+(0,6)="T2",(5,0)="d"}@i},
#1="S"@=@(@=#2@@{@i@(@=@(@=#3@@{@i,
"T1";"T2"**\crv{"T1"+"S"&"T2"-"S"},"B1";"B2"**\crv{"B1"+"S"&"B2"-"S"}
@)}@i@)@@{@i,"T1";"B1"**\crv{"T1"+"S"&"B1"+"S"}@)}@)}
@i@)@@{@i@=#3@@{@i,"T2";"B2"**\crv{"T2"-"S"&"B2"-"S"}}@)}}}
\newcommand{\basiccob}[5][(10,0)]{\cobord[#1]{#2}{#3}{#4}{#5}{@i}{@i}{@i}}

\newcommand{\cylinder}[2]{\basiccob{(0,0)}{(20,0)}{#1}{#2}}
\newcommand{\shortcylinder}[2]{\basiccob{(0,0)}{(10,0)}{#1}{#2}}
\newcommand{\twist}[2]{\xybox{*{\basiccob[(5,0)]{(0,0)}{(20,12)}{#1}{#2}}*{\basiccob[(5,0)]{(0,12)}{(20,0)}{#1}{#2}}}}
\newcommand{\pants}[2]{\basiccob{(0,0),(0,12)}{(20,6)}{#1}{#2}}
\newcommand{\pair}[1]{\basiccob[(15,0)]{(0,0),(0,12)}{@i}{#1}{}}
\newcommand{\copants}[2]{\basiccob{(0,6)}{(20,0),(20,12)}{#1}{#2}}
\newcommand{\copair}[1]{\basiccob[(15,0)]{@i}{(20,0),(20,12)}{}{#1}}
\newcommand{\birth}[1]{\basiccob{@i}{(0,0)}{}{#1}}
\newcommand{\death}[1]{\basiccob{(0,0)}{@i}{#1}{}}
\newcommand{\rp}[1]{\xybox{*{\birth{}},(10,0)*{\shortcylinder{}{#1}},(7,0)*{\crosscap}}}
\newcommand{\flip}[2]{\xybox{(0,0)*{#1},(10,0)*{#2},(0,3);(10,-3)**\crv{(3,3)&(7,-3)},(0,-3);(10,3)**\crv{(3,-3)&(7,3)}}}

\newcommand{\co}{\colon\thinspace}
\newcommand{\dual}{\mathbf{D}}
\newcommand{\mob}{\mathrm{M}}
\newcommand{\modc}[1]{\overline{#1}}
\newcommand{\ctft}{\mathrm{TFT}}
\newcommand{\otft}{\mathrm{OTFT}}
\newcommand{\cktft}{\mathrm{pKTFT}}
\newcommand{\oktft}{{\mathrm{OKTFT}}}
\newcommand{\com}{\mathcal{C}om}
\newcommand{\ass}{\mathcal{A}ss}
\newcommand{\mcom}{\mob\com}
\newcommand{\mass}{\mob\ass}
\newcommand{\dass}{\dual\ass}
\newcommand{\dmass}{\dual\mass}
\newcommand{\modcom}{\modc{\com}}
\newcommand{\modass}{\modc{\ass}}
\newcommand{\modmcom}{\modc{\mcom}}
\newcommand{\modmass}{\modc{\mass}}
\newcommand{\moddass}{\modc{\dass}}
\newcommand{\moddmass}{\modc{\dmass}}
\newcommand{\no}{\mathcal{N}}
\newcommand{\ko}{\mathcal{K}}
\newcommand{\mo}{\mathcal{M}^\mathbb{R}}
\newcommand{\nb}{\overline{\no}}
\newcommand{\kb}{\overline{\ko}}
\newcommand{\mb}{\overline{\mo}}
\newcommand{\dr}{D^\mathbb{R}}
\newcommand{\degree}[1]{\bar{#1}}
\DeclareMathOperator{\vertices}{Vert}
\DeclareMathOperator{\halfedges}{Half}
\DeclareMathOperator{\edges}{Edge}
\DeclareMathOperator{\flags}{Flag}
\DeclareMathOperator{\legs}{Leg}
\DeclareMathOperator{\inputs}{In}
\DeclareMathOperator{\ob}{Ob}
\DeclareMathOperator{\iso}{Iso}
\DeclareMathOperator{\Hom}{Hom}
\DeclareMathOperator{\End}{End}
\DeclareMathOperator{\Det}{Det}
\newcommand{\cob}{\mathbf{2Cob}}
\newcommand{\cobo}{\mathbf{2Cob}^{\mathrm{o}}}
\newcommand{\cobcl}{\mathbf{2Cob}^{\mathrm{cl}}}
\newcommand{\kcob}{\mathbf{2KCob}}
\newcommand{\kcobo}{\mathbf{2KCob}^{\mathrm{o}}}
\newcommand{\kcobcl}{\mathbf{2KCob}^{\mathrm{cl}}}
\newcommand{\vect}{\mathbf{Vect}_k}
\newcommand{\dgvect}{\mathbf{dgVect}_k}
\newcommand{\topol}{\mathbf{Top}}
\newcommand{\dgop}{\mathbf{dgOp}}
\newcommand{\klein}{\mathbf{Klein}}
\newcommand{\dklein}{\mathbf{dKlein}}
\newcommand{\symriem}{\mathbf{SymRiem}}
\newcommand{\dsymriem}{\mathbf{dSymRiem}}
\newcommand{\nklein}{\mathbf{nKlein}}
\newcommand{\dnklein}{\mathbf{dnKlein}}
\newcommand{\nsymriem}{\mathbf{nSymRiem}}
\newcommand{\dnsymriem}{\mathbf{dnSymRiem}}

\begin{document}
\def\sectionautorefname{Section}

\begin{abstract}
We consider the extension of classical $2$--dimensional topological quantum field theories to Klein topological quantum field theories which allow unorientable surfaces. We approach this using the theory of modular operads by introducing a new operad governing associative algebras with involution. This operad is Koszul and we identify the dual dg operad governing $A_\infty$--algebras with involution in terms of M\"obius graphs which are a generalisation of ribbon graphs. We then generalise open topological conformal field theories to open Klein topological conformal field theories and give a generators and relations description of the open KTCFT operad. We deduce an analogue of the ribbon graph decomposition of the moduli spaces of Riemann surfaces: a M\"obius graph decomposition of the moduli spaces of Klein surfaces (real algebraic curves). The M\"obius graph complex then computes the homology of these moduli spaces. We also obtain a different graph complex computing the homology of the moduli spaces of admissible stable symmetric Riemann surfaces which are partial compactifications of the moduli spaces of Klein surfaces.
\end{abstract}

\maketitle 
\tableofcontents

\section*{Introduction}
One property of the original axiomatic definition by Atiyah \cite{atiyah} of a topological quantum field theory (TFT) is that all the manifolds considered are oriented. Alexeevski and Natanzon \cite{alexeevskinatanzon} considered a generalisation to manifolds that are not oriented (or even necessarily orientable) in dimension $2$. An unoriented TFT in this sense is then called a Klein topological quantum field theory (KTFT).

It is well known that $2$--dimensional closed TFTs are equivalent to commutative Frobenius algebras and open TFTs are equivalent to symmetric (but not necessarily commutative) Frobenius algebras, for example see Moore \cite{moore2} and Segal \cite{segal}. Theorems of this flavour identifying the algebraic structures of KTFTs have also been shown. In the language of modular operads, developed by Getzler and Kapranov \cite{getzlerkapranov}, these results for oriented TFTs say that the modular operads governing closed and open TFTs are $\modcom$ and $\modass$ which are the modular closures (the smallest modular operad containing a cyclic operad) of $\com$ and $\ass$, which govern commutative and associative algebras.

It is also possible to generalise TFTs by adding extra structure to our manifolds such as a complex structure which gives the notion of a topological conformal field theory (TCFT). We can also find topological modular operads governing TCFTs constructed from moduli spaces of Riemann surfaces.

The ribbon graph decomposition of moduli space is an orbi-cell complex homeomorphic to $\mathcal{M}_{g,n}\times \mathbb{R}^n_{>0}$ with cells labelled by ribbon graphs, introduced in Harer \cite{harer} and Penner \cite{penner}. Ribbon graphs arise from the modular closure of the $A_\infty$ operad (cf Kontsevich \cite{kontsevich}). Indeed the cellular chain complex of the operad given by gluing stable holomorphic discs with marked points on the boundary is equivalent to the $A_\infty$ operad and can be thought of as the genus $0$ part of the operad governing open TCFTs. It was shown by Kevin Costello \cite{costello1,costello2} that this gives a dual point of view on the ribbon graph decomposition of moduli space: The operad governing open TCFTs is homotopy equivalent to the modular closure of the suboperad of conformal discs and so this gives a quasi-isomorphism on the chain complex level to the modular closure of the $A_\infty$ operad. The moduli spaces underlying the open TCFT operad are those of stable Riemann surfaces with boundary and marked points on the boundary. In particular this yields new proofs of ribbon graph complexes computing the homology of these moduli spaces.

We wish to consider the corresponding theory for KTFTs. Alexeevski and Natanzon \cite{alexeevskinatanzon} considered open--closed KTFTs and Turaev and Turner \cite{turaevturner} considered just closed KTFTs. We will concentrate mainly on the open version in order to parallel the theory outlined above. We begin by recasting the definitions for KTFTs in terms of modular operads. We show that the open KTFT operad is given by the modular closure of the cyclic operad $\mass$ which is the operad governing associative algebras with involution. The corresponding notion of a ribbon graph, a M\"obius graph, is also developed to identify $\mass$ and the various operads obtained from it. On the other hand the closed KTFT operad is not the modular closure of a cyclic operad.

We then generalise to the Klein analogue of open TCFTs (open KTCFTs). The correct notion here of an `unoriented Riemann surface' is a Klein surface, where we allow transition functions between charts to be anti-analytic. Alling and Greenleaf \cite{allinggreenleaf} developed some of the classical theory of Klein surfaces and showed that Klein surfaces are equivalent to smooth projective real algebraic curves. We find appropriate partial compactifications of moduli spaces of Klein surfaces which form the modular operad governing open KTCFTs. We also consider other different (although more common) partial compactifications giving rise to a quite different modular operad. The underlying moduli spaces of this latter operad are spaces of `admissible' stable symmetric Riemann surfaces (which are open subspaces of the usual compactifications containing all stable symmetric surfaces).

By following the methods of Costello \cite{costello1,costello2} we can obtain graph decompositions of these moduli spaces. Precisely this means we find orbi-cell complexes homotopy equivalent to these spaces with each orbi-cell labelled by a type of graph. As a consequence we see that open KTCFTs are governed by the modular closure of the operad governing $A_\infty$--algebras with involution and we obtain a M\"obius graph complex computing the homology of the moduli spaces of smooth Klein surfaces. We also obtain a different graph complex computing the homology of the other partial compactifications.

This work was done as part of a PhD at the University of Leicester under the supervision of Andrey Lazarev.

\begin{notation*}
Throughout this paper $k$ will denote a field which, for simplicity and convenience, we will normally assume to be $\mathbb{Q}$ unless stated otherwise. Many of the definitions and results should of course work over more general fields.
\end{notation*}

\subsection*{Outline and main results}
The first two sections provide the necessary background and notation. In the first section definitions of topological quantum field theories and their Klein analogues are briefly introduced in terms of symmetric monoidal categories of cobordisms. The known results concerning the structure of KTFTs are stated and we provide some pictures that hopefully shed light on how these results arise. In the second section the definitions from the theory of modular operads that we use is recalled and the cobar construction is outlined. We include a slight generalisation of modular operads: extended modular operads (which is very similar to the generalisation of Chuang and Lazarev \cite{chuanglazarev}). For the reader familiar with modular operads this section will likely be of little interest apart from making clear the notation used here.

The third section introduces the open KTFT modular operad denoted $\oktft$. M\"obius trees and graphs are discussed in detail and the operad $\mass$ is defined in terms of M\"obius trees. This is the operad governing associative algebras with an involutive anti-automorphism. We then show $\oktft\cong\modmass$ thereby providing a generators and relations description of $\oktft$ in terms of M\"obius graphs. We show $\mass$ is its own quadratic dual, is Koszul and identify the dual dg operad $\dmass$ (governing $A_\infty$--algebras with an involution) and its modular closure. Finally we generalise our construction and discuss the closed KTFT operad, showing that only part of the closed KTFT operad is the modular closure of an operad $\mcom$.

In the fourth section we generalise to open KTCFTs. We discuss the necessary definitions and theory of Klein surfaces and nodal Klein surfaces. A subtlety arises when considering nodal surfaces and we find there are two different natural notions of a node. We provide some clarity on this difference by establishing some equivalences of categories: we show that one sort of nodal Klein surface is equivalent to a certain sort of symmetric nodal Riemann surface with boundary and the other is equivalent to a certain sort of symmetric nodal Riemann surface without boundary. We obtain moduli spaces $\kb_{g,u,h,n}$ of stable nodal Klein surfaces with $g$ handles, $u$ crosscaps, $h$ boundary components and $n$ oriented marked points using one definition of a node. We also obtain quite different moduli spaces $\mb_{\tilde{g},n}$ of `admissible' stable symmetric Riemann surfaces without boundary of genus $\tilde{g}$ and $n$ fixed marked points using the other definition. The spaces $\kb_{g,u,h,n}$ are homotopy equivalent to their interiors which are the spaces $\ko_{g,u,h,n}$ of smooth Klein surfaces with oriented marked points. The spaces $\mb_{\tilde{g},n}$ are partial compactifications of the spaces $\mo_{\tilde{g},n}$ of smooth symmetric Riemann surfaces, which are the same as the spaces of smooth Klein surfaces with unoriented marked points. Let $D_{g,u,h,n}\subset\kb_{g,u,h,n}$ be the locus of surfaces such that each irreducible part is a disc. Let $\dr_{\tilde{g},n}$ be the corresponding subspace of $\mb_{\tilde{g},n}$. We obtain topological modular operads $\kb$ and $\mb$ by gluing at marked points. The operad $\kb$ gives the correct generalisation governing open KTCFTs. We then show the inclusions of the suboperads arising from the spaces $D_{g,u,h,n}$ and $\dr_{\tilde{g},n}$ are homotopy equivalences.
\begin{theorem*}
\Needspace*{4\baselineskip}\mbox{}
\begin{itemize}
\item The inclusion $D\hookrightarrow\kb$ is a homotopy equivalence of extended topological modular operads.
\item The inclusion $\dr\hookrightarrow\mb$ is a homotopy equivalence of extended topological modular operads.
\end{itemize}
\end{theorem*}

Applying an appropriate chain complex functor $C_*$ from topological spaces to dg vector spaces over $\mathbb{Q}$ we obtain dg modular operads and the above result translates to:
\begin{theorem*}
There are quasi-isomorphism of extended dg modular operads over $\mathbb{Q}$
\begin{gather*}
C_*(D)\simeq C_*(\kb)\\
C_*(D)/(a=1)\simeq C_*(\mb)
\end{gather*}
where $a\in C_*(D)((0,2))\cong\mathbb{Q}[\mathbb{Z}_2]$ is the involution.
\end{theorem*}

The spaces $D_{g,u,h,n}$ decompose into orbi-cells labelled by M\"obius graphs and so we can identify the cellular chain complexes $C_*(D)$ in terms of the operad $\mass$ so that $C_*(D)\cong\moddmass$. Therefore we see that an open KTCFT is a Frobenius $A_\infty$--algebra with involution and we also obtain M\"obius graph complexes computing the homology of the moduli spaces of smooth Klein surfaces as well as different graph complexes (arising from $\moddmass/(a=1)$) computing the homology of the partial compactifications given by $\mb$. Unlike $H_\bullet(\kb)$, the genus $0$ part of $H_\bullet(\mb)$ has non-trivial components in higher degrees. The gluings for the operad $\mb$ can be thought of as `closed string' gluings similar to those for the Deligne--Mumford operad.

We finish by unwrapping our main theorems to give concrete and elementary descriptions of the different graph complexes and explain the isomorphisms of homology without reference to operads.

\section{Topological quantum field theories}
Since we will be working in dimension $2$ we restrict our definitions to dimension $2$. It is possible to give definitions in arbitrary dimension easily for closed field theories, but for open field theories it is necessary to mention manifolds with faces in order to glue cobordisms properly. This is a technical concern not of interest to us here. We will first briefly recall the details of oriented topological quantum field theories and then define a Klein topological quantum field theory and recall well known results about dimension $2$ topological field theories and their unoriented analogues.

\subsection{Oriented topological field theories}
We begin by recalling the classical definitions.

\begin{definition}
We define the category $\cob$ as follows:
\begin{itemize}
\item Objects of $\cob$ are compact oriented $1$--manifolds (disjoint unions of circles and intervals).
\item Morphisms between a pair of objects $\Sigma_0$ and $\Sigma_1$, are oriented cobordisms from $\Sigma_0$ to $\Sigma_1$ up to diffeomorphism. That is a compact, oriented $2$--manifold $M$ together with orientation preserving diffeomorphisms $\Sigma_0\simeq\partial M_{\mathrm{in}}\subset\partial M$ and $\Sigma_1\simeq\overline{\partial M}_{\mathrm{out}}\subset\partial M$ (where $\overline{\partial M}_{\mathrm{out}}$ means $\partial M_{\mathrm{out}}$ with the opposite orientation) with $\partial M_{\mathrm{in}}\cap\partial M_{\mathrm{out}}=\emptyset$. We call $\partial M_{\mathrm{in}}$, $\partial M_{\mathrm{out}}$ and $\partial M_{\mathrm{free}}=\partial M\setminus(\partial M_{\mathrm{in}}\cup\partial M_{\mathrm{out}})$ the \emph{in boundary}, the \emph{out boundary} and the \emph{free boundary} respectively. We say two cobordisms $M$ and $M'$ are diffeomorphic if there is an orientation preserving diffeomorphism $\psi\co M\stackrel{\sim}{\rightarrow}M'$ where the following commutes:
\[\xymatrix{
 & M\ar[dd]_\psi^\simeq\\
\Sigma_0\ar[ur]\ar[dr] & &\Sigma_1\ar[ul]\ar[dl]\\
 & M'
}\]
\item Composition is given by gluing cobordisms together. As mentioned above, care must be taken to ensure that gluing is well defined up to diffeomorphism. In dimension $2$ we know that smooth structure depends only on the topological structure of our manifold so we will avoid discussing the technical issues. Gluing is associative and the identity morphism from $\Sigma$ to itself is given by the cylinder $\Sigma\times I$.
\end{itemize}
\end{definition}

It can be shown that the category $\cob$ is a symmetric monoidal category with the tensor product operation given by disjoint union of manifolds.

\begin{definition}
An \emph{open--closed topological field theory} is a symmetric monoidal functor $\cob\rightarrow\vect$, where $\vect$ is the category of vector spaces over the field $k$.
\end{definition}

We can now consider open and closed theories separately by restricting to the appropriate subcategory.

\begin{definition}
\Needspace*{4\baselineskip}\mbox{}
\begin{itemize}
\item The category $\cobcl$ is the (symmetric monoidal) subcategory of $\cob$ with objects closed oriented $1$--manifolds (disjoint unions of circles) and morphisms with empty free boundary. A \emph{closed topological field theory} is a symmetric monoidal functor to $\vect$.
\item The category $\cobo$ is the full (symmetric monoidal) subcategory of $\cob$ with those objects which are not in $\cobcl$ (disjoint unions of intervals). An \emph{open topological field theory} is a symmetric monoidal functor to $\vect$.
\end{itemize}
\end{definition}

We then have the following classical results:

\begin{proposition}\label{prop:ctft}
Closed topological field theories of dimension $2$ are equivalent to commutative Frobenius algebras (see for example the book by Kock \cite{kock}).
\end{proposition}

\begin{proposition}\label{prop:otft}
Open topological field theories of dimension $2$ are equivalent to symmetric Frobenius algebras (in other words not necessarily commutative but the bilinear form is symmetric, see Moore \cite{moore2}, Segal \cite{segal} or Chuang and Lazarev \cite{chuanglazarev}).
\end{proposition}

\begin{proposition}\label{prop:octft}
Open--closed topological field theories of dimension $2$ are equivalent to `knowledgeable Frobenius algebras' (see Lauda and Pfeiffer \cite{laudapfeiffer} for definitions and proof or also Lazaroiu \cite{lazaroiu} and Moore \cite{moore1}).
\end{proposition}

\subsection{Klein topological field theories}
To extend to the unorientable case we suppress all mentions of orientations. This leads to the following definition:

\begin{definition}
We define the category $\kcob$ as follows:
\begin{itemize}
\item Objects of $\kcob$ are compact $1$--manifolds (disjoint unions of circles and intervals).
\item Morphisms between a pair of objects $\Sigma_0$ and $\Sigma_1$, are (not necessarily orientable) cobordisms from $\Sigma_0$ to $\Sigma_1$ up to diffeomorphism. That is a compact $2$--manifold $M$ together with diffeomorphisms $\Sigma_0\simeq\partial M_{\mathrm{in}}\subset\partial M$ and $\Sigma_1\simeq\partial M_{\mathrm{out}}\subset\partial M$ with $\partial M_{\mathrm{in}}\cap\partial M_{\mathrm{out}}=\emptyset$. We say two cobordisms $M$ and $M'$ are diffeomorphic if there is a diffeomorphism $\psi\co M\stackrel{\sim}{\rightarrow}M'$ where the following commutes:
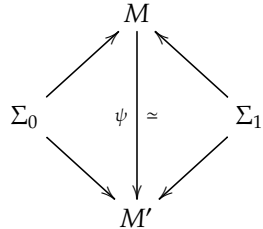
\[\xymatrix{
 & M\ar[dd]_\psi^\simeq\\
\Sigma_0\ar[ur]\ar[dr] & &\Sigma_1\ar[ul]\ar[dl]\\
 & M'
}\]
\item Composition is given by gluing cobordisms together. The identity morphism from $\Sigma$ to itself is given by the cylinder $\Sigma\times I$.
\end{itemize}
\end{definition}

As in the orientable case $\kcob$ is a symmetric monoidal category by disjoint union of manifolds.

It is convenient to identify $\cob$ and $\kcob$ with their skeletons. Recall that since all oriented circles are isomorphic (since $S^1$ is diffeomorphic to itself with the opposite orientation) the skeleton of $\cob$ is the full subcategory with objects disjoint unions of copies of a single oriented $S^1$ (so the set of objects can be identified with the natural numbers). Similarly the skeleton of $\kcob$ is the full subcategory with objects disjoint unions of copies of a single unoriented $S^1$ (so again the set of objects can be identified with the natural numbers). In this way we can think of $\cob$ as a subcategory of $\kcob$ by forgetting orientations. Note that even if the underlying manifold $M$ of a cobordism in $\kcob$ is orientable the cobordism itself is not necessarily in $\cob$, since it may not be possible to choose an orientation of $M$ such that the embeddings $\Sigma_0\hookrightarrow\partial M\hookleftarrow\Sigma_1$ are orientation preserving. Consider for example:
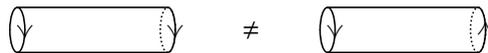
\[
\begin{xy}
*{\cylinder{\cstart}{\cstop}},
(-9,-1)*\dir2{<},(11,-1)*\dir2{<}
\end{xy}
\qquad\neq\qquad
\begin{xy}
*{\cylinder{\cstart}{\cstop}},
(-9,-1)*\dir2{<},(11,1)*\dir2{>}
\end{xy}
\]
The cobordisms above are both morphisms from $S^1$ to itself (where the arrows denote the directions of the embeddings of $S^1$). However while the cobordism on the left is the identity morphism, the cobordism on the right is in $\kcob$ but not in $\cob$.

\begin{definition}
An \emph{open--closed Klein topological field theory} is a symmetric monoidal functor $\kcob\rightarrow\vect$.
\end{definition}

\begin{definition}
\Needspace*{4\baselineskip}\mbox{}
\begin{itemize}
\item The category $\kcobcl$ is the (symmetric monoidal) subcategory of $\kcob$ with objects closed $1$--manifolds without boundary (disjoint unions of unoriented circles) and morphisms with empty free boundary. A \emph{closed Klein topological field theory} is a symmetric monoidal functor to $\vect$.
\item The category $\kcobo$ is the full (symmetric monoidal) subcategory of $\kcob$ with those objects which are not in $\kcobcl$ (disjoint unions of intervals). An \emph{open Klein topological field theory} is a symmetric monoidal functor to $\vect$.
\end{itemize}
\end{definition}

We then have analogues of \autoref{prop:ctft}, \autoref{prop:otft} and \autoref{prop:octft}.

\begin{proposition}\label{prop:ocktft}
Open--closed Klein topological field theories of dimension $2$ are equivalent to `structure algebras' (see Alexeevski and Natanzon \cite{alexeevskinatanzon} for a definition and a proof).
\end{proposition}

In particular we can immediately deduce from the above result proved in \cite{alexeevskinatanzon}, by setting the open part of a structure algebra to $0$, the result for closed KTFTs. It is also proved separately by Turaev and Turner \cite{turaevturner}.

\begin{proposition}\label{prop:cktft}
Closed Klein topological field theories of dimension $2$ are equivalent to the following structures.
\begin{itemize}
\item A commutative Frobenius algebra $A$ with an involutive anti-automorphism\footnote{Since $A$ is commutative an anti-automorphism is of course just an automorphism. Here however it is best thought of as an anti-automorphism on an algebra that just happens to be commutative for comparison with open KTFTs.} $x\mapsto x^*$ preserving the pairing. That is, $(x^*)^*=x$, $(xy)^*=y^*x^*$ and $\langle x^*,y^*\rangle=\langle x,y \rangle$.
\item There is an element $U\in A$ such that $(aU)^*=aU$ for any $a\in A$ and $U^2=\sum\alpha_i\beta_i^*$, where the copairing $\Delta\co k\rightarrow A\otimes A$ is given by $\Delta(1)=\sum\alpha_i\otimes\beta_i$.
\end{itemize}
\end{proposition}

We will not reproduce a proof of \autoref{prop:cktft}, however we will now briefly recall with pictures where each part of the structure comes from. In pictures of cobordisms we denote a crosscap attached to a surface by a dotted circle with a cross. So for example the following is an unorientable cobordism with an underlying surface made with $1$ handle, $1$ crosscap and $5$ holes:
\[
\begin{xy}
*{\cobord{(0,6),(0,18)}{(20,0),(20,12),(20,24)}{\cstart}{\cstop}{(10,8)}{@i}{(10,16)}},
(-9,5)*\dir2{<},(11,1)*\dir2{>},
(-9,-7)*\dir2{<},(11,11)*\dir2{<},
(11,-13)*\dir2{<}
\end{xy}
\]

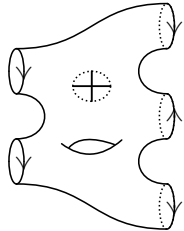
\autoref{fig:2cob} shows the generators of the orientable part of $\kcobcl$. 
\begin{figure}[ht!]
\centering
\begin{gather*}
\begin{xy}
*{\cylinder{\cstart}{\cstop}},
(-9,-1)*\dir2{<},(11,-1)*\dir2{<}
\end{xy}\qquad
\begin{xy}
*{\pants{\cstart}{\cstop}},
(-9,5)*\dir2{<},(11,-1)*\dir2{<},(-9,-7)*\dir2{<}
\end{xy}\qquad
\begin{xy}
*{\copants{\cstart}{\cstop}},
(-9,-1)*\dir2{<},(11,5)*\dir2{<},(11,-7)*\dir2{<}
\end{xy}
\\
\begin{xy}
*{\birth{\cstop}},
(5.5,-1)*\dir2{<}
\end{xy}\qquad
\begin{xy}
*{\death{\cstart}},
(-3.5,-1)*\dir2{<}
\end{xy}\qquad
\begin{xy}
*{\twist{\cstart}{\cstop}},
(-9,-7)*\dir2{<},(-9,5)*\dir2{<},(11,5)*\dir2{<},(11,-7)*\dir2{<}
\end{xy}
\end{gather*}
\caption{Generators of $\cobcl$ (considered as a subcategory of $\kcobcl$)}
\label{fig:2cob}
\end{figure}

\begin{figure}[ht!]
\centering
\[
\begin{xy}
*{\cylinder{\cstart}{\cstop}},
(-9,-1)*\dir2{<},(11,1)*\dir2{>}
\end{xy}\qquad
\begin{xy}
*{\rp{\cstop}},
(10.25,-1)*\dir2{<}
\end{xy}
\]
\caption{Additional generators of $\kcobcl$ not in $\cobcl$}
\label{fig:2kcob}
\end{figure}

By moving crosscaps and flipping orientations of boundaries we can decompose any cobordism into an orientable cobordism composed with copies of the two cobordisms in \autoref{fig:2kcob}. For example we can decompose our previous example as:
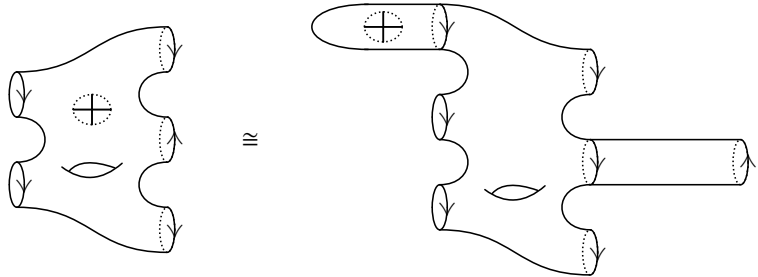
\[
\begin{xy}
*{\cobord{(0,6),(0,18)}{(20,0),(20,12),(20,24)}{\cstart}{\cstop}{(10,8)}{@i}{(10,16)}},
(-9,5)*\dir2{<},(11,1)*\dir2{>},
(-9,-7)*\dir2{<},(11,11)*\dir2{<},
(11,-13)*\dir2{<}
\end{xy}\qquad \cong \quad
\begin{xy}
*{\cobord{(0,6),(0,18),(0,30)}{(20,0),(20,12),(20,24)}{}{\cstop}{(10,8)}{@i}{@i}},
(-9.5,2)*\dir2{<},(10.5,-4)*\dir2{<},
(-9.5,-10)*\dir2{<},(10.5,8)*\dir2{<},
(10.5,-16)*\dir2{<},(-9.5,14)*\dir2{<},
(-19.75,15)*{\rp{\cstop}},
(20,-3)*{\cylinder{}{\cstop}},(30.5,-2)*\dir2{>},
(-10.5,-9)*{\cstart},
(-10.5,3)*{\cstart}
\end{xy}
\]

This shows us that the cobordisms in \autoref{fig:2cob} and \autoref{fig:2kcob} together generate $\kcobcl$. In particular we see that a closed KTFT is given by a commutative Frobenius algebra $A$ together with a linear map corresponding to the cobordism on the left in \autoref{fig:2kcob} which is clearly an involution and an element $U\in A$ given by the image of $1\in k$ under the map corresponding to the cobordism on the right. That the involution is an anti-automorphism corresponds to the relation
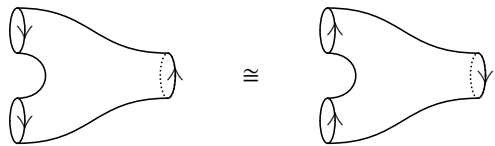
\[
\begin{xy}
*{\pants{\cstart}{\cstop}},
(-9,5)*\dir2{<},(11,1)*\dir2{>},(-9,-7)*\dir2{<}
\end{xy}\qquad\cong\qquad
\begin{xy}
*{\pants{\cstart}{\cstop}},
(-9,7)*\dir2{>},(11,-1)*\dir2{<},(-9,-5)*\dir2{>}
\end{xy}
\]
which can be seen by reflecting the cobordism in a suitable horizontal plane.

The relation $U^2=\sum\alpha_i\beta_i^*$ arises from the fact that $2$ crosscaps are diffeomorphic to a Klein bottle with a hole which can be decomposed into orientable surfaces:
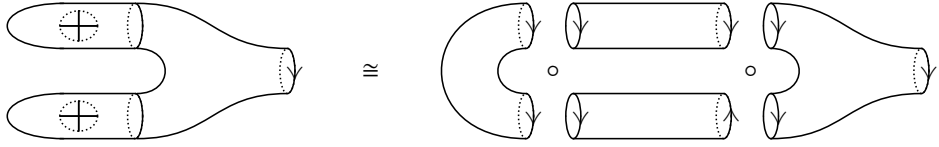
\[
\begin{xy}
*{\pants{}{\cstop}},
(10.5,-1)*\dir2{<},
(-20,6)*{\rp{\cstop}},
(-20,-6)*{\rp{\cstop}}
\end{xy}\qquad\cong\quad
\begin{xy}
*{\copair{\cstop}},
(8,5)*\dir2{<},(8,-7)*\dir2{<}
\end{xy}\circ
\begin{xy}
(0,6)*{\cylinder{\cstart}{\cstop}},
(0,-6)*{\cylinder{\cstart}{\cstop}},
(-9,5)*\dir2{<},(11,5)*\dir2{<},(-9,-7)*\dir2{<},(11,-5)*\dir2{>}
\end{xy}\circ
\begin{xy}
*{\pants{\cstart}{\cstop}},
(-9,5)*\dir2{<},(11,-1)*\dir2{<},(-9,-7)*\dir2{<}
\end{xy}
\]

Finally the relation $(aU)^*=aU$ can be seen by considering a M\"obius strip (which is equivalent to a crosscap) with a hole:

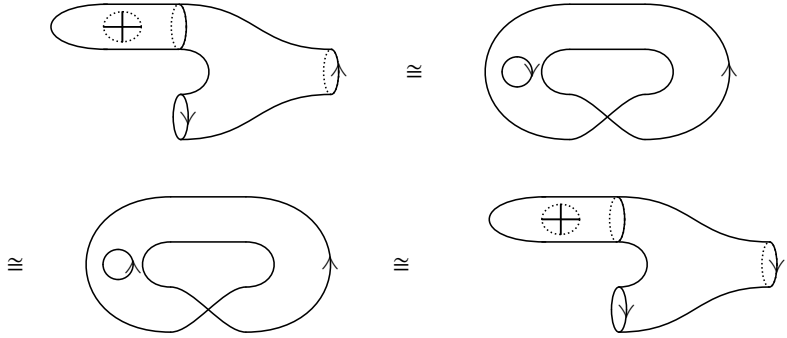
\begin{gather*}
\begin{xy}
*{\pants{}{\cstop}},
(10.5,1)*\dir2{>},
(-20,6)*{\rp{\cstop}},
(-9.5,-7)*\dir2{<},
(-10.5,-6)*{\cstart}
\end{xy}\qquad\cong\quad
\begin{xy}
(-12.5,0)*{\copair{}},
(0,6)*{\shortcylinder{}{}},
(0,-6)*{\flip{}{}},
(12.5,0)*{\pair{}},
(-10,-0.5)*\dir2{<},
(16.2,1)*\dir2{>},
(-12,0)*{\puncture}
\end{xy}\\\\
\cong\quad
\begin{xy}
(-12.5,0)*{\copair{}},
(0,6)*{\shortcylinder{}{}},
(0,-6)*{\flip{}{}},
(12.5,0)*{\pair{}},
(-10,0.5)*\dir2{>},
(16.2,1)*\dir2{>},
(-12,0)*{\puncture}
\end{xy}
\quad\cong\qquad
\begin{xy}
*{\pants{}{\cstop}},
(10.5,-1)*\dir2{<},
(-20,6)*{\rp{\cstop}},
(-9.5,-7)*\dir2{<},
(-10.5,-6)*{\cstart}
\end{xy}
\end{gather*}
Here the second diffeomorphism can be seen by pushing the left hole once around the M\"obius strip (so its orientation changes when it passes through the twist).

It is not too difficult to convince oneself that these relations generate all relations and hence give a sufficient set of relations.

Open KTFTs are however our main object of study. We will prove the following result for open KTFTs later as a corollary of our approach using operads.

\begin{proposition}\label{prop:oktft}
Open Klein topological field theories of dimension $2$ are equivalent to symmetric Frobenius algebras together with an involutive anti-automorphism $x\mapsto x^*$ preserving the pairing.
\end{proposition}

\begin{examples}
\Needspace*{4\baselineskip}\mbox{}
\begin{itemize}
\item Any matrix algebra over a field is a symmetric Frobenius algebra with $\langle A,B \rangle = \operatorname{tr} AB$. With an involution given by the transpose we obtain an open KTFT.
\item Let $G$ be a finite group. Then the group algebra $\mathbb{C}[G]$ is a symmetric Frobenius algebra with bilinear form $\langle a,b \rangle$ given by the coefficient of the identity element in $ab$. Define an involution as the linear extension of $g\mapsto g^{-1}$. This is an open KTFT.
\item If $G$ is abelian, then $\mathbb{C}[G]$ forms a closed KTFT with \mbox{$U=\frac{1}{\sqrt{|G|}}\sum g^2$}.
\end{itemize}
\end{examples}

\section{Preliminaries on operads}
We wish to reinterpret KTFTs in the language of modular operads. To make clear our notation we recall here some of the relevant notions from the theory of operads.

\subsection{Graphs, trees, operads and modular operads}
In this section we will outline the notation we will use and recall for convenience the definitions of (modular) operads with some minor modifications. For full details see Ginzburg and Kapranov \cite{ginzburgkapranov} and Getzler and Kapranov \cite{getzlerkapranov}.

We need the notions of graphs and trees. A graph is what we expect but we allow graphs with external half edges (legs). Precisely a graph can be defined as follows:

\begin{definition}
A graph $G$ consists of the following data:
\begin{itemize}
\item Finite sets $\vertices(G)$ and $\halfedges(G)$ with a map $\lambda\co\halfedges(G)\rightarrow\vertices(G)$
\item An involution $\sigma\co\halfedges(G)\rightarrow\halfedges(G)$
\end{itemize}
The set $\vertices(G)$ is the set of \emph{vertices} of $G$ and $\halfedges(G)$ is the set of \emph{half edges} of $G$. A half edge $a$ is connected to a vertex $v$ if $\lambda(a)=v$. We denote the set of half edges connected to $v$ by $\flags(v)$ and we write $n(v)$ for the cardinality of $\flags(v)$ (the \emph{valence} of $v$). Two half edges $a\neq b$ form an \emph{edge} if $\sigma(a)=b$. The set $\edges(G)$ is the set of unordered pairs of half edges forming an edge. We call half edges that are fixed by $\sigma$ the \emph{legs} of $G$ and denote the set of legs as $\legs(G)$.
\end{definition}

\begin{definition}
An isomorphism of graphs $f\co G\rightarrow G'$ consists of bijections $f_1\co\vertices(G)\rightarrow\vertices(G')$ and  $f_2\co\halfedges(G)\rightarrow\halfedges(G')$ satisfying $\lambda\circ f_2 = f_1\circ\lambda$ and $\sigma\circ f_2 = f_2\circ\sigma$.
\end{definition}

Given a graph $G$ we can associate a finite $1$--dimensional cell complex $|G|$ in the obvious way with $0$--cells corresponding to vertices and the ends of legs and $1$--cells corresponding to edges and legs. We say $G$ is \emph{connected} if $|G|$ is connected.

\begin{definition}
By a \emph{tree} we mean a connected graph $T$ with at least $2$ legs such that $\dim{H_1(|T|)}=0$ (equivalently $|T|$ is contractible).
\end{definition}

\begin{definition}
A \emph{labelled graph} is a connected non-empty graph $G$ together with a labelling of the $n$ legs of $G$ by the set $\{1,\ldots,n\}$ and a map $g\co \vertices(G)\rightarrow \mathbb{Z}_{\geq 0}$. We call the value $g(v)$ the \emph{genus} of $v$. The genus of a labelled graph $G$ is defined by the formula:
\[g(G) = \dim{H_1(|G|)}+\sum_{v\in\vertices(G)}g(v)\]
Clearly this is the number of loops in the graph obtained by gluing $g(v)$ loops to each vertex $v$ of the underlying graph and contracting all internal edges that are not loops. A vertex of a labelled graph is called \emph{stable} if $2g(v)+n(v)>2$. A labelled graph is called stable if all its vertices are stable. An extended stable graph is defined in the same way except a vertex is called extended stable if $2g(v)+n(v) \geq 2$. An isomorphism of labelled graphs is an isomorphism of graphs preserving the label of each leg and the genus of each vertex.
\end{definition}

\begin{definition}
 By a \emph{labelled tree} we mean a tree $T$ with $n+1\geq 2$ legs with a labelling of the legs by the set $\{1,\ldots,n+1\}$. Given such a labelled tree we call the leg labelled by $n+1$ the output or root of $T$ and the other legs the inputs of $T$, denoted $\inputs(T)$. This induces  a direction on the tree where each half edge is directed towards the output and given a vertex $v$ we write $\inputs(v)\subset\flags(v)$ for the set of $n(v)-1$ incoming half edges at $v$. Note that $n(v)\geq 2$ for all vertices $v$. We call $v$ \emph{reduced} if $n(v)>2$. We call a labelled tree reduced if all its vertices are reduced. An isomorphism of labelled trees is an isomorphism of trees that preserves the labelling. 

We denote by $\edges^+(T)=\edges(T)\cup\inputs(T)$ the set of internal edges together with the inputs of $T$.
\end{definition}

\begin{remark}
Note that a labelled tree is equivalent to an extended stable graph of genus $0$ (by assigning a genus of $0$ to each vertex). Reduced trees can then be thought of as stable graphs of genus $0$. We use the term `reduced' as opposed to `stable' here to emphasise the fact that we do not consider the vertices as having a genus.
\end{remark}

Given a labelled graph $G$ we denote by $G/e$ the labelled graph obtained by contracting the internal edge $e$. The genus of each of the vertices of $G/e$ is defined in the natural way, so that the overall genus of the graph remains constant. More precisely, if we contract an edge $e$ connected to two different vertices $v_1$ and $v_2$ into a single vertex $v$ then we set $g(v)=g(v_1)+g(v_2)$. If we contract an edge $e$ connected to a single vertex $v$ (so $e$ is a loop) then the genus of $v$ increases by one.

Observe that if we contract multiple edges it does not matter (up to isomorphism) in which order we contract them. We write $\Gamma((g,n))$ for the category of extended stable graphs of genus $g$ with $n$ legs with morphisms generated by isomorphisms of labelled graphs and edge contractions. For a labelled tree $T$ we define $T/e$ similarly. We denote by $T((n))$ the category of trees with $n$ legs. By an $n$--tree we mean a tree with $n$ inputs (equivalently $n+1$ legs). We denote by $T(n)$ the category of $n$--trees. Note that $T(n)$ is isomorphic to $T((n+1))$ and $\Gamma((0,n+1))$.

We can glue graphs with legs. If $G'$ has $n>0$ legs and $G$ has $m>0$ legs then we write $G\circ_i G'$ for the graph obtained by gluing the leg of $G'$ labelled by $n$ to the leg of $G$ labelled by $i$. For trees this corresponds to gluing the output of one tree to the $i$--th input of the other.

\begin{definition}
Let $k$ be a field.
\begin{itemize}
\item The symmetric monoidal category $\vect$ is the category of vector spaces over $k$ with the tensor product.
\item The symmetric monoidal category $\dgvect$ is the category of differential graded vector spaces over $k$ with morphisms given by chain maps and with symmetry $s\co V\otimes W\rightarrow W\otimes V$ given by $s(v\otimes w)=(-1)^{\degree{v}\degree{w}}w\otimes v$. Here $\degree{v}$ and $\degree{w}$ are the degrees of the homogeneous elements $v$ and $w$.
\item The symmetric monoidal category $\topol$ is the category of topological spaces with the usual product.
\end{itemize}
\end{definition}

Fix $\mathcal{C}$ to be one of the symmetric monoidal categories $\vect$, $\dgvect$ or $\topol$.

\begin{definition}
\Needspace*{4\baselineskip}\mbox{}
\begin{itemize}
\item An \emph{$S$--module} is a collection $V=\{V(n) : n\geq 1\}$ with $V(n)\in\ob\mathcal{C}$ equipped with a left action of $S_n$ (the symmetric group on $n$ elements) on $V(n)$.
\item A \emph{cyclic $S$--module} is a collection $U=\{U((n)) : n\geq 2\}$ with $U((n))\in\ob\mathcal{C}$ equipped with a left action of $S_n$ on $U((n))$.
\item An \emph{extended stable $S$--module}\footnote{This differs slightly from the definition in \cite{chuanglazarev} since we also allow the pair $(g,n)=(1,0)$. This makes very little difference in practice however.} is a collection $W=\{W((g,n)) : n,g \geq 0\}$ with $W((g,n))\in\ob\mathcal{C}$ equipped with a left action of $S_n$ on $W((g,n))$ and where $W((g,n))=0$ whenever $2g+n \leq 1$. We call an extended stable $S$--module a \emph{stable $S$--module} if $W((g,n))=0$ whenever $2g+n \leq 2$.
\end{itemize}
A morphism of (cyclic/extended stable) $S$--modules is given by a collection of $S_n$--equivariant morphisms.
\end{definition}

\begin{remark}\label{rem:cyclic}
Note that a cyclic $S$--module can also be defined as an $S$--module $V$ with an action of $S_{n+1}$ extending the action of $S_n$ on $V(n)$. This can be seen by setting $V((n))=V(n-1)$. Similarly given a cyclic $S$--module $U$, by restricting to the action of $S_{n}\subset S_{n+1}$ on $U(n)=U((n+1))$ we see that a cyclic $S$--module has an underlying $S$--module.
\end{remark}

Given an $S$--module $V$ and a finite set $I$ with $n$ elements we define
\[
V(I)=\left(\bigoplus_{f\in\iso([n],I)}V(n)\right)_{S_n}
\]
the coinvariants with respect to the simultaneous action of $S_n$ on $\iso([n],I)$ and $V(n)$ (where $[n]=\{1,\ldots,n\}$). Similarly given a cyclic $S$--module $U$ we define
\[
U((I))=\left(\bigoplus_{f\in\iso([n],I)}U((n))\right)_{S_n}
\]
and given an extended stable $S$--module $W$ we define:
\[
W((g,I))=\left(\bigoplus_{f\in\iso([n],I)}W((g,n))\right)_{S_n}
\]

\begin{remark}
For simplicity we have used direct sums above since we shall normally be working in the category of (differential graded) vector spaces. More generally one should use coproducts so, for example, in the case that $V$ is an $S$--module in $\topol$ direct sums in the above definitions are replaced by disjoint unions.
\end{remark}

If $T$ is a labelled tree and $V$ is an $S$--module then we define the space of $V$--decorations on $T$ as:
\[
V(T) = \bigotimes_{v\in\vertices(T)}V(\inputs(v))
\]
Similarly for $U$ a cyclic $S$--module the space of $U$--decorations on $T$ is
\[
U((T)) = \bigotimes_{v\in\vertices(T)}U((\flags(v)))
\]
and for $W$ an extended stable module and $G$ an extended stable graph we define the space of $W$--decorations on $G$ as:
\[
W((G)) = \bigotimes_{v\in\vertices(G)}W((g(v),\flags(v)))
\]

Given an isomorphism of labelled graphs $G\rightarrow G'$ or labelled trees $T\rightarrow T'$ there are induced isomorphisms on the corresponding spaces of decorations. 

Note that if $W$ is a stable $S$--module, then $W((G))=0$ unless $G$ is also stable.

\begin{definition}
We define an endofunctor $\mathbb{O}$ on the category of $S$--modules by the formula:
\[
\mathbb{O}V(n)=\operatorname*{colim}_{T\in\iso T(n)}V(T)
\]

We define an endofunctor $\mathbb{C}$ on the category of cyclic $S$--modules by the formula:
\[
\mathbb{C}U((n))=\operatorname*{colim}_{T\in\iso T((n))}U((T))
\]

We define an endofunctor $\mathbb{M}$ on the category of extended stable $S$--modules by the formula:
\[
\mathbb{M}W((g,n))=\operatorname*{colim}_{G\in\iso\Gamma((g,n))}W((G))
\]

Each of these endofunctors can be given the structure of a monad (triple) in the natural way as shown by Getzler and Kapranov \cite{getzlerkapranov}. We call an algebra over these monads an operad, a cyclic operad and an extended modular operad respectively.  A modular operad is an extended modular operad whose underlying $S$--module is stable.
\end{definition}

We use the term `extended modular operad' to bring our definitions closer to \cite{getzlerkapranov,chuanglazarev}. However we are not concerned with the distinction between a modular operad and an extended modular operad. Therefore we will from now on use the term `modular operad' to mean extended modular operad unless explicitly stated otherwise.

\begin{remark}\label{rem:classicaloperaddefs}
We can unpack these somewhat technical definitions to gain more concrete descriptions closer to the classical definition of operads.
\begin{itemize}
\item An operad is an $S$--module $\mathcal{P}$ together with composition maps $\circ_i\co\mathcal{P}(n)\otimes\mathcal{P}(m)\rightarrow\mathcal{P}(n+m-1)$ for $n,m\geq 1$, $1\leq i\leq n$. These maps must satisfy equivariance and associativity conditions.
\item A cyclic operad is a cyclic $S$--module $\mathcal{Q}$ together with composition maps $\circ_i\co\mathcal{Q}((n))\otimes\mathcal{Q}((m))\rightarrow\mathcal{Q}((n+m-2))$ for $n,m\geq 2$, $1\leq i\leq n$. These maps must satisfy equivariance and associativity conditions.
\item A modular operad is an extended stable $S$--module $\mathcal{O}$ together with composition maps $\circ_i\co\mathcal{O}((g,n))\otimes\mathcal{O}((g',m))\rightarrow\mathcal{O}((g+g',n+m-2))$ for  $n,m\geq 1$, $1\leq i \leq n$ and contraction maps $\xi_{ij}\co\mathcal{O}((g,n))\rightarrow\mathcal{O}((g+1,n-2))$ for $n\geq 2$, $1\leq i\neq j \leq n$. These maps must satisfy equivariance and associativity conditions.
\end{itemize}
\end{remark}

We can understand the associativity and equivariance conditions mentioned in \autoref{rem:classicaloperaddefs} in a simple way using trees and graphs as in \cite{getzlerkapranov,ginzburgkapranov}. Given a tree $T$ with a vertex $v$ with $n(v)=n$ and an $S$--module $V$ we observe that choosing a particular direct summand representing $V(\inputs(v))$ is equivalent to choosing a labelling of $\inputs(v)$ by the set $[n-1]$. Similarly given a cyclic $S$--module $U$ choosing a particular direct summand representing $U((\flags(v)))$ is equivalent to choosing a labelling of $\flags(v)$ by $[n]$. Given an extended stable graph $G$ with a vertex $v$ and an extended stable module $W$ choosing a particular direct summand representing $W((g(v),\flags(v)))$ is equivalent to choosing a labelling of $\flags(v)$ by $[n]$.

By choosing appropriate labellings of $\inputs(v)$ or $\flags(v)$ at two vertices connected by the edge $e$ in a tree $T$ we can use the composition maps of an operad $\mathcal{P}$ or a cyclic operad $\mathcal{Q}$ to define a map $\mathcal{P}(T)\rightarrow\mathcal{P}(T/e)$ or $\mathcal{Q}((T))\rightarrow\mathcal{Q}((T/e))$ in the obvious way by considering $\circ_i$ as gluing the output (labelled by $n(v)$ in the cyclic case) at one vertex to the $i$--th input/leg (the leg labelled by $i$) at the other vertex. The equivariance condition simply says that this is well defined regardless of the particular labellings (choice of direct summands) we choose. The associativity condition corresponds to these maps assembling to a well defined functor on $T(n)$ or $T((n))$. Precisely this simply means that no matter in which order we contract the edges of a tree $T$, the induced map on $\mathcal{P}(T)$ or $\mathcal{Q}((T))$ is the same.

In the case of modular operads the same applies but since we are using graphs the edge $e$ could be a loop at a vertex $v$ and then we must use the contraction maps, considering $\xi_{ij}$ as gluing together the half edges making up $e$ labelled by $i$ and $j$ to define a map $\mathcal{O}((G))\rightarrow\mathcal{O}((G/e))$.

\begin{definition}
A \emph{unital} operad is an operad $\mathcal{P}$ with an element $1\in\mathcal{P}(1)$ such that $1\circ_1 a=a=a\circ_i 1$ for any $a\in\mathcal{P}(n)$ with $n\geq 1$ and $1\leq i\leq n$.
\end{definition}

For completeness we note the following lemma/alternative definition which follows from considering \autoref{rem:cyclic}. This allows us to ask whether an operad can be given the additional structure of a cyclic operad.

\begin{lemma}\label{lem:cyclic}
A cyclic operad is a cyclic $S$--module $\mathcal{Q}$ whose underlying $S$--module has the structure of an operad such that $(a\circ_m b)^* = b^*\circ_1 a^*$ for any $a\in\mathcal{Q}(m), b\in\mathcal{Q}(n)$ where $c^*$ is the result of applying the cycle $(1\quad 2\,\ldots\, n+1)\in S_{n+1}$  to $c\in \mathcal{Q}(n)=\mathcal{Q}((n+1))$
\end{lemma}

There is clearly a functor from cyclic operads to operads. Given a modular operad $\mathcal{O}$, the genus $0$ part consisting of the spaces $\mathcal{O}((0,n))$ forms a cyclic operad. This gives a functor from modular operads to cyclic operads. If $\mathcal{Q}$ is a cyclic operad then the \emph{modular closure}\footnote{This is also sometimes called the modular envelope and denoted $\mathbf{Mod}(\mathcal{Q})$ as in \cite{costello1}.} $\modc{\mathcal{Q}}$ is the left adjoint functor to this functor and the \emph{na\"ive closure} $\underline{\mathcal{Q}}$ is the right adjoint.

The modular closure is obtained from $\mathcal{Q}$ by freely adjoining the contraction maps and imposing only those relations necessary for associativity and equivariance to still hold. The na\"ive closure is obtained by setting all contraction maps to zero.

\begin{definition}
Given $A,B\in\vect$, by $\Hom(A,B)\in\vect$ we mean the space of linear maps from $A$ to $B$. Given $A',B'\in\dgvect$, by $\Hom(A',B')^n$ we mean the vector space of homogeneous linear maps of degree $n\in\mathbb{Z}$ (maps of vector spaces $f\co A'\rightarrow B'$ such that $f(A'_i)\subset B'_{i+n}$). By $\Hom(A',B')\in\dgvect$ we mean the space of graded linear maps $\bigoplus_n\Hom(A',B')^n$ equipped with differential given by $dm = d_1\circ m - (-1)^{\degree{m}}m\circ d_2$ where $d_1$ and $d_2$ are the differentials on $B'$ and $A'$ respectively and $m$ is a homogeneous map of degree $\degree{m}$.
\end{definition}

\begin{definition}
Let $\mathcal{C}$ be one of $\vect$ or $\dgvect$ and let $V\in\ob\mathcal{C}$. The \emph{endomorphism operad} of $V$, denoted $\End[V]$, is defined as having underlying $S$--module given by setting $\End[V](n)=\Hom(V^{\otimes n},V)$ with the natural action of $S_n$. Composition maps are given by composing morphisms in the obvious way.

Now assume we have a symmetric non-degenerate bilinear form $\langle -, - \rangle \co V\otimes V\rightarrow k$. We define the \emph{endomorphism cyclic operad} of $V$ as having underlying cyclic $S$--module $\mathcal{E}[V]((n))=V^{\otimes n}$ with the natural action of $S_n$. If $a\in V^{\otimes n}$ and $b\in V^{\otimes m}$ then $a\circ_i b\in V^{\otimes(n+m-2)}$ is defined by contracting $a\otimes b$ with the bilinear form, applied to the $i$--th factor of $a$ and the $m$--th factor of $b$. Using the isomorphism $V^{\otimes(n+1)}\cong\Hom(V^{\otimes n},V)$ we see the underlying operad of the endomorphism cyclic operad is just the endomorphism operad. We define the \emph{endomorphism modular operad} as having underlying $S$--module $\mathcal{E}[V]((g,n))=V^{\otimes n}$ with composition maps defined as for the endomorphism cyclic operad and for $a\in\mathcal{E}[V]((g,n))$ we define $\xi_{ij}(a)\in\mathcal{E}[V]((g+1,n-2))$ by contracting the $i$--th factor and the $j$--th factor of $a$ using the bilinear form.
\end{definition}

\begin{definition}
Given an operad $\mathcal{P}$ in $\vect$ or $\dgvect$  an algebra over $\mathcal{P}$ is a vector space/differential graded vector space $V$ together with a morphism of operads $\mathcal{P}\rightarrow\End[V]$. Similarly an algebra over a cyclic/modular operad $\mathcal{O}$ is a vector space/differential graded vector space $V$ with a symmetric non-degenerate bilinear form $B$, together with a morphism of cyclic/modular operads $\mathcal{O}\rightarrow\mathcal{E}[V]$.
\end{definition}

Clearly algebras over various types of operads can be given by a collection of maps in $\Hom(V^{\otimes n},V)$ satisfying certain conditions.

\begin{remark}
  We will refer to operads in the category $\dgvect$ as dg operads and operads in the category $\topol$ as topological operads. Obviously by considering a vector space as concentrated in degree $0$ with the zero differential we can consider $\vect$ as a subcategory of $\dgvect$ and hence an operad in $\vect$ can be considered as a dg operad.
\end{remark}

\subsection{Quadratic operads and Koszul duality}\label{subsec:koszulduality}
We now restrict ourselves to operads and recall the theory of Koszul duality from \cite{ginzburgkapranov}. Let $k$ be a field and let $K$ be an associative unital $k$--algebra. All our operads in this section are required to be unital.

\begin{definition}
A \emph{$K$--collection} is a collection $E=\{E(n) : n\geq 2\}$ of $k$--vector spaces equipped with the following structures:
\begin{itemize}
\item A left $S_n$ action on $E(n)$ for each $n\geq 2$
\item A $(K,K^{\otimes n})$--bimodule structure on $E(n)$ that is compatible with the $S_n$ action. This means for any $\sigma\in S_n$, $\mu,\lambda_i\in K$ and $a\in E(n)$ we have $\sigma(\mu a) = \mu\sigma(a)$ and
\[
\sigma(a(\lambda_1\otimes\ldots\otimes\lambda_n)) = \sigma(a)(\lambda_{\sigma(1)}\otimes\ldots\otimes\lambda_{\sigma(n)})
\]
\end{itemize}
\end{definition}

By setting $E(1)=K$, a $K$--collection should be thought of as an $S$--module $E$ together with composition maps $\circ_i\co E(n)\otimes E(1)\rightarrow E(n)$ and $\circ_1\co E(1)\otimes E(n)\rightarrow E(n)$ satisfying associativity and equivariance conditions. A morphism of $K$--collections is then a morphism of the underlying $S$--modules that preserve these composition maps.

Given a reduced tree $T$ and a $K$--collection $E$ we define
\[
E(T)=\bigotimes_{v\in \vertices(T)}E(\inputs(v))
\]
where the tensor product is taken over $K$ using the $(K,K^{\otimes\inputs(v)})$--bimodule structure on each $E(\inputs(v))$. Given an isomorphism of trees $T\rightarrow T'$ we have an induced isomorphism $E(T)\rightarrow E(T')$.

Clearly if $\mathcal{P}$ is an operad with $\mathcal{P}(1)=K$ then $\{\mathcal{P}(n) : n\geq 2\}$ is a $K$--collection. Given a $K$--collection $E$ we can form the free operad $F(E)$ consisting of $E$--decorated reduced trees with composition given by gluing trees. More precisely, denoting the category of reduced $n$--trees by $\widehat{T}(n)$, we set
\[
F(E)(n) = \operatorname*{colim}_{T\in\iso \widehat{T}(n)}E(T)
\]
and compositions are induced by the natural maps
\[\circ_i\co E(T)\otimes E(T')\rightarrow E(T)\otimes_K E(T')\cong E(T\circ_i T')\]
where the tensor product over $K$ is using the right $K$--module structure on $E(T)$ corresponding to the $i$--th input.

Let $K$ be semisimple and let $E$ be a finite dimensional $K$--collection with $E(n)=0$ for $n>2$. We will denote the $(K,K^{\otimes 2})$--bimodule also by $E$. Let $R\subset F(E)(3)$ be an $S_3$--stable $(K,K^{\otimes 3})$--sub-bimodule. Let $(R)$ be the ideal in $F(E)$ generated by $R$. We define an operad $\mathcal{P}(K,E,R)=F(E)/(R)$. An operad of type $\mathcal{P}(K,E,R)$ is called a \emph{quadratic operad}.

\begin{definition}
Given a $(K,K^{\otimes n})$--bimodule $E$ with a compatible $S_n$ action we denote by $E^{*}=\Hom_K(E,K)$ the space of (left) $K$--linear maps. This has the natural structure of a $(K^{\mathrm{op}},(K^{\mathrm{op}})^{\otimes n})$--bimodule with the transposed action of $S_n$. We can also equip it with the transposed action of $S_n$ twisted by the sign representation in which case we denote it $E^{\vee}=\Hom_K(E,K)\otimes\mathrm{sgn}_n$.
\end{definition}

\begin{definition}
Given a quadratic operad $\mathcal{P}(K,E,R)$ we can form a $K^{\mathrm{op}}$--collection from $E^{\vee}$. Observe that $F(E^{\vee})(3)=F(E(3))^{\vee}$. Let $R^{\perp}\subset F(E^{\vee})(3)$ be the orthogonal complement of $R$, which is an $S_3$--stable $(K,K^{\otimes 3})$--sub-bimodule. We define the dual quadratic operad $\mathcal{P}^!$ to be
\[
\mathcal{P}^!=\mathcal{P}(K^{\mathrm{op}},E^{\vee},R^{\perp})
\]
\end{definition}

We next briefly recall the definitions and results on the cobar construction and the dual dg operad, full details of which can be found in Ginzburg and Kapranov \cite{ginzburgkapranov}. Recall that for a dg operad $\mathcal{P}$ the cobar construction is the operad $F(\mathcal{P}^*[-\mathrm{1}])$ with differential coming from the internal differential and the unique differential dual to the composition of $\mathcal{P}$. Here we give the construction explicitly.

Let $V$ be a finite dimensional vector space. We denote by $\Det(V)$ the top exterior power of $V$. Given a tree $T$ we set $\det(T)=\Det(k^{\edges(T)})$ and $\Det(T)=\Det(k^{\edges^+(T)})$. We denote by $|T|$ the number of internal edges of $T$.

Let $\mathcal{P}$ be a dg operad with $\mathcal{P}(n)$ finite dimensional and $K=\mathcal{P}(1)$ a semisimple unital $k$--algebra concentrated in degree $0$. We call such a dg operad \emph{admissible} and denote the category of admissible dg operads by $\dgop(K)$. For $n\geq 2$ we construct complexes $C'(\mathcal{P})(n)^s = 0 $ for $s\leq 0$ and
\[
C'(\mathcal{P})(n)^s = \bigoplus_{\substack{n\text{--trees }T\\|T|=s-1}}\mathcal{P}(T)^*\otimes\det(T)
\]
where the direct sums are over isomorphism classes of reduced trees and $\mathcal{P}(T)$ is defined by considering the underlying dg $K$--collection of $\mathcal{P}$ (and so tensor products are taken over $K$).

To define the differential $\delta$ recall if $T_2=T_1/e$ is obtained by contracting an internal edge $e$, we have a composition map $\gamma_{T_1,T_2}\co\mathcal{P}(T_1)\rightarrow\mathcal{P}(T_2)$. We define $\delta$ on the direct summands by maps
\[
\delta_{T'}\co\mathcal{P}(T')^*\otimes\det(T')\rightarrow\bigoplus_{\substack{n\text{--trees }T\\|T|=i+1}}\mathcal{P}(T)^*\otimes\det(T)
\]
for $T'$ an $n$--tree with $|T'|=i$, with
\[
\delta_{T'}=\bigoplus_{\substack{(T,e)\\T'=T/e}}(\gamma_{T,T'})^*\otimes l_e
\]
and $l_e\co\det(T')\rightarrow\det(T)$ is defined by:
\[
l_e(f_1\wedge\ldots\wedge f_i)=e\wedge f_1\wedge\ldots\wedge f_i
\]

Since $\mathcal{P}$ is a dg operad each term of the complex just defined has an internal differential $d$. This is compatible with $\delta$ and we write $C(\mathcal{P})(n)^{\bullet}$ for the total complex of the double complex. These complexes together form a dg $K^{\mathrm{op}}$--collection $C(\mathcal{P})$. 

\begin{definition}
It can be shown (by comparing to the operad $F(\mathcal{P^*[-\mathrm{1}]})$) that $C(\mathcal{P})$ has a natural structure of a dg operad. We call this operad the \emph{cobar construction} of $\mathcal{P}$.
\end{definition}

Let $T$ and $T'$ be $n$--trees and $m$--trees respectively with $|T|=p$ and $|T'|=q$. Composition can be obtained explicitly using the maps $\circ_i\co(\mathcal{P}(T)\otimes\det(T))\otimes(\mathcal{P}(T')\otimes\det(T'))\rightarrow\mathcal{P}(T\circ_i T')\otimes\det(T\circ_i T')$ given by
\[
\begin{split}
(a_1\otimes\ldots\otimes a_{p+1})\otimes(e_1\wedge\ldots\wedge e_p)\circ_i(b_1\otimes\ldots\otimes b_{q+1})(f_1\wedge\ldots\wedge f_q)
\\=
(a_1\otimes\ldots\otimes a_{p+1}\otimes b_1\otimes\ldots\otimes b_{q+1})\otimes(e_1\wedge\ldots\wedge e_p\wedge f_1\wedge\ldots\wedge f_q\wedge e)
\end{split}
\]
where $e$ is the new internal edge formed from gluing the root of $T'$ to the $i$--th input of $T$.

\begin{definition}\label{def:dualdg}
The \emph{dual dg operad} $\dual\mathcal{P}$ is defined as
\[
\dual\mathcal{P}=C(\mathcal{P})\otimes\Lambda
\]
where $\Lambda$ is the determinant operad with $\Lambda(n)=k$ concentrated in degree $1-n$ carrying the sign representation of $S_n$.
\end{definition}

Further, from the definitions, it follows that $\mathcal{P}\mapsto \dual\mathcal{P}$ extends to a contravariant functor $\dual\co\dgop(K)\rightarrow\dgop(K^{\mathrm{op}})$ which takes quasi-isomorphisms to quasi-isomorphisms.

\begin{remark}\label{rem:dualdg}
$\dual\mathcal{P}$ can also be obtained from the cobar construction by shifting the grading by $1-n$, twisting by the sign representation and introducing a sign $(-1)^{(m-1)i-1}$ to the composition $\circ_i\co\dual\mathcal{P}(n)\otimes\dual\mathcal{P}(m)\rightarrow\dual\mathcal{P}(n+m-1)$. If $\mathcal{P}$ is an admissible dg operad concentrated in degree $0$ then the highest non-zero term of $\dual\mathcal{P}$ is in degree $0$ and is given by:
\[
\dual\mathcal{P}(n)^0=\bigoplus_{\substack{n\text{--trees }T\\|T|=n-2}}\mathcal{P}(T)^*\otimes\Det(T)
\]
\end{remark}

To justify the notion of duality we have the following shown by Ginzburg and Kapranov \cite{ginzburgkapranov}:

\begin{theorem}
Let $\mathcal{P}$ be an admissible dg operad. Then there is a canonical quasi-isomorphism $\dual\dual\mathcal{P}\rightarrow\mathcal{P}$.
\end{theorem}

Finally we briefly recall the definition of a Koszul operad. Let $\mathcal{P}=\mathcal{P}(K,E,R)$ be a quadratic operad. As in \autoref{rem:dualdg} for every $n$ we have
\[
\dual\mathcal{P}(n)^0=
\bigoplus_{\substack{\text{binary}\\n\text{--trees }T}}E^{*}(T)\otimes\Det(T)=F(E)(n)^{\vee}=F(E^{\vee})(n)
\]
and so we have a morphism of dg operads $\gamma_{\mathcal{P}}\co\dual\mathcal{P}\rightarrow\mathcal{P}^!$ given in degree $0$ by taking the quotient of $\dual\mathcal{P}(n)^0$ by the relations in $R^{\perp}$. In fact this induces an isomorphism $H^0(\dual\mathcal{P}(n))\rightarrow\mathcal{P}^!(n)$. 

\begin{definition}
We call $\mathcal{P}$ \emph{Koszul} if $\gamma_{\mathcal{P}}$ is a quasi-isomorphism. In other words each $\dual\mathcal{P}(n)$ is exact everywhere but the right end.
\end{definition}

\begin{definition}
If $\mathcal{P}$ is Koszul then a homotopy $\mathcal{P}$--algebra\footnote{More generally, a homotopy $\mathcal{P}$--algebra is an algebra over a cofibrant replacement for $\mathcal{P}$. That $\mathcal{P}$ is Koszul means that $\dual(\mathcal{P}^!)$ is such a cofibrant replacement. For simplicity we take this to be the definition, so as to avoid the need to discuss in any detail the model category structure on $\dgop(K)$.} is an algebra over $\dual(\mathcal{P}^!)$.
\end{definition}

\section{The open KTFT modular operad}
We will now reinterpret KTFTs in the language of operads.

\begin{definition}
We define the $k$--linear extended modular operad $\oktft$ (open Klein topological field theory) as follows:
\begin{itemize}
\item For $n,g \geq 0$ and $2g+n \geq 2$ the vector space $\oktft((g,n))$ is generated by diffeomorphism classes of surfaces with $m$ handles, $u$ crosscaps and $h$ boundary components with $2m+h+u-1=g$ and with $n$ disjoint copies of the unit interval embedded in the boundary labelled by $\{1,\ldots,n\}$, with an action of $S_n$ permuting the labels.
\item Composition and contraction is given by gluing along intervals.
\end{itemize}
\end{definition}

\begin{remark}\label{rem:nounit}
Since the connected sum of $3$ crosscaps is diffeomorphic to the connected sum of $1$ handle and $1$ crosscap, the value of $2m+h+u-1$ is well defined regardless of how we choose to represent the topological type of the surface. We note that all classes of surfaces feature in $\oktft$ except the disc with no marked points and the disc with one marked point due to the condition $2g+n\geq 2$.
\end{remark}

\subsection{M\"obius trees and the operad $\mass$}
We wish to give a simple algebraic description of the modular operad $\oktft$. In order to do this we begin by defining planar trees and M\"obius trees.

\begin{definition}
A \emph{planar tree} is a labelled tree with a cyclic ordering of the half edges at each vertex. An isomorphism of planar trees is an isomorphism of labelled trees that preserves the cyclic ordering at each vertex.
\end{definition}

\begin{definition}
A \emph{M\"obius tree} is a planar tree $T$ with a colouring of the half edges by two colours. Here a `colouring by two colours' means a map $c\co \halfedges(T)\rightarrow\{0,1\}$. An isomorphism of M\"obius trees is an isomorphism of labelled trees such that at each vertex $v$ either
\begin{enumerate}
\item the map preserves the cyclic ordering at $v$ and the colouring of the half edges connected to $v$
\item the map reverses the cyclic ordering at $v$ and reverses the colouring of the half edges connected to $v$ (we refer to this as \emph{reflection at $v$}).
\end{enumerate}
\end{definition}

\begin{remark}
A planar tree can be drawn in the plane with the cyclic ordering at each vertex given by the clockwise ordering. When drawing labelled trees we shall place the output leg unlabelled at the bottom (so the induced direction on the tree is downwards). For example the following two trees are isomorphic as labelled trees, but not as planar trees:
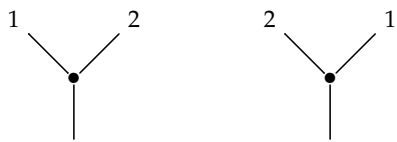
\[
\xygraph{!{<0cm,0cm>;<0.8cm,0cm>:<0cm,0.8cm>::}
{\textrm{\small$1$}}="1"&&{\textrm{\small$2$}}="2" \\
&*{\bullet}="v"\\
&="0"
"v"-"1" "v"-"2" "v"-"0"
}\qquad\qquad
\xygraph{!{<0cm,0cm>;<0.8cm,0cm>:<0cm,0.8cm>::}
{\textrm{\small$2$}}="2"&&{\textrm{\small$1$}}="1" \\
&*{\bullet}="v"\\
&="0"
"v"-"1" "v"-"2" "v"-"0"
}
\]

When drawing M\"obius trees, we shall draw the half edges coloured by $0$ as straight lines and the half edges coloured by $1$ as dotted lines. For example:
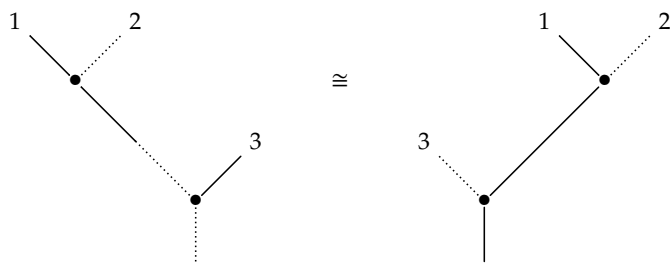
\[
\xygraph{!{<0cm,0cm>;<0.8cm,0cm>:<0cm,0.8cm>::}
{\textrm{\small$1$}}="1"&&{\textrm{\small$2$}}="2"\\
&*{\bullet}="v1"\\
&&="v1v2"&&{\textrm{\small$3$}}="3"\\
&&&*{\bullet}="v2"\\
&&&="0"
"1"-"v1" "2"-@{.}"v1" "v1"-"v1v2" "v1v2"-@{.}"v2" "v2"-"3" "v2"-@{.}"0"
}\qquad\cong\qquad
\xygraph{!{<0cm,0cm>;<0.8cm,0cm>:<0cm,0.8cm>::}
&&{\textrm{\small$1$}}="1"&&{\textrm{\small$2$}}="2"\\
&&&*{\bullet}="v1"\\
{\textrm{\small$3$}}="3"\\
&*{\bullet}="v2"\\
&="0"
"1"-"v1" "2"-@{.}"v1" "v1"-"v2" "v2"-@{.}"3" "v2"-"0"
}
\]
\end{remark}

If $T$ is a planar tree then we can define edge contraction by equipping the vertex in $T/e$ that the edge $e$ contracts to with the obvious cyclic ordering coming from the two cyclic orderings at the vertices either end of $e$, for example:

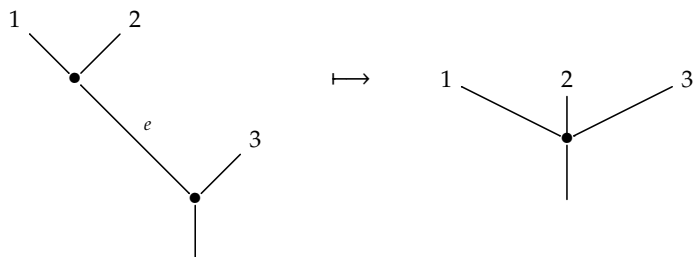
\[\xygraph{!{<0cm,0cm>;<0.8cm,0cm>:<0cm,0.8cm>::}
{\textrm{\small$1$}}="1"&&{\textrm{\small$2$}}="2"\\
&*{\bullet}="v1"\\
&&&&{\textrm{\small$3$}}="3"\\
&&&*{\bullet}="v2"\\
&&&="0"
"1"-"v1" "2"-"v1" "v1"-^e"v2" "v2"-"3" "v2"-"0"
}\qquad\longmapsto\qquad
\xygraph{!{<0cm,0cm>;<0.8cm,0cm>:<0cm,0.8cm>::}
\\{\textrm{\small$1$}}="1"&&{\textrm{\small$2$}}="2"&&{\textrm{\small$3$}}="3"\\
&&*{\bullet}="v"\\
&&="0"
"1"-"v" "2"-"v" "v"-"3" "v"-"0"
}
\]

If $T$ is a M\"obius tree and $e$ is an internal edge where both the half edges of $e$ are coloured the same then we define the tree $T/e$ as for planar trees, with the obvious colouring on $T/e$. If $f\co T'\rightarrow T$ is an isomorphism with $f(e)$ also an edge where both half edges are coloured the same we note that $T/e \cong T'/f(e)$. Therefore this is a well defined operation on isomorphism classes of M\"obius trees. Furthermore for any internal edge of $T$ we can find a tree in the same isomorphism class of $T$ such that this edge is made up of two similarly coloured half edges (by considering, if necessary, a tree with one of the vertices adjacent to the edge reflected). Therefore we have an edge contraction operation on isomorphism classes of M\"obius trees defined for any edge. For example:

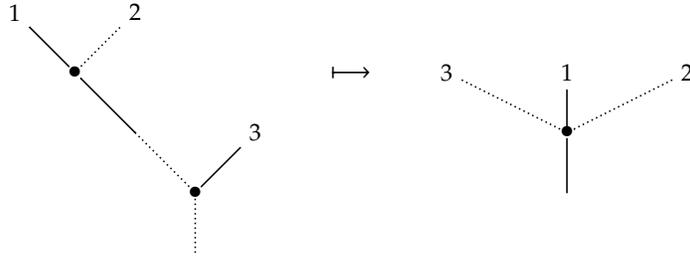
\[\xygraph{!{<0cm,0cm>;<0.8cm,0cm>:<0cm,0.8cm>::}
{\textrm{\small$1$}}="1"&&{\textrm{\small$2$}}="2"\\
&*{\bullet}="v1"\\
&&="v1v2"&&{\textrm{\small$3$}}="3"\\
&&&*{\bullet}="v2"\\
&&&="0"
"1"-"v1" "2"-@{.}"v1" "v1"-"v1v2" "v1v2"-@{.}"v2" "v2"-"3" "v2"-@{.}"0"
}\qquad\longmapsto\qquad
\xygraph{!{<0cm,0cm>;<0.8cm,0cm>:<0cm,0.8cm>::}
\\{\textrm{\small$3$}}="1"&&{\textrm{\small$1$}}="2"&&{\textrm{\small$2$}}="3"\\
&&*{\bullet}="v"\\
&&="0"
"1"-@{.}"v" "2"-"v" "v"-@{.}"3" "v"-"0"
}
\]

Finally we observe that $(T/e)/e'\cong(T/e')/e$ so it does not matter in which order we contract edges.

Now recall that the associative operad $\ass$ can be defined as consisting of the vector spaces generated by planar corollas (isomorphism classes of planar trees with $1$ vertex) where composition is given by gluing such corollas and contracting internal edges. This leads us to the following definition:

\begin{definition}
The operad $\mass$ is defined as follows:
\begin{itemize}
\item $\mass(n)$ is the vector space generated by M\"obius corollas (isomorphism classes of M\"obius trees with $1$ vertex) with $n$ inputs, where $S_n$ acts by relabelling the inputs.
\item Composition maps are given by gluing corollas and contracting the internal edges. These maps satisfy associativity since, as mentioned previously, it does not matter in which order we contract internal edges.
\end{itemize}
\end{definition}

\begin{remark}
It is easy to see that as for $\ass$ the operad $\mass$ can be given the structure of a cyclic operad in the obvious way.
\end{remark}

It is important to note that with these definitions planar trees are $\ass$--decorated trees and M\"obius trees are $\mass$--decorated trees (where we are considering decorations by the underlying $S$--modules), since at each vertex $v$ decorated by a M\"obius corolla (an element of $\mass$) one obtains a cyclic ordering at $v$ from the ordering of the inputs of the M\"obius corolla and a colouring of the half edges of $v$ from the colouring of the corolla. Therefore given a labelled tree $T$, the space of $\mass$--decorations on $T$ is generated by the set of M\"obius trees up to isomorphism whose underlying labelled tree is $T$.

\begin{remark}\label{rem:assinmass}
$\ass$ is a suboperad of $\mass$. In fact $\mass$ is obtained from the operad generated by adjoining an involutive operation to $\ass$ by taking the quotient by the ideal generated by the reflection relation for the binary operation:
\[\xygraph{!{<0cm,0cm>;<0.8cm,0cm>:<0cm,0.8cm>::}
{\textrm{\small$1$}}="1"&&{\textrm{\small$2$}}="2" \\
&*{\bullet}="v"\\
&="0"\\
&*{\circ}="a"\\
&="00"
"v"-"1" "v"-"2" "v"-"0" "0"-"a" "a"-"00"
}\qquad=\qquad
\xygraph{!{<0cm,0cm>;<0.8cm,0cm>:<0cm,0.8cm>::}
{\textrm{\small$2$}}="2"&&{\textrm{\small$1$}}="1" \\
*{\circ}="a1"&&*{\circ}="a2"\\
="22"&&="11"\\
&*{\bullet}="v"\\
&="0"
"v"-"11" "v"-"22" "v"-"0" "11"-"a2" "22"-"a1" "a1"-"2" "a2"-"1"
}\]
Here
$\xygraph{!{;<0.5cm,0cm>:::}
="l"&*{\circ}="a"&="r"
"l"-"a" "a"-"r"}$
denotes the involution.
\end{remark}

\begin{proposition}
$\mass$ is the operad governing (non-unital) associative algebras with an involutive anti-automorphism.
\end{proposition}

\begin{proof}
This follows immediately from \autoref{rem:assinmass}.
\end{proof}

\subsection{M\"obius graphs and $\oktft$ as the modular closure of  $\mass$}
We recall that in the case of oriented topological field theories the corresponding modular operads are the modular closures of their genus $0$ part which in turn are identified with the commutative and associative operads: $\ctft\cong\modcom$ and $\otft\cong\modass$ (this formulation in terms of modular operads can be seen in Chuang and Lazarev \cite[Theorem 2.7]{chuanglazarev}). In particular the genus $0$ cyclic part contains all the relations. These results are modular operad versions of \autoref{prop:ctft} and \autoref{prop:otft} identifying $2$--dimensional TFTs as Frobenius algebras. 

The same is true for $\oktft$, giving us the desired simple algebraic description of $\oktft$.

\begin{theorem}\label{thm:oktftmass}
$\oktft\cong\modmass$.
\end{theorem}

Before proving \autoref{thm:oktftmass} we will identify the operad $\modmass$ in terms of graphs. Therefore we need to extend our definitions of M\"obius trees to graphs.

\begin{definition}
A \emph{ribbon graph} is a graph with all vertices having valence at least $2$ equipped with a cyclic ordering of the half edges at each vertex and a labelling of the legs. An isomorphism of ribbon graphs is an isomorphism of graphs that preserves the cyclic ordering at each vertex and the labelling of the legs.
\end{definition}

\begin{definition}
A \emph{M\"obius graph} is a ribbon graph with a colouring of the half edges by two colours. An isomorphism of M\"obius graphs is an isomorphism of graphs preserving the labelling of the legs such that at each vertex $v$ either
\begin{enumerate}
\item the map preserves the cyclic ordering at $v$ and the colouring of the half edges at $v$
\item the map reverses the cyclic ordering at $v$ and reverses the colouring of the half edges connected to $v$ (again we refer to this as reflection at $v$).
\end{enumerate}
\end{definition}

\begin{remark}
Obviously our notions of planar and M\"obius trees correspond to ribbon and M\"obius graphs with no loops. Once again we can draw these graphs in the plane (although possibly with some edges intersecting of course) with the cyclic ordering at each vertex given by the clockwise ordering:
\[
\xygraph{!{<0cm,0cm>;<1.2cm,0cm>:<0cm,1.2cm>::}
!{(0,0)}*{\bullet}="v1"
!{(2,0)}*{\bullet}="v2"
!{(3,0.5)}*+{\textrm{\small$2$}}="2"
!{(3,-0.5)}*+{\textrm{\small$1$}}="1"
!{(0,-1)}*+{\textrm{\small$3$}}="3"
!{(-1,0)}*{}="loop"
!{(1,0.4)}*{}="u"
!{(1,-0.4)}*{}="d"
"v1" -@`{c+(0,-0.5),p+(0,-0.5)}@{.} "loop"
"loop" -@`{c+(0,0.5),p+(0,0.5)} "v1"
"v1"-"3"
"v1" -@`{c+(0.25,0.3),p+(-0.25,0)} "u"
"u" -@`{c+(0.25,0),p+(-0.25,0.3)} "v2"
"v1" -@`{c+(0.25,-0.3),p+(-0.25,0)} "d"
"d" -@`{c+(0.25,0),p+(-0.25,-0.3)}@{.} "v2"
"v2"-@{.} "2"
"v2"-"1"
}
\]
\end{remark}

\begin{remark}\label{rem:decgraphs}
Ribbon and M\"obius graphs are $\underline{\ass}$--decorated graphs and $\underline{\mass}$--decorated graphs respectively.
\end{remark}

If $G$ is a ribbon graph and $e$ is an internal edge of $G$ which is not a loop we can define edge contraction by equipping $G/e$ with the obvious cyclic ordering coming from $G$ as for trees.

If $G$ is a M\"obius graph and $e$ is an internal edge of $G$ that is not a loop, where both the half edges of $e$ are the same colour then we define $G/e$ as we did for trees. Further we observe that as for M\"obius trees this is well defined on isomorphism classes and can be extended to all internal edges except loops regardless of colour.

Let $G$ be a M\"obius or ribbon graph. Given two internal edges $e$ and $e'$ of $G$ that are not loops we have $(G/e)/e'\cong (G/e')/e$ provided both sides are defined. However if $e$ and $e'$ are connected to the same vertices contracting one will make the other into a loop. As a result we do not obtain a well defined operation on graphs by repeatedly contracting edges until we have only one vertex, which we did for trees. See \autoref{fig:graphcontract} for an example.

\begin{figure}[ht!]
\centering
\[
\begin{xy}
(0,10)*\xybox{
(0,0)*\xybox{
\xygraph{!{<0cm,0cm>;<1cm,0cm>:<0cm,1cm>::}
!{(0,0)}*{\bullet}="v1"
!{(2,0)}*{\bullet}="v2"
!{(0.9,-0.35)}*{}="o"
"v1" -@`{(0.5,1),(1.5,1)} "v2" ^{e_1}
"v1" -@`{(0.5,-1),(1.5,-1)} "v2" _{e_2}
"v1" - "v2"
"v1" - "o"
}};
(-15,-25)*\xybox{
\xygraph{!{<0cm,0cm>;<1cm,0cm>:<0cm,1cm>::}
!{(0,0)}*{\bullet}="v1"
!{(0,-0.8)}*{}="o"
"v1" -@`{(-1.8,-1.8),(-1.8,+1.8)} "v1" ^{}
"v1" -@`{(+1.8,-1.8),(+1.8,+1.8)} "v1"
"v1" - "o"
}};
(35,-25)*\xybox{
\xygraph{!{<0cm,0cm>;<1cm,0cm>:<0cm,1cm>::}
!{(0,0)}*{\bullet}="v1"
!{(-0.8,0)}*{}="o"
"v1" -@`{(-1.8,-1.8),(-1.8,+1.8)} "v1"
"v1" -@`{(+1.8,-1.8),(+1.8,+1.8)} "v1"
"v1" - "o"
}};
(10,-25)*{\textrm{\small$\ncong$}};
(-5,0);(-15,-18)
**\crv{(-13,-7)};?(1)*\dir2{>};
(25,0);(35,-18)
**\crv{(33,-7)}?(1)*\dir2{>};
(-17,-7)*{\textrm{\small$G/e_1$}};
(37,-7)*{\textrm{\small$G/e_2$}}
}
\end{xy}
\]
\caption{Contracting all edges that are not loops is not well defined for ribbon graphs.}
\label{fig:graphcontract}
\end{figure}
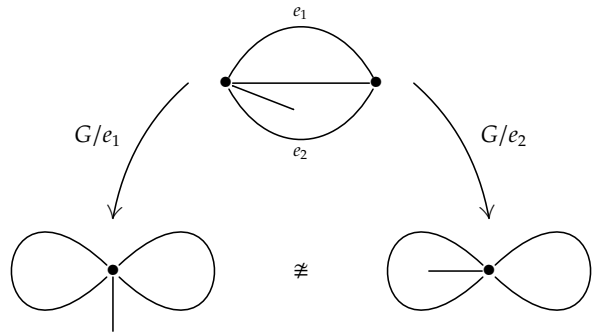

Consequently we define a relation $\approx$ on (isomorphism classes of) M\"obius or ribbon graphs where $G\approx G'$ whenever one is obtained from the other by an edge contraction so that $G=G'/e$ or $G'=G/e$. The transitive closure of this is then an equivalence relation we will also denote by $\approx$. All elements of an equivalence class have the same genus and the same number of legs. There is also at least one graph with one vertex in each class. Observe that the space of corollas is obtained from the space of trees modulo this relation. As a result $\ass$ and $\mass$ could be defined as the operads of planar and M\"obius trees modulo $\approx$ with composition given by gluing trees.

We can now describe the operad $\modmass$. Recall that $\modass$ is the modular operad given by ribbon graphs up to the relation $\approx$ with composition and contraction given by the gluing legs of graphs. This is true since the modular closure of $\ass$ is generated by freely adding contractions and applying just those relations necessary to ensure that associativity and equivariance holds. More explicitly, we can first identify the space of planar corollas with contractions added in freely as ribbon graphs with $1$ vertex with loops directed and ordered. The equivariance condition means that we must forget the directions and order of the loops. Composition is given by gluing such objects and contracting internal edges that are not loops. The associativity condition requires that it does not matter in what order we contract internal edges. The relation $\approx$ (induced on ribbon graphs with $1$ vertex) is precisely the minimal relation required to ensure this is true. For example the bottom two graphs in \autoref{fig:graphcontract} are equivalent under $\approx$ but not isomorphic. It is clear that the same argument holds true for $\modmass$.

\begin{lemma}
The extended modular operad $\modmass$ can be described as follows:
\begin{itemize}
\item If $2g+n\geq 2$ then $\modmass((g,n))$ is the vector space generated by isomorphism classes of M\"obius graphs with $n$ legs and genus $g$ modulo the relation $\approx$.
\item Composition and contraction are given by gluing legs of graphs.
\qed
\end{itemize}
\end{lemma}

We next describe the main construction arising in the proof of \autoref{thm:oktftmass}. Let $G$ be a ribbon graph. The ribbon structure of $G$ allows one to replace each edge with a thin oriented strip and each vertex with an oriented disc using the cyclic ordering to glue the strips to discs in an orientation preserving manner. As such we obtain an oriented surface with boundary well defined up to diffeomorphism. Further we can identify the legs as labelled copies of the interval embedded in the boundary in an orientation preserving manner.

We can generalise this to a similar construction for M\"obius graphs. We replace each vertex $v$ with an oriented disc and we replace each edge $e$ with an oriented strip. We then use the cyclic ordering to glue the strips to discs. If the edge $e$ is connected to the vertex $v$ by a half edge coloured by $0$ we glue the strip corresponding to $e$ to the disc corresponding to $v$ such that their orientations are compatible. However if the half edge is coloured by $1$ we glue such that the orientations are not compatible. We identify the legs as labelled copies of the interval embedded compatibly with the disc's orientation if the leg is coloured by $0$ and incompatibly otherwise. We finally forget all the orientations on each part of our surface. This yields a surface that is not necessarily orientable. These constructions coincide for those M\"obius graphs that are just ribbon graphs (that is, graphs all of the same colour).

We should verify this construction is well defined up to diffeomorphism. However this is clear since applying the reflection relation at a vertex $v$ corresponds to constructing a surface identical everywhere except at the disc corresponding to $v$ which has been reflected (see \autoref{fig:discreflection}). Reflection of the disc is a smooth (orientation reversing) map so the construction yields a diffeomorphic surface.

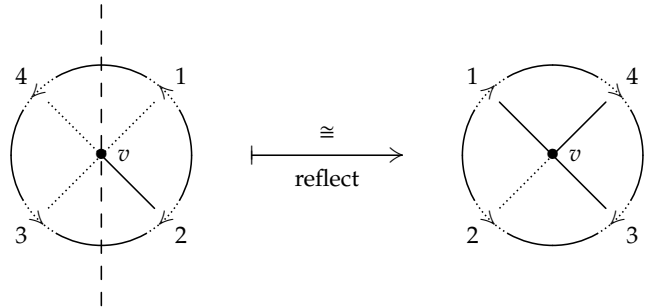
\begin{figure}[ht!]
\centering
\[
\begin{xy}
(0,0)*\xybox{*\xybox{\ellipse(12,12):a(60),:a(120){-}},
*\xybox{\ellipse(12,12):a(150),:a(210){-}},
*\xybox{\ellipse(12,12):a(240),:a(300){-}},
*\xybox{\ellipse(12,12):a(330),:a(30){-}},
*\xybox{\ellipse(12,12):a(30),:a(60){.}},
*\xybox{\ellipse(12,12):a(120),:a(150){.}},
*\xybox{\ellipse(12,12):a(210),:a(240){.}},
*\xybox{\ellipse(12,12):a(300),:a(330){.}},
*{\bullet},(3,0)*{\textrm{\small$v$}},
(0,0);a(45):(0,0);(10,0)**@{.}+(5,0)*{\textrm{\small$1$}},
(0,0);(-10,0)**@{.}-(5,0)*{\textrm{\small$3$}},
(-1,12)*\dir2{>},(-1,-12)*\dir2{>},
(0,0);(0,1):(0,0);(10,0)**@{.}+(5,0)*{\textrm{\small$4$}},
(0,0);(-10,0)**@{-}-(5,0)*{\textrm{\small$2$}},
(-1,12)*\dir2{>},(1,-12)*\dir2{<}},
(0,20);(0,-20)**@{--},
(60,0)*\xybox{*\xybox{\ellipse(12,12):a(60),:a(120){-}},
*\xybox{\ellipse(12,12):a(150),:a(210){-}},
*\xybox{\ellipse(12,12):a(240),:a(300){-}},
*\xybox{\ellipse(12,12):a(330),:a(30){-}},
*\xybox{\ellipse(12,12):a(30),:a(60){.}},
*\xybox{\ellipse(12,12):a(120),:a(150){.}},
*\xybox{\ellipse(12,12):a(210),:a(240){.}},
*\xybox{\ellipse(12,12):a(300),:a(330){.}},
*{\bullet},(3,0)*{\textrm{\small$v$}},
(0,0);a(45):(0,0);(10,0)**@{-}+(5,0)*{\textrm{\small$4$}},
(0,0);(-10,0)**@{.}-(5,0)*{\textrm{\small$2$}},
(1,12)*\dir2{<},(-1,-12)*\dir2{>},
(0,0);(0,1):(0,0);(10,0)**@{-}+(5,0)*{\textrm{\small$1$}},
(0,0);(-10,0)**@{-}-(5,0)*{\textrm{\small$3$}},
(-1,12)*\dir2{>},(-1,-12)*\dir2{>}},
(20,0);(40,0)**@{-}?(0)*\dir2{|}?(1)*\dir2{>}?(0.5)-(0,3)*{\textrm{\small{reflect}}}+(0,6)*{\textrm{\small$\cong$}}
\end{xy}
\]
\caption{The reflection relation at a vertex $v$ corresponds to reflection of the disc associated to $v$.}
\label{fig:discreflection}
\end{figure}

Since contracting an edge corresponds to contracting a strip this construction is in fact well defined on equivalence classes of $\approx$. \autoref{fig:basicsurfaces} shows the basic graphs corresponding to a handle, a crosscap and a boundary component (annulus). From this we can see that if a M\"obius graph has genus $g$ and the corresponding surface consists of $m$ handles, $u$ crosscaps and $h$ boundary components then $2m+h+u-1 = g$. This means that by this construction we obtain maps of the underlying vector spaces $\modmass((g,n))\rightarrow\oktft((g,n))$.

\begin{figure}[ht!]
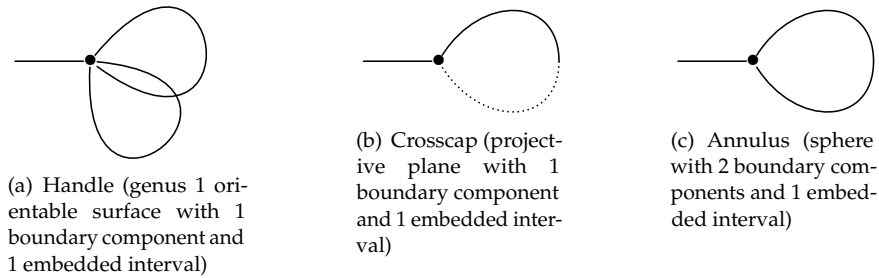

\centering
\subfigure[Handle (genus $1$ orientable surface with $1$ boundary component and $1$ embedded interval)]{
\xygraph{!{<0cm,0cm>;<1cm,0cm>:<0cm,1cm>::}
!{(0,0)}*{\bullet}="v1"
!{(-1,0)}*{}="o"
"v1" -@`{(+2.8,-0.2),(-0.2,-2.8)} "v1"
"v1" -@`{(+2.3,-1.8),(+1.8,+2.3)} "v1"
"v1" - "o"}}
\qquad\qquad
\subfigure[Crosscap (projective plane with $1$ boundary component and $1$ embedded interval)]{
\xygraph{!{<0cm,0cm>;<1cm,0cm>:<0cm,1cm>::}
!{(0,0)}*{\bullet}="v1"
!{(-1,0)}*{}="o"
!{(1.6,0)}*{}="loop"
"v1" -@`{c+(0.5,-0.9),p+(0,-0.9)}@{.} "loop"
"loop" -@`{c+(0,0.9),p+(0.5,0.9)} "v1"
"v1"-"o"}}
\qquad\qquad
\subfigure[Annulus (sphere with $2$ boundary components and $1$ embedded interval)]{
\xygraph{!{<0cm,0cm>;<1cm,0cm>:<0cm,1cm>::}
!{(0,0)}*{\bullet}="v1"
!{(-1,0)}*{}="o"
!{(1.6,0)}*{}="loop"
"v1" -@`{c+(0.5,-0.9),p+(0,-0.9)} "loop"
"loop" -@`{c+(0,0.9),p+(0.5,0.9)} "v1"
"v1"-"o"}}
\caption{M\"obius graphs corresponding to basic surfaces}
\label{fig:basicsurfaces}
\end{figure}

It is also clear that these constructions are compatible with operadic gluings so we obtain maps $\modass\rightarrow\otft$ and $\modmass\rightarrow\oktft$. As shown by Chuang and Lazarev \cite{chuanglazarev} the former is an isomorphism. We can now prove \autoref{thm:oktftmass} by showing the latter is too.

\begin{proof}[Proof of \autoref{thm:oktftmass}]
It is sufficient to show the map $\modmass\rightarrow\oktft$ described above is an isomorphism of the underlying $S$--modules. The surjectivity of this map follows from the classification of unoriented topological surfaces with boundary and \autoref{fig:basicsurfaces}, which shows how to build a surface of any topological type. To see that it is injective it is necessary to show that any two graphs with the same topological type are equivalent\footnote{This is analogous to proving the sufficiency of a set of relations on the generators of $\kcob$ if we were proving \autoref{prop:oktft} without mention of operads.} under the relation $\approx$. We first note two graphs that are equivalent as shown by the following diagram:
\begin{equation}\label{eqn:rel1}
\begin{split}
\xygraph{!{<0cm,0cm>;<1cm,0cm>:<0cm,1cm>::}
!{(0,0)}*{\bullet}="v1"
!{(-1,+0.5)}*+{\textrm{\small$1$}}="1"
!{(-1,-0.5)}*+{\textrm{\small$2$}}="2"
!{(1.3,0)}*{}="loop"
"v1" -@`{c+(0.5,-0.7),p+(0,-0.7)} "loop"
"loop" -@`{c+(0,0.7),p+(0.5,0.7)}@{.} "v1"
"v1"-"1" "v1"-"2"}
\quad\approx\quad
\xygraph{!{<0cm,0cm>;<1cm,0cm>:<0cm,1cm>::}
!{(0,0)}*{\bullet}="v1"
!{(1.5,0)}*{\bullet}="v2"
!{(-1,0)}*+{\textrm{\small$2$}}="2"
!{(2.5,0)}*+{\textrm{\small$1$}}="1"
!{(0.75,0.4)}*{}="u"
!{(0.75,-0.4)}*{}="d"
"v1" -@`{c+(0.25,0.3),p+(-0.25,0)} "u"
"u" -@`{c+(0.25,0),p+(-0.25,0.3)} "v2"
"v1" -@`{c+(0.25,-0.3),p+(-0.25,0)} "d"
"d" -@`{c+(0.25,0),p+(-0.25,-0.3)}@{.} "v2"
"v2"- "1"
"v1"-"2"}
\\
\cong\quad
\xygraph{!{<0cm,0cm>;<1cm,0cm>:<0cm,1cm>::}
!{(0,0)}*{\bullet}="v1"
!{(1.5,0)}*{\bullet}="v2"
!{(-1,0)}*+{\textrm{\small$2$}}="2"
!{(0.6,0)}*+{\textrm{\small$1$}}="1"
!{(0.75,0.4)}*{}="u"
!{(0.75,-0.4)}*{}="d"
"v1" -@`{c+(0.25,0.3),p+(-0.25,0)} "u"
"u" -@`{c+(0.25,0),p+(-0.25,0.3)}@{.} "v2"
"v1" -@`{c+(0.25,-0.3),p+(-0.25,0)} "d"
"d" -@`{c+(0.25,0),p+(-0.25,-0.3)} "v2"
"v2"-@{.} "1"
"v1"-"2"}
\quad\approx\quad
\xygraph{!{<0cm,0cm>;<1cm,0cm>:<0cm,1cm>::}
!{(0,0)}*{\bullet}="v1"
!{(1,0)}*+{\textrm{\small$1$}}="1"
!{(-1,0)}*+{\textrm{\small$2$}}="2"
!{(1.3,0)}*{}="loop"
"v1" -@`{c+(0.5,-0.7),p+(0,-0.7)}@{.} "loop"
"loop" -@`{c+(0,0.7),p+(0.5,0.7)} "v1"
"v1"-@{.}"1" "v1"-"2"}
\end{split}
\end{equation}

Using this we can prove the relation corresponding to the fact that the connected sum of $3$ crosscaps corresponds to the connected sum of a handle and a crosscap:

\begin{equation}\label{eqn:rel2}
\xygraph{!{<0cm,0cm>;<1cm,0cm>:<0cm,1cm>::}
!{(0,0)}*{\bullet}="v1"
!{(1.3,0)}*{}="loop"
!{(-1.3,0)}*{}="loop2"
!{(0,+1.3)}*{}="loop3"
!{(0,-1)}*{}="o"
"v1" -@`{c+(0.5,-0.5),p+(0,-0.5)} "loop"
"loop" -@`{c+(0,0.5),p+(0.5,0.5)}@{.} "v1"
"v1" -@`{c+(-0.5,-0.5),p+(0,-0.5)} "loop2"
"loop2" -@`{c+(0,0.5),p+(-0.5,0.5)}@{.} "v1"
"v1" -@`{c+(-0.5,0.5),p+(-0.5,0)} "loop3"
"loop3" -@`{c+(0.5,0),p+(0.5,0.5)}@{.} "v1"
"v1"-"o"
}
\quad\approx\quad
\xygraph{!{<0cm,0cm>;<1cm,0cm>:<0cm,1cm>::}
!{(0,0)}*{\bullet}="v1"
!{(0,-1)}*{}="o"
!{(-1.3,0)}*{}="loop2"
"v1" -@`{(+2.8,+0.2),(-0.2,+2.8)} "v1"
"v1" -@`{(+2.3,+1.8),(+1.8,-2.3)} "v1"
"v1" -@`{c+(-0.5,-0.5),p+(0,-0.5)} "loop2"
"loop2" -@`{c+(0,0.5),p+(-0.5,0.5)}@{.} "v1"
"v1" - "o"}
\end{equation}
This can be shown by drawing graphs and repeatedly applying relation \ref{eqn:rel1}. We leave this to the reader.

Given a graph we can contract all internal edges that are not loops. Then we can ensure that all loops which are composed of half edges of the same colour (which we will call untwisted loops) are all coloured by $0$ since a loop coloured by $1$ is equivalent to a loop coloured by $0$ (by expanding the loop into $2$ edges of different colours and contracting the edge of colour $1$). We then apply  relation \ref{eqn:rel1} repeatedly to ensure that all the twisted loops are adjacent and have no half edges or legs on their inside. Finally we use relation \ref{eqn:rel2} repeatedly until there are at most $2$ twisted loops remaining. Therefore any graph is equivalent to a `normal form' consisting of either a ribbon graph with $1$ vertex or the connected sum (by which we mean vertices connected by a single untwisted edge) of a ribbon graph with $1$ vertex and a M\"obius graph with $1$ vertex and at most $2$ twisted loops. If two graphs have the same topological type then in this normal form the M\"obius graph components must be isomorphic. But the ribbon graph components must therefore be of the same topological type and we know that they are equivalent under the relation $\approx$ since $\modass\rightarrow\otft$ is an isomorphism.
\end{proof}

\begin{corollary}
Algebras over the modular operad $\oktft$ are (non-unital) symmetric Frobenius algebras together with an involution preserving the inner product.
\end{corollary}

\begin{proof}
Since $\oktft$ is the modular closure of its genus $0$ cyclic operad then algebras over $\oktft$ are simply algebras over $\mass$ considered as a cyclic operad (this is immediate as in \cite[Proposition 2.4]{chuanglazarev}). Cyclic $\mass$--algebras are just $\mass$--algebras with an invariant inner product which are precisely Frobenius algebras with an involution.
\end{proof}

In the formulation of topological field theories as a symmetric monoidal functor from some cobordism category the only difference is that we have a unit and counit (see \autoref{rem:nounit}). Therefore we have now fulfilled our earlier promise and shown:

\begin{corollary}[\autoref{prop:oktft}]
Open Klein topological field theories of dimension $2$ are equivalent to symmetric Frobenius algebras together with an involutive anti-automorphism preserving the pairing.
\qed
\end{corollary}

\subsection{Cobar duality for $\mass$}
We will now consider the operad $\mass$ in more detail.

Recall that the free operad generated by the vector space $\ass(2)$ over $k$ is the operad of binary planar trees and that $\ass$ is the quotient of this by the associativity relation. It is therefore a quadratic operad since the associativity relation is a quadratic relation. Further $\ass^!\cong\ass$.

$\mass$ is also quadratic: let $K$ be the semisimple algebra $\mass(1)=\langle 1, a\rangle / (a^2=1)=k[\mathbb{Z}_2]$. By taking the quotient of the free operad generated by the $(K,K^{\otimes 2})$--bimodule $\mass(2)$ by the associativity relation we obtain $\mass$. In fact, as we shall see, all the usual duality properties of $\ass$ hold for $\mass$.

\begin{proposition}\label{prop:massqdual}
$(\mass)^!\cong\mass$.
\end{proposition}

\begin{proof}
As in the case of $\ass$ we can simply give an explicit isomorphism. The only potential difficulty arises from the quadratic dual being twisted by the sign representation, however this turns out not to be an issue. Let $K=\mass(1)=\langle 1, a\rangle / (a^2=1)=k[\mathbb{Z}_2]$ and $E=\mass(2)$.

Let $\psi_1\co K\rightarrow K^{\mathrm{op}}=K$ be the isomorphism with $\psi_1(a)=-a$. We define a map $\psi_2\co E\rightarrow E^{\vee}$ of $k$--linear $S_2$--representations as follows. Let $\sigma\in S_2$ denote the transposition and denote by $m$ the corolla:
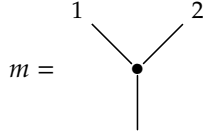
\[m = 
\xygraph{!{<0cm,0cm>;<0.8cm,0cm>:<0cm,0.8cm>::}
{\textrm{\small$1$}}="1"&&{\textrm{\small$2$}}="2" \\
&*{\bullet}="v"\\
&="0"\\
"v"-"1" "v"-"2" "v"-"0"}
\]

Let $B=\{m,\sigma m(a \otimes 1), \sigma m(1 \otimes a), m(a\otimes a)\}$. Observe that $B$ is a $K$--linear basis for the left $K$--module $E$. For each $e\in B$ we denote by $e^*\in\Hom_K(E,K)$ the element of the dual basis for $E^{\vee}$. By this we mean $e^*$ is defined on each $e'\in B$ by $e^*(e') = 1$ if $e'=e$ and $e^*(e')=0$ otherwise. For $e\in B$ set $\psi_2(e)=e^*$. Now observe that $B$ also freely generates $E$ as a $k$--linear $S_2$--module, so $\psi_2$ extends to an isomorphism $E\rightarrow E^{\vee}$ of $k$--linear $S_2$--modules. Explicitly this sends an element $f$ of the $K$--linear basis $\sigma B=\{\sigma m, m(a\otimes 1), m(1\otimes a), \sigma m(a\otimes a)\}$ to $-f^*$ (where $f^*$ denotes the element of the dual basis of $\sigma B$).

We claim the map $\Psi = (\psi_1,\psi_2,0,\ldots)$ gives an isomorphism of $K$--collections. By definition it is an isomorphism of $S$--modules. Some straightforward calculations verify that it is also a map of $K$--collections. For example, for $e\in B$ we have $\psi_2(ae)=\psi_2(\sigma e (a\otimes a))=\sigma(\psi_2(e(a\otimes a))) = -a\psi_2(e) = \psi_1(a)\psi_2(e)$.

Therefore $F(E)\cong F(E^{\vee})$ with $\Psi$ extending to an isomorphism of operads. Let $R\subset F(E)(3)$ be the $S_3$--stable sub-bimodule generated by the associativity relation for $m$. It remains to show that $R^{\perp}=\Psi(R)$. Since $\dim(R)=\dim(F(E)(3))/2$ it is sufficient to check the associativity relation for $m$ is in $R^{\perp}$. This is a simple check, which we omit.
\end{proof}

We now describe the cobar construction for $\mass$. To do this we will need to identify the space of decorations on a tree $T$ by the underlying $K$--collection of $\mass$ (recall that in general this is different from the space of decorations on a tree $T$ by the underlying $S$--module). We therefore define the notion of reduced M\"obius and planar trees.

\begin{definition}\label{def:reducedtree}
Given a planar or M\"obius tree with at least one vertex of valence at least $3$ we can associate (possibly several) reduced trees by repeatedly contracting an edge attached to a vertex of valence $2$ until the tree is reduced. We say two reduced M\"obius or planar trees are equivalent if they are obtained from the same tree in this manner. When we refer to a reduced M\"obius or planar tree we will mean an isomorphism class of reduced M\"obius or planar trees up to this equivalence.
\end{definition}

\begin{remark}\label{rem:reducedmobius}
This has no effect for planar trees but for M\"obius trees we have that the following reduced M\"obius trees are the same for example:
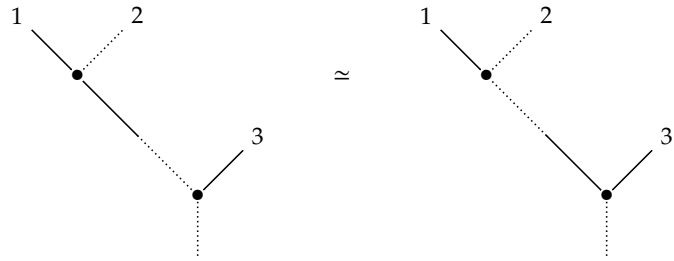
\[
\xygraph{!{<0cm,0cm>;<0.8cm,0cm>:<0cm,0.8cm>::}
{\textrm{\small$1$}}="1"&&{\textrm{\small$2$}}="2"\\
&*{\bullet}="v1"\\
&&="v1v2"&&{\textrm{\small$3$}}="3"\\
&&&*{\bullet}="v2"\\
&&&="0"
"1"-"v1" "2"-@{.}"v1" "v1"-"v1v2" "v1v2"-@{.}"v2" "v2"-"3" "v2"-@{.}"0"
}\qquad\simeq\qquad
\xygraph{!{<0cm,0cm>;<0.8cm,0cm>:<0cm,0.8cm>::}
{\textrm{\small$1$}}="1"&&{\textrm{\small$2$}}="2"\\
&*{\bullet}="v1"\\
&&="v1v2"&&{\textrm{\small$3$}}="3"\\
&&&*{\bullet}="v2"\\
&&&="0"
"1"-"v1" "2"-@{.}"v1" "v1"-@{.}"v1v2" "v1v2"-"v2" "v2"-"3" "v2"-@{.}"0"
}
\]
Also note that edge contraction is still well defined on reduced M\"obius trees.
\end{remark}

Thus defined, the space of decorations on a reduced tree $T$ by the $k$--collection $\ass$ is generated by the set of reduced planar trees whose underlying tree is $T$. The space of decorations on $T$ by the $K$--collection $\mass$ is spanned by the set of reduced M\"obius trees whose underlying tree is $T$.

\begin{definition}\label{def:orientedtree}
The space of oriented planar (or M\"obius) trees is generated by planar (or M\"obius) trees equipped with an ordering of the internal edges subject to the relations arising by requiring that swapping the order of two edges is the same as multiplying by $-1$ so that, for example, the space $\ass(T)\otimes\det(T)$ (see \autoref{subsec:koszulduality}) can then be identified with the space of reduced oriented planar trees whose underlying tree is $T$.
\end{definition}

Now recall the operad $\ass$ is Koszul, $\ass(n)\cong\ass(n)^*$ and $\dass$ is the operad of reduced oriented planar trees (where $S_n$ acts by the sign representation) which governs $A_\infty$--algebras.

By sending a corolla $m$ in $\mass(n)$ to the map $\psi(m)(m)=1$, $\psi(m)(am)=a$, $\psi(m)(m')=0$ for the other corollas $m'$ (similar to the map in the proof of \autoref{prop:massqdual} but without the different signs since $\mass(n)^*$ is not twisted by the sign representation) we obtain an $S_n$--equivariant map $\mass(n)\rightarrow\mass(n)^*$ that is also a map of $(K,K^{\otimes n})$--bimodules. Therefore the underlying spaces $C(\mass)(n)$ are spanned by reduced oriented M\"obius trees with $n$ inputs, graded appropriately by the number of internal edges. Composition corresponds to gluing oriented trees. The differential corresponds to expanding vertices of valence greater than 3, for example:
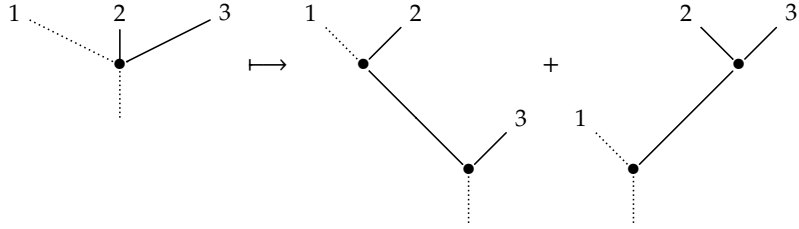
\[
\xygraph{!{<0cm,0cm>;<0.7cm,0cm>:<0cm,0.7cm>::}
{\textrm{\small$1$}}="1"&&{\textrm{\small$2$}}="2"&&{\textrm{\small$3$}}="3"\\
&&*{\bullet}="v"\\
&&="0"
"1"-@{.}"v" "2"-"v" "3"-"v" "v"-@{.}"0"
}
\longmapsto
\xygraph{!{<0cm,0cm>;<0.7cm,0cm>:<0cm,0.7cm>::}
{\textrm{\small$1$}}="1"&&{\textrm{\small$2$}}="2"\\
&*{\bullet}="v1"\\
&&="v1v2"&&{\textrm{\small$3$}}="3"\\
&&&*{\bullet}="v2"\\
&&&="0"
"1"-@{.}"v1" "2"-"v1" "v1"-"v1v2" "v1v2"-"v2" "v2"-"3" "v2"-@{.}"0"
}
+
\xygraph{!{<0cm,0cm>;<0.7cm,0cm>:<0cm,0.7cm>::}
&&{\textrm{\small$2$}}="2"&&{\textrm{\small$3$}}="3"\\
&&&*{\bullet}="v1"\\
{\textrm{\small$1$}}="1"\\
&*{\bullet}="v2"\\
&="0"
"2"-"v1" "3"-"v1" "v1"-"v2" "v2"-@{.}"1" "v2"-@{.}"0"
}
\]
When drawing oriented M\"obius trees like the above we give them the orientation on the edges by ordering the internal edges from left to right, from bottom to top.

\begin{remark}\label{rem:cassincmass}
Observe as in \autoref{rem:assinmass} that $C(\ass)$ is a suboperad of $C(\mass)$, since planar trees are M\"obius trees with straight edges. Indeed once again $C(\mass)$ is generated by adjoining an involution of degree $0$ to $C(\ass)$, this time modulo the reflection relation on all corollas.
\end{remark}

\begin{lemma}\label{lem:koszul}
As dg vector spaces $C(\mass)(n)=\bigoplus^{2^n}C(\ass)(n)$.
\end{lemma}

\begin{proof}
Given a reduced M\"obius tree $T$ we will find a unique reduced tree $T'$ isomorphic to it with the root and all internal edges coloured by $0$ (in other words all straight lines). This is then a tree in $C(\ass)$ with coloured inputs of which there are $2^n$ possibilities. The differentials clearly coincide as $K$ is concentrated in degree $0$.

To find such a tree we apply a sequence of transformations to $T$ that result in an isomorphic tree at each stage. The basic transformations are either using the reflection relation at a vertex or swapping the colourings of an edge (as in, for example, \autoref{rem:reducedmobius}) in a reduced tree. The process is as follows: we first apply the reflection relation if necessary to ensure the root is coloured by $0$. We then apply the edge relations to all the inputs of the bottom vertex to ensure all the half edges connected to the bottom vertex are coloured by $0$. We then repeat this inductively at each of the vertices at the next level until we have transformed the whole tree. The resulting tree $T'$ is unique since there is no choice in this process.
\end{proof}

\begin{corollary}\label{cor:koszul}
$\mass$ is Koszul.
\end{corollary}

\begin{proof}
$\mass$ is Koszul if and only if the complexes $C(\mass)(n)$ are exact everywhere but the right end. This is true by \autoref{lem:koszul} since $\ass$ is Koszul.
\end{proof}

We now consider homotopy $\mass$--algebras. That is to say, algebras over $\dmass$. From \autoref{rem:cassincmass} we have that $\dmass$ is generated by the operations $m_i\in\dass$ for $i\geq 2$, together with an involution of degree $0$, which by convention we will say corresponds to $-a\in K$. The differential on $\dmass$ is the same as that on $\dass$ for the operations $m_i$ so it yields the usual $A_\infty$ conditions. The reflection relation on the $m_i$ however introduces an extra sign since we have now twisted $C(\mass)$ by the sign representation. The sign of the permutation reversing $n$ labels is $(-1)^{n(n-1)/2}$. So we have shown the following:

\begin{proposition}
Algebras over $\dmass$ are $A_\infty$--algebras with an involution such that
\[
m_n(x_1,\ldots,x_n)^*=(-1)^{\epsilon}(-1)^{n(n+1)/2-1}m_n(x_n^*,\ldots,x_1^*)
\]
where $\epsilon=\sum_{i=1}^{n}\degree{x}_i\left(\sum_{j=i+1}^{n}\degree{x}_j\right)$ arises from permuting the $x_i$ with degrees $\degree{x}_i$.
\qed
\end{proposition}

\begin{remark}
$\dmass$ can be given the structure of a cyclic operad in the obvious way (permuting the labellings of M\"obius trees).
\end{remark}

An important operad for us (which we shall see later controls open Klein topological conformal field theory) is the modular operad $\moddmass$ which we shall now describe explicitly by identifying it as the operad of reduced oriented M\"obius graphs with the expanding differential. This is the analogue of the fact that $\moddass$ is the operad of reduced oriented ribbon graphs with the expanding differential. We need to define these terms of course, which are analogues of \autoref{def:reducedtree} and \autoref{def:orientedtree}.

\begin{definition}
A \emph{reduced} M\"obius or ribbon graph\footnote{As for trees we use the word `reduced' as opposed to `stable' to emphasise that the vertices of these graphs are not equipped with a genus.} is a graph where each vertex has valence at least $3$. Given a graph with at least one vertex of valence at least $3$ we can associate (possibly several) reduced graphs to it by repeatedly contracting an edge attached to a vertex of valence $2$ until the graph is reduced. We say two reduced graphs are equivalent if they are obtained from the same graph in this manner. When we refer to a reduced graph we will mean an isomorphism class of reduced graphs up to this equivalence.
\end{definition}

\begin{remark}\label{rem:reducedmobgraph}
As for trees this equivalence on reduced graphs has no effect for ribbon graphs. However for stable M\"obius graphs we have an additional relation changing the colours on half edges belonging to the same edge as in \autoref{rem:reducedmobius}. If two half edges in an edge are coloured by $0$ then we can replace them by half edges coloured by $1$ and get an equivalent reduced graph. If they are different colours we can swap the colours and get an equivalent reduced graph. It is clear relations of this form are the only ones arising from this equivalence relation.
\end{remark}

Let $G$ be a stable graph of genus $g$ with $n$ legs and $e$ edges. Let $\det(G)=\Det(k^{\edges(G)})\otimes\Det(H_1(|G|))$ be concentrated in degree $e+3-3g-n$.

\begin{proposition}\label{prop:orientedmobgraph}
There are isomorphisms of chain complexes
\[
\moddmass((g,n))\cong\bigoplus_{G\in\iso\Gamma((g,n))}\underline{\mass}((G))\otimes\det(G)
\]
where here $\underline{\mass}((G))=\bigotimes_v\underline{\mass}((g(v),\flags(v)))$ is defined by taking the tensor product over $K$ using the $(K,K^{\otimes \flags(v)})$--bimodule structures. The action of $S_n$ on the right permutes labels of $G$ twisted by the sign. The differential is the natural differential expanding vertices of valence greater than $3$.
\end{proposition}

\begin{proof}
It is easy to convince oneself this is true since both sides are related to the free modular operad generated by $\mass$. We just need to explain how the $\det(G)$ term arises. First we observe that as in \autoref{rem:decgraphs} the space $\underline{\mass}((G))$ can be identified with the space generated by reduced M\"obius graphs. Given a reduced oriented M\"obius tree $T$ with $\omega\in\det(T)$ representing the orientation of $T$ and a contraction $\xi_{ij}$ with $i<j$ we can glue the $i$--th and $j$--th legs of $T$ to obtain a reduced M\"obius graph $G$ with newly formed edge $e$. We direct the edge $e$ such that the $i$--th leg is outgoing and the $j$--th leg incoming. This gives an oriented cycle $c$ in $H_1(|G|)$, by using the canonical direction on the tree $T$. Therefore we map $\xi_{ij}(T)$ to $G\otimes\omega\wedge e\otimes c$. We then extend this map inductively by mapping $\xi_{kl}(G)$ with $k<l$ to the graph $G'$ obtained by gluing the $k$--th and $l$--th legs of $G$, orienting the new edge as before, which gives a new oriented cycle $c'$ so we take the element $\omega'\wedge c'\in\Det(H_1(|G'|))$ given $\omega'\in\Det(H_1(|G|)$. We must check of course that this is a well defined map. In particular we must check it is well defined for the various associativity and equivariance relations. We omit the details, however the main point to observe is that the minus sign arising when we apply the transposition $(ij)\in S_n$ to a reduced oriented M\"obius tree and then contract the $i$--th and $j$--th legs is reflected in the fact that the direction of the edge formed by gluing legs $i$ and $j$ is then reversed so the orientation of the cycle $c$ is reversed and also when we carry out contractions in a different order, we swap the ordering of the cycles, but we also swap the ordering of the new edges. 

It is completely clear that the gradings and the differentials coincide. To see this map is an isomorphism note it is clearly surjective then compare dimensions by observing that both sides are closely related to the free modular operad generated by $\mass$.
\end{proof}

The compositions are of course simply gluing graphs and ordering the edges in the same way as we do for oriented trees (cf \autoref{subsec:koszulduality}). Contractions are also obvious and induce the orientation as detailed in the above proof.

We can talk about oriented graphs as we did for trees in \autoref{def:orientedtree}.

\begin{definition}
The space of oriented ribbon/M\"obius graphs is generated by ribbon/M\"obius graphs $G$ equipped with an ordering of the internal edges and an ordering of a basis of cycles in $H_1(|G|)$ subject to the relations arising by requiring that swapping the order of two edges or of two cycles is the same as multiplying by $-1$. In particular the space $\underline{\mass}((G))\otimes\det(G)$ from above can be identified as the space of reduced oriented M\"obius graphs whose underlying graph is $G$.
\end{definition}

\begin{remark}\label{rem:orientation}
When $k=\mathbb{Q}$ or $k=\mathbb{R}$ an orientation on a graph is equivalent to an ordering of its vertices and directing its edges, up to an even permutation. For example see \cite{conantvogtmann,getzlerkapranov}.
\end{remark}

\subsection{M\"obiusisation of operads and closed KTFTs}
We briefly outline the general construction for operads that follows from considering the above arguments and we also briefly consider closed KTFTs. As usual we let $k$ be a field and $K$ be the unital associative algebra over $k$ generated by an involution $a$ so $K=\langle 1, a \rangle / (a^2=1)=k[\mathbb{Z}_2]$.

\begin{definition}
Let $\mathcal{P}\in\dgop(k)$ be an admissible dg operad so $\mathcal{P}(1)=k$ is concentrated in degree $0$. The \emph{M\"obiusisation} of $\mathcal{P}$ is an operad $\mob\mathcal{P}\in\dgop(K)$ obtained by freely adjoining an element $a$ to $\mathcal{P}(1)$ in degree $0$ and imposing the relations $da=0$, $a^2=1$ and $am = \tau_n(m)a^{\otimes n}$ (the reflection relation) for all $m\in\mathcal{P}(n)$ where $\tau_n=(1\quad n)(2\quad n-1)(3\quad n-2)\ldots\in S_n$ is the permutation reversing $n$ labels. This construction extends to a functor $\mob\co \dgop(k)\rightarrow\dgop(K)$.

Given a unital extended modular operad $\mathcal{O}$ with $\mathcal{O}((0,2))=k$ we define $\mob\mathcal{O}$ in a similar way.
\end{definition}

Note that $\mathcal{P}$ is a suboperad of $\mob\mathcal{P}$. Clearly $\mass$ as defined above is indeed the M\"obiusisation of $\ass$. We have the following properties that generalise those shown for $\mass$ in the previous section:

\begin{theorem}
Let $\mathcal{P}\in\dgop(k)$.
\begin{enumerate}
\item If $\mathcal{P}$ is quadratic then so is $\mob\mathcal{P}$ and $(\mob\mathcal{P})^!\cong\mob(\mathcal{P}^!)$
\item As dg vector spaces $C(\mob\mathcal{P})(n)=\bigoplus^{2^n}C(\mathcal{P})(n)$
\item $C(\mob\mathcal{P})=\mob C(\mathcal{P})$
\item If $\mathcal{P}$ is Koszul then $\mob\mathcal{P}$ is Koszul
\item If $\mathcal{P}$ is cyclic then $\modc{\mob\mathcal{P}}=\mob\modc{\mathcal{P}}$
\end{enumerate}
\end{theorem}

\begin{proof}[Sketch proof]
\Needspace*{4\baselineskip}\mbox{}
\begin{enumerate}
\item This is a general version of ideas in \autoref{prop:massqdual}. Let $\mathcal{P}=\mathcal{P}(k,E,R)$. Let $E'=K\otimes_k E\otimes _k K^{\otimes 2}$ and $\mob E=E'/I$ where $I$ is generated by the reflection relations $a\otimes m\otimes 1\otimes 1=1 \otimes \sigma(m)\otimes a\otimes a$. Then $\mob R$ is generated by $R\subset F(E)(3)\subset F(\mob E)(3)$ and $\mob\mathcal{P}=\mob\mathcal{P}(K,\mob E,\mob R)$. Given $\psi\in\Hom_k(E,k)$ we define $\psi'\in\Hom_K(E',K)$ by $\psi'(1\otimes m\otimes 1\otimes 1)=\psi(m)$, $\psi'(a\otimes m\otimes 1\otimes 1)=a\psi(m)$, $\psi'(1\otimes m\otimes a\otimes a)=a\psi(\sigma(m))$ and $\psi(1\otimes m\otimes a\otimes 1)=\psi(1\otimes m\otimes 1\otimes a)=0$ for $m\in E$ and $\sigma=(1\quad 2)$. This in turn gives a well defined element of $\Hom_K(\mob E, K)$. This map extends to an isomorphism of $K$--collections $\Psi\co M(E^{\vee})\rightarrow (\mob E)^{\vee}$ and the result follows since $\Psi(\mob(R^{\perp}))=(\mob R)^{\perp}$.
\item This is a general version of \autoref{lem:koszul}. The same simple inductive proof works in the general case. Let $T$ be a tree with $n$ inputs. It's enough to show that $\mob\mathcal{P}(T)^* \cong \mathcal{P}(T)^*\otimes_k K^{\otimes n}$. Write $T$ as $T=T''\circ_i T'$ for $T'$ a corolla with $n'$ inputs. Then by induction on the number of internal edges $\mob\mathcal{P}(T)^* \cong (\mathcal{P}(T'')^*\otimes_k K^{\otimes n-n'+1})\otimes_K (\mathcal{P}(T')^*\otimes_k K^{\otimes n'})\cong \mathcal{P}(T)^*\otimes_k K^{\otimes n}$.
\item This follows from the above result together with a similar inductive argument showing that the reflection relation does indeed hold for any $m\in C(\mathcal{P})(n)\subset C(\mob\mathcal{P})(n)$.
\item This follows from the above results (cf \autoref{cor:koszul}).
\item We observe that both operads are generated by their genus $0$ parts which coincide.
\qedhere
\end{enumerate}
\end{proof}

Finally we briefly consider the situation of closed KTFTs.

\begin{definition}
We define the $k$--linear extended modular operad $\cktft$ (partial closed Klein topological field theory) as follows:
\begin{itemize}
\item For $n,g\geq 0$ and $2g+n\geq 2$ the vector space $\cktft((g,n))$ is generated by diffeomorphism classes of surfaces with $m$ handles and $u$ crosscaps and $n$ boundary components with $m+u/2 = g$ with $n$ copies of the circle embedded into the boundary, labelled by $\{1,\ldots,n\}$ with an action of $S_n$ permuting the labels.
\item Composition and contraction is given by gluing boundary components.
\end{itemize}
\end{definition}

\begin{remark}
This is `partial' closed Klein topological field theory in the sense that $u$ must be divisible by $2$. Therefore this operad features just those surfaces obtained as the connected sum of tori and Klein bottles. Since the connected sum of $2$ Klein bottles is diffeomorphic to the sum of $1$ handle and $1$ Klein bottle we see that $m+u/2$ is well defined regardless of how we represent the topological type. 
\end{remark}

Note that a Klein bottle (with boundary) must have genus $1$ in the modular operad sense since it is obtained from self gluing a genus $0$ surface. Therefore in full closed KTFT a crosscap would necessarily have genus $\frac{1}{2}$.

\begin{theorem}
$\cktft\cong\modmcom$.
\end{theorem}

\begin{proof}[Idea of proof]
The operad $\modmcom$ can be described in terms of graphs by forgetting mention of cyclic orderings of half edges at the vertices in our definition of M\"obius graphs. By replacing vertices with spheres with holes and edges by cylinders we obtain surfaces corresponding to such graphs. Then the above proposition follows by a similar argument to the proof of \autoref{thm:oktftmass}.
\end{proof}

It should not be at all surprising that $\modmcom$ does not describe full closed KTFT since \autoref{prop:cktft} tells us closed KTFTs are not just commutative Frobenius algebras with involution but rather have additional structure and additional relations that do not arise from relations in genus $0$, unlike in the open case.

\section{Graph complexes and moduli spaces of Klein surfaces}
In this section we will obtain results that are analogues of results concerning the ribbon graph decomposition of moduli spaces of Riemann surfaces with boundary. In particular we will be following methods of Costello \cite{costello1,costello2} which relate the operad $\moddass$ to certain moduli spaces and show $\moddass$ governs open topological conformal field theory. For us the unoriented analogue of a Riemann surface is a Klein surface and M\"obius graphs serve the same role as ribbon graphs.

Klein surfaces are `unoriented Riemann surfaces' (or more correctly Riemann surfaces are oriented Klein surfaces) in the sense that they have a dianalytic structure instead of an analytic structure. Klein surfaces are equivalent to symmetric Riemann surfaces (Riemann surfaces with an antiholomorphic involution) without boundary. In fact it follows Klein surfaces are equivalent to projective real algebraic curves (see Alling and Greenleaf \cite{allinggreenleaf} or Natanzon \cite{natanzon}). Since we wish to use techniques from hyperbolic geometry we will be concerned with the analytic theory.

\subsection{Outline and informal discussion}
We will first provide an informal discussion outlining the content of this section since there are some (slightly tedious) technical issues arising from the need to consider nodal surfaces, which are more subtle for Klein surfaces than for Riemann surfaces and which can hide the more important general picture.

A Klein surface is the natural extension of a Riemann surface allowing unorientable surfaces. Klein surfaces have a dianalytic structure instead of an analytic structure. However, given a Klein surface we can construct a double cover (the complex double) for the surface which is a Riemann surface so that we can use much of the theory of Riemann surfaces to study Klein surfaces. Indeed it is actually the case that Klein surfaces are equivalent to symmetric Riemann surfaces (Riemann surfaces with an antiholomorphic involution) identifying the Klein surface with the quotient. This is the standard way of approaching Klein surfaces, where a Klein surface is then simply a pair $(X,\sigma)$ with $X$ a Riemann surface and $\sigma$ and antiholomorphic involution.

However we will want to consider surfaces with nodes. With our previous comment in mind the obvious way to approach this is to define a nodal Klein surface as a pair $(X,\sigma)$ where $X$ is now a nodal Riemann surface and $\sigma$ is an antiholomorphic involution. This is the standard approach used for example by Sepp\"al\"a \cite{seppala} to construct a compactification of the moduli space of Klein surfaces/symmetric Riemann surfaces. 

There is also another natural way to define a nodal Klein surface without passing to the complex double. A nodal Klein surface is then a surface with some nodal singularities and a dianalytic structure, including at the nodes.

Although one might expect these two notions to coincide they do not. The reason for this is that it is no longer possible to form a unique complex double of a nodal Klein surface in the latter sense. This can be understood by considering a `strangulated' M\"obius strip (see \autoref{fig:strangulatedmobius}). In the second definition the dianalytic structure about the node encodes the twist in the M\"obius strip. If we pass to the complex double of a M\"obius strip, which is a torus, then a node on a strangulated torus does not encode any form of twisting. Indeed if we take the quotient of such a torus by an antiholomorphic involution then there is not a well defined way of giving a dianalytic structure at the node in the quotient.

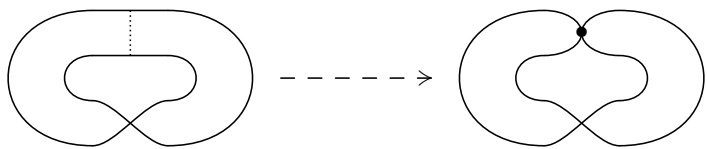
\begin{figure}[ht!]
\centering
\[
\begin{xy}
(-30,0)*\xybox{
(-12.5,0)*{\copair{}},
(0,6)*{\shortcylinder{}{}}="top",
c+U="a","top",c+D;"a"**@{.},
(0,-6)*{\flip{}{}},
(12.5,0)*{\pair{}}},
(30,0)*\xybox{
(-12.5,0)*{\copair{}},
(-5,6)*\xybox{\ellipse(5,3)_,=:a(180){-}},
(5,6)*\xybox{\ellipse(5,3)^,=:a(180){-}},
(0,6)*{\bullet},
(0,-6)*{\flip{}{}},
(12.5,0)*{\pair{}}},
(-10,0);(10,0)**\dir{--};?(1)*\dir2{>}
\end{xy}
\]
\caption{A strangulated M\"obius strip, obtained by contracting the dotted line to a node.}
\label{fig:strangulatedmobius}
\end{figure}

With this in mind it is natural to ask if there is another double cover that we can construct for this second type of nodal surface. The solution is given by the orienting double which is a Riemann surface but possibly with a boundary.

The important difference between the complex double and the orienting double is that the complex double takes the boundary of a Klein surface to the fixed points in the interior of a symmetric Riemann surface and, as mentioned above, if we take the quotient of a symmetric Riemann surface with an interior node fixed by the symmetry then there is not a well defined way of giving a dianalytic structure at the corresponding boundary node in the quotient. However the orienting double takes the boundary of a Klein surface to the boundary of a symmetric Riemann surface and so boundary nodes now correspond to boundary nodes. For example, the quotient of a strangulated torus (with an antiholomorphic involution such that the node in the quotient is a boundary node) could be given the dianalytic structure of either a strangulated M\"obius strip or a strangulated annulus. However the orienting doubles of these are respectively an annulus with two strangulated points (the covering map wraps such an annulus twice around the M\"obius strip) and a disjoint union of two strangulated annuli.

It is also natural to ask if there is another notion of a boundary node on a Klein surface that is equivalent to an interior node fixed by the symmetry on a symmetric Riemann surface. The answer to this is a na\"ive node, which is simply a singularity without a dianalytic structure at the node.

So we can still obtain equivalences for each of these types of nodal surface, although they are different. It is perhaps not entirely necessary to consider these equivalences of categories in too much detail in order to obtain our main results but we give a fairly detailed overview in order to make the subtlety arising from the different types of nodes clearer.

The conclusion of all this is that we obtain two different partial compactifications of the moduli space of Klein surfaces by allowing nodes on the boundary. It turns out the second type of nodal surface is the natural notion for defining Klein topological conformal field theory since it generalises the gluing of intervals discussed in the previous sections. We obtain a topological modular operad which we will denote $\kb$, the operad of Klein surfaces with boundary nodes. This notation reflects the notation $\nb$ used by Costello \cite{costello1,costello2} for the operad of Riemann surfaces with boundary and boundary nodes. The first type of nodal surface gives rise to an operad that is closer in spirit to the Deligne--Mumford operad since we are gluing symmetric Riemann surfaces without boundary at interior marked points. We obtain an operad which we will denote $\mb$, the operad of `admissible' symmetric nodal Riemann surfaces without boundary. This notation reflects the common notation for the space of symmetric Riemann surfaces and the fact that symmetric Riemann surfaces are equivalent to real algebraic curves. Note however that for us $\mb$ is \emph{not} the full space of nodal symmetric Riemann surfaces (stable real algebraic curves) and should not be confused with the full compactification obtained by taking all ways of forming nodes (it is an open subspace of this).

We wish to apply the methods of Costello \cite{costello1,costello2} to find a graph decomposition of both of these operads. The operad $\kb$, being the operad of Klein topological conformal field theory, is similar to $\nb$ and most of the results concerning the latter have a corresponding version for the former. In particular $\moddmass$ (over $\mathbb{Q}$) is a chain model for the homology of $\kb$ (which is homotopy equivalent to the moduli space of smooth Klein surfaces) just as $\moddass$ is for $\nb$. It is from this that we obtain a M\"obius graph decomposition of moduli spaces of Klein surfaces which is a direct analogue to the ribbon graph decomposition of moduli spaces of Riemann surfaces.

In the second case we find that $\moddmass/(a=1)$ (where $a\in\dmass(1)=\mathbb{Q}[\mathbb{Z}_2]$ is the involution) is a chain model for the homology of $\mb$. This gives a `dianalytic ribbon graph' decomposition of the partial compactification $\mb$. This partial compactification is quite different from the other. For example unlike $\moddmass$ the quotient has non-trivial homology in genus $0$.

Later we give a concrete explanation of the graph complexes obtained for each of these spaces and the corresponding isomorphisms on homology without using the language of operads.

We will finish this outline with a few words about the proof of these results. It is important to note that the proof of the results by Costello \cite{costello2} transfers easily to our situation and as such we reference \cite{costello2} heavily. This is not that surprising since we have already mentioned that we can form the orienting double, which is a Riemann surface without boundary, which are the objects considered in \cite{costello1,costello2}. In addition hyperbolic geometry features heavily in the proof and the same methods apply directly to Klein surfaces (which again can be seen by considering an appropriate double).

We begin by reviewing the necessary definitions and theory of Klein surfaces following Alling and Greenleaf \cite{allinggreenleaf} and Liu \cite{liu}, with some modifications.

\subsection{Klein surfaces and symmetric Riemann surfaces}
Let $D$ be a non-empty open subset of $\mathbb{C}$ and $f\co D\rightarrow\mathbb{C}$ be a smooth map. Recall $f$ is analytic on $D$ if $\frac{\partial f}{\partial \bar{z}}=0$ and anti-analytic if $\frac{\partial f}{\partial z}=0$. We say $f$ is \emph{dianalytic} if its restriction to each component of $D$ is either analytic or anti-analytic. If $A$ and $B$ are any non-empty subsets of $\mathbb{C}^+$ (the upper half plane) we say a function $g\co A\rightarrow B$ is analytic (or anti-analytic) on $A$ if it extends to an analytic (respectively anti-analytic) function $g'\co U\rightarrow\mathbb{C}$ where $U$ is an open neighbourhood of $A$. Once again we call $g$ dianalytic if its restriction to each component of $A$ is either analytic or anti-analytic.

For us a surface is a compact and connected (unless otherwise stated) topological manifold of dimension $2$. Our surfaces can have a boundary. Recall that a smooth structure on a surface is determined by a smooth atlas (an atlas $\mathcal{A}$ such that all the transition functions of $\mathcal{A}$ are smooth) and similarly an analytic structure is given by an atlas such that all transition functions are analytic. A Riemann surface\footnote{When we use the term Riemann surface we are allowing surfaces possibly with non-empty boundary.} is a surface with an analytic structure and morphisms of Riemann surfaces are non-constant analytic maps (maps that are analytic on coordinate charts) that restrict to maps on the boundary. In order to bring our definitions closer to \cite{liu} we refer to Riemann surfaces with non-empty boundary as bordered Riemann surfaces. A Riemann surface is canonically oriented by its analytic structure.

\begin{definition}
An atlas $\mathcal{A}$ on a surface $K$ is dianalytic if all the transition functions of $\mathcal{A}$ are dianalytic. A \emph{dianalytic structure} on $K$ is a maximal dianalytic atlas. A \emph{Klein surface} is a surface equipped with a dianalytic structure.
\end{definition}

An analytic structure can be extended to a dianalytic structure and so a Riemann surface can be viewed as a Klein surface. In doing so we no longer have a canonical orientation. Klein surfaces in general need not be orientable. It is shown in \cite{allinggreenleaf} that every compact surface can carry a dianalytic structure.

\begin{definition}
A morphism between Klein surfaces $K$ and $K'$ is a non-constant continuous map $f\co (K,\partial K)\rightarrow(K',\partial K')$ such that for all $x\in K$ there are charts $(U,\phi)$ and $(V,\psi)$ around $x$ and $f(x)$ respectively and an analytic function $F\co \phi(U)\rightarrow\mathbb{C}$ such that the following diagram commutes:
\[\xymatrix{
U\ar[rr]^f\ar[d]^\phi & & V\ar[d]^\psi \\
\phi(U)\ar[r]^F & \mathbb{C}\ar[r]^\Phi & \mathbb{C}^+
}\]
Here $\Phi(x+iy)=x+i|y|$ and is called the \emph{folding map}. We call $f$ a \emph{dianalytic morphism} if we can choose charts so that $\Phi\circ F$ in the above diagram is dianalytic.
\end{definition}

\begin{remark}
Note that when we consider Riemann surfaces as Klein surfaces morphisms of Riemann surfaces can be thought of as morphisms of Klein surfaces. Note also that morphisms of Klein surfaces are not always dianalytic since we are allowing maps which `fold' along the boundary of $K'$. This is useful since, for example, it means that the category of Klein surfaces is the correct domain for the complex double (see Alling and Greenleaf \cite{allinggreenleaf}) and other quotients and that the category of Klein surfaces is equivalent to the category of symmetric Riemann surfaces without boundary (and then Klein surfaces are real algebraic curves, again see \cite{allinggreenleaf}). If $K,K'$ have no boundary then morphisms between them are dianalytic.
\end{remark}

A morphism $f$ is dianalytic if and only if $f^{-1}(\partial K')=\partial K$. The composition of two dianalytic morphisms is dianalytic.

\begin{definition}
A \emph{symmetric Riemann surface}\footnote{Note again that this is a slightly different definition to that which is normally found elsewhere since we are allowing our Riemann surfaces to have a boundary unless otherwise stated.} $(X,\sigma)$ is a Riemann surface with an antiholomorphic involution $\sigma\co X\rightarrow X$ (which of course restricts to the boundary if our surface is bordered).  A morphism $f\co (X,\sigma)\rightarrow(X',\sigma')$ is a morphism of Riemann surfaces such that $f\circ\sigma=\sigma'\circ f$. By convention we allow symmetric Riemann surfaces to be disconnected provided the quotient surface $X/\sigma$ is connected.
\end{definition}

Given a symmetric Riemann surface $(X,\sigma)$ the quotient surface $X/\sigma=K$ has a unique dianalytic structure such that the quotient map $q$ is a morphism of Klein surfaces, see \cite{allinggreenleaf}. Again $q^{-1}(\partial K)=\partial X$ if and only if $q$ is a dianalytic morphism of Klein surfaces. In this case we call $(X,\sigma)$ a \emph{dianalytic symmetric Riemann surface}. 

\begin{definition}
\Needspace*{4\baselineskip}\mbox{}
\begin{itemize}
\item The category $\klein$ has objects Klein surfaces with morphisms as defined above.
\item The category $\dklein$ has objects Klein surfaces with just the dianalytic morphisms.
\item The category $\symriem$ has objects symmetric Riemann surfaces without boundary and morphisms analytic maps as defined above.
\item The category $\dsymriem$ has objects dianalytic symmetric Riemann surfaces (possibly with boundary) and morphisms analytic maps as defined above.
\end{itemize}
\end{definition}

To understand the category of Klein surfaces better we recall the existence of the orienting double of a Klein surface.

\begin{lemma}\label{lem:odouble}
\Needspace*{4\baselineskip}\mbox{}
\begin{itemize}
\item Let $K$ be a Klein surface. Then there exists a Riemann surface $K_\mathbf{O}$ and a morphism $f\co K_\mathbf{O}\rightarrow K$ such that $f^{-1}(\partial K)=\partial K_\mathbf{O}$ (so $f$ is dianalytic) and $K_\mathbf{O}$ is universal with respect to this property. This means if $X$ is a Riemann surface and $h\co X\rightarrow K$ is a morphism with $h^{-1}(\partial K)=\partial X$ then there is a unique analytic morphism $g\co X\rightarrow K_\mathbf{O}$ such that $h=f\circ g$ (more succinctly $K_\mathbf{O}$ is the universal Riemann surface over $K$ in the category $\dklein$). In particular this means $K_\mathbf{O}$ is unique up to unique isomorphism. We call it the \emph{orienting double} of $K$.
\item The map $f\co K_\mathbf{O}\rightarrow K$ is a double cover.
\item $K_\mathbf{O}$ has an antiholomorphic involution $\sigma$ such that $f\circ \sigma=f$.
\item Any double cover $h\co X\rightarrow K$ admitting such an involution and satisfying the property $h^{-1}(\partial K)=\partial X$ is universal with respect to this property (and hence is uniquely isomorphic to $K_\mathbf{O}$ as a double cover).
\item The map $f$ is unramified, $\sigma$ is unique and $K_\mathbf{O}$ is disconnected if and only if $K$ is orientable.
\end{itemize}
\end{lemma}

\begin{proof}
This is just a slight rewording of Alling and Greenleaf \cite[Theorem 1.6.7]{allinggreenleaf}.
\end{proof}
 
If $K$ is a Klein surface then $(K_\mathbf{O},\sigma)$ is a dianalytic symmetric Riemann surface, bordered if and only if $\partial K \neq \emptyset$. Given a dianalytic morphism of Klein surfaces it lifts to a morphism of dianalytic symmetric Riemann surfaces. This defines a functor from  Klein surfaces with dianalytic morphisms to dianalytic symmetric Riemann surfaces. Given a dianalytic symmetric Riemann surface $(X,\sigma')$ then $(X,q)\cong((X/\sigma')_\mathbf{O},f)$. Since $f$ is unramified then maps of dianalytic symmetric Riemann surfaces give dianalytic maps of the underlying Klein surfaces. In particular we can deduce:

\begin{proposition}\label{prop:odoubleequiv}
There is an equivalence of categories $\dklein\rightarrow\dsymriem$ given by taking the orienting double.
\qed
\end{proposition}

Given a Klein surface $K$ we can also construct the \emph{complex double} $K_\mathbf{C}$ of a Klein surface $K$. The complex double $K_\mathbf{C}$ is a symmetric Riemann surface without boundary that is disconnected if and only if $K$ is orientable and has empty boundary. In particular for an orientable surface of genus $g$ with $h$ boundary components it is obtained by taking two copies of the surface with opposite orientations and gluing along the boundary in an orientation preserving manner to give a symmetric Riemann surface without boundary of genus $2g+h-1$, with the antiholomorphic involution switching the two copies. For an unorientable surface with $g$ handles, $u$ crosscaps and $h$ boundary components the complex double is a connected symmetric Riemann surface without boundary of genus $2g+h+u-1$ although the construction in this case is less simple to describe and we refer to Alling and Greenleaf \cite[Theorem 1.6.1]{allinggreenleaf} for full details. In particular, similar to the orienting double, the complex double can in fact be realised as the universal Riemann surface without boundary over $K$ in the category $\klein$. It is then not hard to follow the same process as above and show the well known result:

\begin{proposition}\label{prop:cdoubleequiv}
There is an equivalence of categories $\klein\rightarrow\symriem$ given by taking the complex double.
\qed
\end{proposition}

\begin{remark}
The categories $\klein$ and $\dklein$ have the same objects and will also have the same moduli spaces (which can be identified with those of symmetric Riemann surfaces without boundary by \autoref{prop:cdoubleequiv}). As mentioned in the outline of this section the difference becomes much more noticeable when we consider nodal Klein surfaces, which are then no longer equivalent to nodal symmetric Riemann surfaces without boundary. However nodal Klein surfaces (with just dianalytic morphisms) are still equivalent to nodal symmetric Riemann surfaces possibly with boundary. Therefore we will actually obtain two different partial compactifications of moduli spaces.
\end{remark}

Given a Klein or Riemann surface whose underlying surface has $g$ handles, $0\leq u\leq 2$ crosscaps and $h$ boundary components we define the topological type to be $(g,u,h)$.

\subsection{Klein surfaces and hyperbolic geometry}

Recall that a connected hyperbolic Riemann surface without boundary admits a unique complete hyperbolic metric. If $X$ is a bordered Riemann surface whose complex double $X_\mathbf{C}$ (which is connected without boundary) is hyperbolic, then the antiholomorphic involution on $X_\mathbf{C}$ is an isometry with respect to the unique hyperbolic metric on $X_\mathbf{C}$ and so we can construct a unique (up to conformal isometry) hyperbolic metric on $X$, compatible with the analytic structure, such that the boundary (which corresponds to the fixed points of the involution) is geodesic.

If $K$ is a Klein surface whose complex double is hyperbolic then we can repeat this construction by taking the unique complete hyperbolic metric on $K_\mathbf{C}$. Therefore $K$ has a unique (up to isometry) hyperbolic metric, compatible with the dianalytic structure, such that the boundary is geodesic. Dianalytic morphisms of Klein surfaces correspond to conformal maps on hyperbolic surfaces. Since our surfaces are now unoriented by conformal maps we mean maps which preserve angles (as opposed to oriented angles).

The only Klein surfaces with a non-hyperbolic complex double are those of topological type $(g,u,h)$ with $2g+h+u-2 \leq 0$, since the complex double has topological type $(2g+h+u-1,0,0)$. The only such surfaces with $h>0$ are the disc, the annulus and the M\"obius strip.

Let $K$ be a Klein surface. Then since $K_\mathbf{O}$ covers $K$ we can pull back the hyperbolic metric on $K$ such that the involution on $K_\mathbf{O}$ is an isometry. Analytic maps between orienting double covers correspond to conformal maps of double covers and the boundary of $K_\mathbf{O}$ is geodesic and this is the same metric inherited from $(K_\mathbf{O})_\mathbf{C}$.

Using the hyperbolic metric on a Klein surface $K$ we can adapt the methods outlined by Costello in \cite{costello1} and elaborated upon in \cite{costello2} to construct a deformation retract on the moduli space of Klein surfaces which we will do below.

\subsection{Nodal Klein surfaces with oriented marked points}
We will need to allow Riemann surfaces and Klein surfaces with certain nodes and marked points.

A singular topological surface $(X,N)$ is a Hausdorff space $X$ with a discrete set $N\subset X$ (the set of singularities) such that $X\setminus N$ is a topological surface. As usual such surfaces will be compact (so $N$ will be finite) and connected and may have boundary unless otherwise stated. The boundary of a singular surface is defined to be the boundary of $X\setminus N$.

\begin{definition}
Let $(X,N)$ be a singular surface. A \emph{boundary node} is a singularity $z\in N$ with a neighbourhood homeomorphic to a neighbourhood of $(0,0)\in\{(x,y)\in(\mathbb{C}^+)^2 : xy=0\}$ such that $z\mapsto (0,0)$. Similarly an \emph{interior node} is a singularity with a neighbourhood homeomorphic to $(0,0)\in\{(x,y)\in\mathbb{C}^2 : xy=0\}$. We set $N_1$ to be the set of interior nodes and $N_2$ to be the set of boundary nodes. If $X$ has only nodal singularities then an atlas on $X$ is given by charts on $X\setminus N$ together with charts at the nodes as described. We call a singular surface with only nodal singularities a \emph{nodal surface}.
\end{definition}

Let $B=\{(x,y)\in(\mathbb{C}^+)^2 : xy=0\}$ and $B^*=B\setminus (0,0)$. Let $I=\{(x,y)\in\mathbb{C}^2 : xy=0\}$ and $I^*=I\setminus (0,0)$. Regarding the smooth curves $B^*$ and $I^*$ as Riemann surfaces with boundary we have a notion of analytic and anti-analytic maps to or from subsets of $B^*$ and $I^*$. We say a map to or from a neighbourhood $U$ of $(0,0)\in B$ or $U'$ of $(0,0)\in I$ is analytic or anti-analytic if it is analytic or anti-analytic when restricted to $U\cap B^*$ or $U'\cap I^*$. Dianalytic maps are again maps which restrict to analytic or anti-analytic maps on each connected component. Note in particular that if $U$ or $U'$ is connected then dianalytic maps on $U$ or $U'$ are either analytic or anti-analytic everywhere (even though $U\cap B^*$ and $U\cap I^*$ are disconnected). We therefore have a notion of a transition function between two charts on a nodal surface being analytic or dianalytic.

\begin{definition}
A \emph{nodal Riemann surface} is a nodal surface $(X,N)$ together with a maximal analytic atlas. A \emph{nodal Klein surface} is a nodal surface $(K,N)$ together with a maximal dianalytic atlas. By an irreducible component of a nodal surface we mean a connected component of the surface obtained by pulling apart all the nodes. Note that this is different from a connected component of $K\setminus N$ since an irreducible component will include points that were formerly nodes.
\end{definition}

We will mostly be concerned with Klein surfaces having only boundary nodes. We will also need a second different notion of a boundary node on a Klein surface.

\begin{definition}
A \emph{na\"ive nodal Klein surface} is a nodal surface $(K,N)$ with only boundary nodes together with a maximal dianalytic atlas on each irreducible component.
\end{definition}

Note that this does differ from the previous notion in the sense that we no longer have a dianalytic structure around the boundary nodes. Indeed, on a neighbourhood of a boundary node there are charts on the intersection with each irreducible component, but not a chart on the whole neighbourhood.

A morphism of nodal Riemann surfaces is a non-constant continuous map that is analytic on the charts (including at nodes). We can also define morphisms easily for na\"ive nodal Klein surfaces: a morphism is given by a non-constant continuous map which takes irreducible components to irreducible components and such that the map induced on each irreducible component is a morphism of smooth Klein surfaces.

\begin{definition}
A \emph{nodal symmetric Riemann surface} $(X, \sigma)$ is a nodal Riemann surface with an antiholomorphic involution $\sigma\co X\rightarrow X$.
\end{definition}

We will now discuss quotients of nodal Riemann surfaces informally. Given a nodal symmetric Riemann surface $(X,\sigma)$ we can form the quotient $q\co X\rightarrow X/\sigma$. This is a topological nodal surface and each irreducible component has a canonical non-singular dianalytic structure. Let $n\in N_1$ be an interior node. Since $\sigma$ must take nodes to nodes, if $\sigma(n)\neq n$ then $q(n)$ will be an interior node and we can extend the dianalytic structure about $q(n)$ in a unique way such that $q$ is dianalytic on the charts about $n$ and $q(n)$. If $n$ is fixed by $\sigma$ then $q(n)$ will be either an interior point or a boundary node. In the second case there is not a canonical way of choosing a dianalytic structure about $q(n)$ and so $q(n)$ is a na\"ive boundary node. Let $n'\in N_2$ be a boundary node. Similarly $q(n')$ will be either a boundary node or, if $n'$ is fixed by $\sigma$, a boundary point. In either case we can choose a dianalytic structure about $q(n')$ in a canonical way. Since we are interested mainly in Klein surfaces with only boundary nodes we make the following definition:

\begin{definition}
An \emph{admissible symmetric Riemann surface} $(X,\sigma)$ is a nodal symmetric Riemann surface $(X,N)$ such that $q(n)$ is a boundary node (a na\"ive boundary node if $n$ is an interior node) for all nodes $n\in N$.
\end{definition}

It is now not too difficult to work out what a morphism of nodal Klein surfaces should be, allowing folding maps along nodes as just described. Then the above informal discussion can be made precise by saying that there is a unique structure of a nodal Klein surface (possibly with some na\"ive nodes) on the quotient of a nodal Riemann surface such that the quotient map is a morphism of Klein surfaces. We will not give the details however since we only really need to consider dianalytic morphisms here. A dianalytic morphism $f\co K\rightarrow K'$ of nodal Klein surfaces is a non-constant continuous map that is dianalytic on all the charts (including at nodes). In particular such a map induces dianalytic maps on the irreducible components. A dianalytic nodal symmetric Riemann surface is an admissible symmetric Riemann surface such that the quotient map is dianalytic. In particular such surfaces have only boundary nodes.

From now on all our surfaces may be nodal unless otherwise stated. We need nodal surfaces with marked points.

\begin{definition}
A \emph{Klein/Riemann surface with $n$ marked points} $(X,\mathbf{p})$ is a nodal Klein/Riemann surface $(X,N)$ equipped with an ordered $n$--tuple $\mathbf{p}=(p_1,\ldots,p_n)$ of distinct points on $X\setminus N$. A morphism $f\co(X,\mathbf{p})\rightarrow(X',\mathbf{p'})$ of surfaces with $n$ marked points is a morphism of the underlying surface such that $f(p_i)=p'_i$.
\end{definition}

\begin{definition}
\Needspace*{4\baselineskip}\mbox{}
\begin{itemize}
\item A \emph{symmetric Riemann surface $X$ with $(m,n)$ marked points} $(X,\sigma,\mathbf{p},\mathbf{p'})$ is a nodal symmetric Riemann surface $(X,\sigma)$ with an ordered $2m$--tuple $\mathbf{p}=(p_1,\ldots,p_{2m})$ of distinct points on $X\setminus N$ such that $\sigma(p_i)=p_{m+i}$ for $i=1,\ldots,m$ and an ordered $n$--tuple $\mathbf{p'}=(p'_1,\ldots,p'_{n})$ of distinct points on $X\setminus N$ such that $\sigma(p'_j)=p'_j$ for $j=1,\ldots,n$.
\item A morphism $f\co (X,\sigma,\mathbf{p},\mathbf{p'})\rightarrow(Y,\tau,\mathbf{r},\mathbf{r'})$ is a morphism of the underlying symmetric Riemann surfaces such that $f(p_i)=r_i$ and $f(p'_j)=r'_j$.
\item We say a marked symmetric Riemann surface is \emph{admissible} if the underlying symmetric Riemann surface is and the points $q(p_i)$ and $q(p'_j)$ all lie in the boundary of the quotient. In this case all the $p_i$ must be on the boundary.
\end{itemize}
\end{definition}

Once again we can discuss quotients of marked symmetric Riemann surfaces. The quotient Klein surface is in a natural way a Klein surface with $m+n$ marked points (and if the surface is admissible all the marked points of the Klein surface lie on the boundary). In fact it has more structure: if $p_i$ is a marked point of a symmetric Riemann surface with $\sigma(p_i)=p_{m+i}\neq p_i$ then this gives locally an orientation about $q(p_i)$ induced from the chart about $p_i$. This motivates the following definition:

\begin{definition}
A \emph{Klein surface with $n$ oriented marked points} is a Klein surface with marked points $(K,\mathbf{p})$ equipped with a choice of orientation locally about each marked point (more precisely, a choice of one of the two germs of orientations on orientable neighbourhoods at each marked point).
\end{definition} 

Note finally that dianalytic marked symmetric Riemann surfaces (which are by definition admissible) can only have marked points on the boundary.

We are now ready to define our categories of interest. In particular we are interested in Klein surfaces with nodes and marked points all on the boundary. Equivalently this means we are also interested in admissible symmetric Riemann surfaces.

\begin{definition}
\Needspace*{4\baselineskip}\mbox{}
\begin{itemize}
\item The category $\nklein$ has objects Klein surfaces with only na\"ive boundary nodes and marked points (not oriented) on the boundary with morphisms as defined above for na\"ive nodal surfaces.
\item The category $\dnklein$ has objects Klein surfaces with only boundary nodes (but not na\"ive nodes) and oriented marked points on the boundary with dianalytic morphisms as defined above.
\item The category $\nsymriem$ has objects admissible symmetric Riemann surfaces with marked points and without boundary. Morphisms are analytic maps as defined above.
\item The category $\dnsymriem$ has objects dianalytic symmetric Riemann surfaces (possibly with boundary) with marked points. Morphisms are analytic maps as defined above.
\end{itemize}
\end{definition}

We can extend the notion of an orienting double to objects $(K,\mathbf{p})$ in $\dnklein$ by first constructing the orienting double on each irreducible component and gluing in the canonical way induced by the dianalytic structure at the nodes to obtain a dianalytic symmetric Riemann surface $K_\mathbf{O}$. If $f\co K_\mathbf{O}\rightarrow K$ is the covering map then $f^{-1}(p_i)$ gives two points in $K_\mathbf{O}$. To make $K_\mathbf{O}$ a marked surface we need to order these two points for each $i$. But we can do this using the local orientation about $p_i$ which allows us to canonically choose an $n$--tuple $\mathbf{q}=(q_1,\ldots,q_n)$ of distinct points on the boundary of $K_\mathbf{O}$ such that $f(q_i)=p_i$ and $f$ preserves the local orientations about $q_i$ and $p_i$. Then the orienting double of $K$ is defined as $(K_\mathbf{O},\sigma,\mathbf{p'},\mathbf{0})$ where $\mathbf{p'}=(q_1,\ldots,q_n,\sigma(q_1),\ldots,\sigma(q_n))$ and $\sigma$ is the antiholomorphic involution on $K_\mathbf{O}$. This is an object in $\dnsymriem$.

Given a Klein surface with marked (but not necessarily oriented) points choosing such an $n$--tuple $\mathbf{q}$ of points in $K_\mathbf{O}$ is clearly equivalent to providing local orientations at each $p_i$. Since this data can sometimes be easier to work with we will therefore also denote a Klein surface with $n$ oriented marked points by $(K,\mathbf{p},\mathbf{q})$.

\begin{remark}\label{rem:riemklein}
A Riemann surface with marked points can be thought of as a Klein surface with oriented marked points using the canonical orientation. If it is in $\dnklein$ then its orienting double is a disjoint union of two copies of itself  and so the canonical orientation means we choose points $q_i$ in the component that maps analytically under the quotient map.
\end{remark}

By showing nodal versions of the properties in \autoref{lem:odouble} it is not too difficult to obtain the following marked nodal analogue of \autoref{prop:odoubleequiv}:

\begin{proposition}\label{prop:odoubleequiv2}
There is an equivalence of categories $\dnklein\rightarrow\dnsymriem$ given by taking the orienting double.
\qed
\end{proposition}

Similarly given an object in $\nklein$ we can construct the complex double by first constructing the complex double on each irreducible component and then gluing to obtain an admissible symmetric Riemann surface $K_\mathbf{C}$ without boundary. If $f\co K_\mathbf{C}\rightarrow K$ is the covering map then $f^{-1}(p_i)$ is a single point in $K_\mathbf{C}$ so we obtain an object in $\nsymriem$. Once again we can show the marked nodal analogue of \autoref{prop:cdoubleequiv}:

\begin{proposition}\label{prop:cdoubleequiv2}
There is an equivalence of categories $\nklein\rightarrow\nsymriem$ given by taking the complex double.
\qed
\end{proposition}

\begin{remark}
We can still define the complex double for surfaces in $\dnklein$ however two Klein surfaces that are not isomorphic may have isomorphic complex doubles.
\end{remark}

Given a boundary node on a Riemann surface $X$ we can replace the node with a narrow oriented strip. We can also replace interior nodes with a narrow oriented cylinder. In this way we can obtain from $X$ a non-singular oriented topological surface.

We define the topological type of a Riemann surface $X$ with $n$ marked points as $(g,0,h,n)$ where $(g,0,h)$ is the topological type of the non-singular oriented surface obtained by the above process.

Given a Klein surface $K$ in $\dnklein$ with $n$ marked points we consider $K_\mathbf{O}$ as a Riemann surface by forgetting the symmetry and let $(\tilde{g},0,\tilde{h},2n)$ be its topological type. Then if $K_\mathbf{O}$ is disconnected the topological type of $K$ is defined as the topological type of one of the connected components of $K_\mathbf{O}$. If $K_\mathbf{O}$ is connected then the topological type of $K$ is defined as $(g,u,h,n)$ where $h=\frac{\tilde{h}}{2}$ and $g$ and $u$ are the unique solutions to $\tilde{g}+1=2g+u$ with $0<u\leq 2$.

For admissible symmetric Riemann surfaces without boundary in $\nsymriem$ we define their topological type as that of the underlying marked Riemann surface obtained by forgetting the symmetry.

The topological type of a symmetric Riemann surface in $\dnsymriem$ is defined as the topological type of its quotient Klein surface in $\dnklein$ and the topological type of a na\"ive nodal Klein surface in $\nklein$ is defined as the topological type of its complex double in $\nsymriem$.

\begin{definition}
A Klein or Riemann surface with $n$ (possibly oriented) marked points is \emph{stable} if it has only finitely many automorphisms.
\end{definition}

A non-singular Klein surface with $n$ (possibly oriented) marked points on the boundary (which we will assume is non-empty) is unstable precisely if it has the topological type of a disc and $n\leq 2$ or if it is an annulus with $n=0$ or a M\"obius strip with $n=0$ (since the orienting double of a M\"obius strip is an annulus). If the Klein surface has singularities (and so is in either $\nklein$ or $\dnklein$) then it is stable if and only if each connected component of its normalisation is. The normalisation is given by pulling apart all the nodes where each node gives two extra boundary marked points. It does not matter how we order these extra marked points.

\subsection{Moduli spaces of Klein surfaces}
In this section we discuss various moduli spaces and their relationships.

Let $\kb_{g,u,h,n}$ be the moduli space of stable Klein surfaces in $\dnklein$ with topological type $(g,u,h,n)$ and $h\geq 1$. Let $\ko_{g,u,h,n}\subset \kb_{g,u,h,n}$ be the subspace of non-singular Klein surfaces.

Due to \autoref{prop:odoubleequiv2} this can be identified with the moduli space of stable dianalytic symmetric Riemann surfaces in $\dnsymriem$. These moduli spaces are non-empty except for the cases when
\[(g,u,h,n)\in\{(0,0,1,0),(0,0,1,1),(0,0,1,2),(0,0,2,0),(0,1,1,0)\}.\]
There is an action of the group $\mathbb{Z}_2^{\times n}$ on $\kb_{g,u,h,n}$ given by flipping the orientations of marked points.

Let $\mb_{\tilde{g},n}$ be the moduli space of stable admissible symmetric Riemann surfaces in $\nsymriem$ with topological type $(\tilde{g},0,0,n)$. Let $\mo_{\tilde{g},n}\subset\mb_{\tilde{g},n}$ be the subspace of non-singular Riemann surfaces.

Due to \autoref{prop:cdoubleequiv2} this can be identified with the moduli space of stable Klein surfaces in $\nklein$. These moduli spaces are non-empty except for the cases when $(g,n)\in\{(0,0),(0,1),(0,2),(1,0)\}$. We observe that:
\[\left ( \coprod_{2g+h+u-1=\tilde{g}}\ko_{g,u,h,n}\right )\Biggr/\mathbb{Z}_2^{\times n} \quad\cong\quad \mo_{\tilde{g},n}\]

\begin{remark}
The slight abuse of notation here is potentially misleading. The full compactification of stable symmetric Riemann surfaces (the space of stable real algebraic curves) allowing all nodal Riemann surfaces without boundary, is very often denoted by $\mb$ (for example as in \cite{seppala,liu}). For us however it is an open subspace of this: the subspace of \emph{admissible} surfaces. Here neither $\mb_{\tilde{g},n}$ or $\kb_{g,u,h,n}$ are compact in general.
\end{remark}

Let $\nb_{g,h,n}$ be the moduli space of stable bordered Riemann surfaces with only boundary nodes and marked points on the boundary with topological type $(g,0,h,n)$. Let $\no_{g,h,n}\subset \nb_{g,h,n}$ be the subspace of non-singular bordered Riemann surfaces.

These are the moduli spaces considered by Costello \cite{costello1,costello2}. These spaces are non-empty except for the cases when
\[(g,h,n)\in\{(0,1,0),(0,1,1),(0,1,2),(0,2,0)\}.\]
By \autoref{rem:riemklein} we have a map $\nb_{g,h,n}\rightarrow\kb_{g,0,h,n}$, injective for $n>0$.

We now outline the construction of these spaces and their topology. We will follow Liu \cite{liu} closely and more details may be found there. The basic idea is to use the symmetric pants decomposition of the complex double to obtain Fenchel--Nielsen coordinates. In fact this will give an orbifold structure.

Recall that any stable Riemann surface $X$ of topological type $(\tilde{g},0,0,n)$ admits a \emph{pants decomposition}. More precisely there are $3\tilde{g} - 3 + n$ disjoint curves $\alpha_i$ on the surface which is obtained from puncturing $X$ at each marked point, with each curve being either a closed geodesic (in the hyperbolic metric) or a node, decomposing $X$ into a disjoint union of $2\tilde{g} - 2 + n$ pairs of pants whose boundary components are either one of the $\alpha_i$ or a puncture corresponding to a marked point. Furthermore if $(X,\sigma)$ is a stable symmetric Riemann surface then there exists a \emph{symmetric pants decomposition} that is invariant with respect to $\sigma$ (see Buser and Sepp\"al\"a \cite{buserseppala}). By this we mean that $\sigma$ induces a permutation on decomposing pairs of pants. 

We call a pants decomposition oriented if the pairs of pants are ordered and the boundary components of each pair of pants are ordered and each have a basepoint and each decomposing curve is oriented. This induces an ordering on the decomposing curves and so defines \emph{Fenchel--Nielsen coordinates}
\[(l_1,\ldots,l_{3\tilde{g}-3+n},\theta_1,\ldots,\theta_{3\tilde{g}-3+n})\]
where the $l_i$ are the lengths (in the hyperbolic metric) of the decomposing curves, and the $\theta_i$ are the angles between the basepoints of the two boundary components corresponding to the $i$--th decomposing curve. More precisely the ordering of the pairs of pants determines an ordering of the two basepoints on each decomposing curve and we set $\theta_i=2\pi\frac{\tau_i}{l_i}$ where $\tau_i$ is the distance one travels from the first basepoint to the second basepoint on the $i$--th decomposing curve in the direction that the curve is oriented. Note that we have $l_i\geq 0$ and $0\leq \theta_i < 2\pi$.

In the case that the pants decomposition is symmetric we may assume that the orientation of the pants decomposition has been chosen so that the symmetry permutes basepoints and reverses the orientation of decomposing curves which are not completely fixed by the symmetry. Note that in this case not all the coordinates are independent. In particular if a decomposing curve $\alpha_i$ is mapped to itself by the symmetry, then $\theta_i = 0$ or $\theta_i=\pi$. Similarly if the decomposing curves $\alpha_i$ and $\alpha_j$ are permuted by the symmetry then $l_i=l_j$ and $\theta_i=2\pi-\theta_j$.

\begin{definition}
A \emph{strong deformation} from a stable symmetric Riemann surface $(X',N',\sigma')$ to a stable symmetric Riemann surface $(X,N,\sigma)$ both of topological type $(\tilde{g},0,\tilde{h},n)$ is a continuous map $\kappa\co (X',N',\sigma')\rightarrow (X,N,\sigma)$ such that
\begin{itemize}
\item $\kappa$ takes boundary components to boundary components, interior nodes to interior nodes, boundary nodes to boundary nodes, preserves marked points and $\kappa\circ \sigma' = \sigma \circ\kappa$
\item for each interior node $n$ we have that $\kappa^{-1}(n)$ is either an interior node or an embedded circle in a connected component of $X'\setminus N'$
\item for each boundary node $n$ we have that $\kappa^{-1}(n)$ is either a boundary node or an embedded arc in a connected component of $X'\setminus N'$ with ends in $\partial X'$
\item $\kappa$ restricts to a diffeomorphism $\kappa\co\kappa^{-1}(X\setminus N)\rightarrow X\setminus N$.
\end{itemize}
\end{definition}

Given $X$ and $X'$ stable symmetric Riemann surfaces without boundary with oriented symmetric pants decompositions having decomposing curves $\alpha_i$ and $\alpha_i'$ respectively and a strong deformation $\kappa\co X'\rightarrow X$ we say that $\kappa$ is compatible with the oriented pants decompositions if $\kappa(\alpha_i')=\alpha_i$ and all the orientation data (the ordering of the pants, the ordering of the boundary components, the basepoints of the boundary components and the orientation of each $\alpha_i$) is preserved under $\kappa$. 

Given an oriented symmetric pants decomposition of $X$ with decomposing curves $\alpha_i$ and any strong deformation $\kappa\co X'\rightarrow X$, by choosing as decomposing curves the closed geodesics homotopic to each of the $\kappa^{-1}(\alpha_i)$ we can obtain a symmetric pants decomposition of $X'$ with decomposing curves $\alpha_i'$. Further there exists another strong deformation $\kappa'$ such that $\kappa'(\alpha_i')=\alpha_i$ and also an orientation of the pants decomposition of $X'$ so that it is pulled back from the oriented pants decomposition of $X$ under $\kappa'$. In particular $\kappa'$ is compatible with the oriented pants decompositions.

Recall that the complex structure on a pair of pants is determined, up to equivalence, by the lengths of the three boundary curves in the unique hyperbolic metric where the boundary curves are geodesic. Recall also that gluing pairs of pants along boundary components of common length is determined completely by the angle between two basepoints on the boundary components. Therefore if $X$ and $X'$ have the same Fenchel--Nielsen coordinates with respect to the pants decompositions preserved by $\kappa'$ then they are in fact biholomorphic.

\begin{definition}
A strong deformation from a stable Klein surface $(K',N')$ in $\dnklein$ to a stable Klein surface $(K,N)$ of topological type $(g,u,h,n)$ is a strong deformation between the orienting doubles.
\end{definition}

Note that a strong deformation of stable Klein surfaces induces a strong deformation on the complex doubles of the underlying na\"ive nodal Klein surfaces. Of course the converse is not true.

We are now ready to describe the topology on $\mb_{\tilde{g},n}$ and $\kb_{g,u,h,n}$.

Given a surface $X\in\mb_{\tilde{g},n}$ and an oriented symmetric pants decomposition of $X$ with coordinates $l_i,\theta_j$ denote by $U(X,\epsilon,\delta)$ the set of surfaces $X'$ with an oriented symmetric pants decomposition having coordinates $l_i',\theta_j'$ and admitting a strong deformation $\kappa\co X'\rightarrow X$ compatible with the pants decompositions such that $|l_i'-l_i|<\epsilon$ and $|\theta_j'-\theta_j|<\delta$. The collection $\{U(X,\epsilon,\delta) : X\in\mb_{\tilde{g},n}, \epsilon > 0,\delta > 0 \}$ then generates the topology on $\mb_{\tilde{g},n}$.

Set $z_j=l_j e^{i\theta_j}$ and let $\tilde{U}$ be the fixed locus under the symmetry of $X$, which, up to permutation of coordinates, consists of points of the form
\[(z_1,\bar{z}_1,z_2,\bar{z}_2,\ldots,\bar{z}_{d_1},x_1,\ldots,x_{d_2})\]
where $2d_1+d_2=3\tilde{g}-3+n$, $z_i\in \mathbb{C}$ and $x_i\in\mathbb{R}$. This is an open subset of $\mathbb{R}^d$ where $d=3\tilde{g}-3+n$. In particular the open sets $U(X,\epsilon,\delta)$ are homeomorphic to $\tilde{U}/\Gamma$ for an appropriate open subset $\tilde{U}$ of $\mathbb{R}^d$ and $\Gamma$ the automorphism group of $X$. Therefore the space $\mb_{\tilde{g},n}$ is an orbifold.

Similarly given a surface $K\in\kb_{g,u,h,n}$ and an oriented symmetric pants decomposition of $K_\mathbf{C}$ with coordinates $l_i,\theta_j$ denote by $U(K,\epsilon,\delta)$ the set of surfaces $K'$ with an oriented symmetric pants decomposition of $K_\mathbf{C}'$ having coordinates $l_i',\theta_j'$ and admitting a strong deformation $\kappa\co K_\mathbf{O}'\rightarrow K_\mathbf{O}$ compatible with the pants decompositions on the complex doubles such that $|l_i'-l_i|<\epsilon$ and $|\theta_j'-\theta_j|<\delta$. The collection $\{U(K,\epsilon,\delta) : K\in\kb_{g,u,h,n}, \epsilon > 0,\delta > 0 \}$ then generates the topology on $\kb_{g,u,h,n}$.

In this case the open sets are now homeomorphic to $\tilde{U}/\Gamma$ for some open neighbourhood $\tilde{U}$ of $\mathbb{R}^m_{\geq 0}\times \mathbb{R}^{d-m}$ and $\Gamma$ the automorphism group of $K$, where $d=6g+3h+3u+n-6$ and $m$ is the number of nodes of $K$. The value of $d$ is obtained by noting that the complex double of $K$ has topological type $(2g+h+u-1,0,0,n)$, so that the number of Fenchel--Nielsen coordinates of the double is $6(2g+h+u-1)-6+2n$ and again only half of these are independent. However, as discussed previously, the interior nodes of the complex double do not encode the dianalytic structure of the boundary nodes in the quotient, nor are the marked points oriented. As a result there is only one way to smooth a node in $K$ given the dianalytic structure so the coordinates corresponding to (smoothings of) the $m$ nodes lie in $\mathbb{R}^m_{\geq 0}$. In terms of the coordinates on the complex double this corresponds to the fact that a node in the complex double is fixed by the symmetry so as we smooth it the Fenchel--Nielsen coordinate corresponding to the gluing angle is either $0$ or $\pi$. However only one choice is possible if the quotient is to have the correct topological type and orientation of marked points. Therefore the space $\kb_{g,u,h,n}$ is an orbifold with corners. Furthermore an orbifold with corners is homotopy equivalent to its interior which in this case is $\ko_{g,u,h,n}$.

We can also carry out a similar construction for the spaces $\nb_{g,h,n}$. See also \cite{liu,costello1}.

Our discussion can be summarised in the following lemma:

\begin{lemma}
\Needspace*{4\baselineskip}\mbox{}
\begin{itemize}
\item $\nb_{g,h,n}$ is an orbifold with corners of dimension $6g-6+3h+n$. The interior is $\no_{g,h,n}$ and the inclusion $\no_{g,h,n}\hookrightarrow\nb_{g,h,n}$ is then a homotopy equivalence.
\item $\kb_{g,u,h,n}$ is an orbifold with corners of dimension $6g+3h+3u+n-6$. The interior is $\ko_{g,u,h,n}$ and the inclusion $\ko_{g,u,h,n}\hookrightarrow\kb_{g,u,h,n}$  is then a homotopy equivalence.
\item $\mb_{\tilde{g},n}$ is an orbifold of dimension $3\tilde{g}-3+n$.
\end{itemize}
\end{lemma}

\autoref{rem:riemklein} can be taken further. Given a Riemann surface with $n$ marked points together with a colouring of the marked points by $\{0,1\}$ we can map it to a Klein surface with $n$ oriented marked points in $\dnklein$ by choosing the canonical orientation about points coloured by $0$ and the opposite orientation otherwise. Equivalently we choose $q_i$ in the component of the orienting double that maps analytically under the quotient map when $i$ is such that $p_i$ is coloured by $0$ and in the component which maps anti-analytically otherwise.

Two isomorphism classes of Riemann surfaces with $n$ coloured marked points map to the same class of Klein surface precisely when there is an antiholomorphic map between them that reverses the colourings of all the marked points and this map is therefore $2$--to--$1$ for $n>0$. For $n=0$ this map is then $2$--to--$1$ on isomorphism classes except when a Klein surface has isomorphic underlying analytic structures (or equivalently when considering a Riemann surface as a Klein surface its automorphism group is no larger, or equivalently the Riemann surface admits an antiholomorphic automorphism).

In particular for $n>0$ if we restrict to Riemann surfaces where the first marked point is coloured by $0$ then this map is injective.

The following lemma follows from this discussion.

\begin{lemma}\label{lem:riemklein}
There is an isomorphism
\[\left (\bigoplus^{2^{n}}\nb_{g,h,n}\right )\biggr/\sim\quad\longrightarrow\quad \kb_{g,0,h,n}\]
where $\sim$ identifies Riemann surfaces with coloured markings that give rise to the same Klein surfaces with oriented markings. For $n>0$ the left hand side is isomorphic to $\bigoplus^{2^{n-1}}\nb_{g,h,n}$.
\end{lemma}

Let $D_{g,u,h,n}\subset\kb_{g,u,h,n}$ be the subspace consisting of those Klein surfaces whose irreducible components are all discs. Let $D_{g,h,n}\subset\nb_{g,h,n}$ be the corresponding subspace of bordered Riemann surfaces. Let $\dr_{\tilde{g},n}\subset\mb_{\tilde{g},n}$ be the subspace consisting of those admissible Riemann surfaces whose irreducible components are all spheres. Note that when we consider $\mb_{\tilde{g},n}$ as a moduli space of na\"ive nodal Klein surfaces then $\dr_{\tilde{g},n}$ is the subspace consisting of those Klein surfaces whose irreducible components are all discs.

\subsection{The open KTCFT operad and related operads}

We recall (see \cite{costello1}) that the spaces $\nb_{g,h,n}$ form a modular operad $\nb$ controlling open topological conformal field theory (TCFT). The spaces $D_{g,h,n}$ form a suboperad. Further it was shown by Costello \cite{costello1,costello2} that these spaces are compact orbispaces and admit a decomposition into orbi-cells and if $n>0$ then $D_{g,h,n}$ is an ordinary space instead of an orbispace so we obtain a cell decomposition. Further these orbi-cells are labelled by reduced ribbon graphs.

The collection of spaces $\kb_{g,u,h,n}$ form a topological modular operad $\kb$ by gluing Klein surfaces with oriented marked points such that the orientations are compatible. We can describe the gluing explicitly via the orienting double. Given dianalytic symmetric Riemann surfaces $(X,\sigma,\mathbf{p})$ and $(X',\sigma',\mathbf{p'})$ with $n$ and $m$ marked points respectively we can define an operation gluing along marked points $p_i$ and $q_j$ as follows: we glue the underlying Riemann surfaces at these points and we also glue the points $\sigma(p_i)$ and $\sigma'(q_i)$. Clearly we can use $\sigma$ and $\sigma'$ to define an antiholomorphic involution on the resulting surface which will clearly be dianalytic and will have $n+m-2$ marked points. We also note that the topological type of the resulting surface is the sum of the topological types of $X$ and $X'$. Similarly we can define contractions/self gluings of dianalytic symmetric Riemann surfaces in this way, in which case the resulting topological type either increases the number of boundary components or the number of crosscaps by $1$.

Therefore the space $\kb((\tilde{g},n))$ is the disjoint union of the spaces $\kb_{g,u,h,n}$ with $\tilde{g}=2g+h+u-1$. The group $S_n$ acts by reordering the $n$--tuple of marked points. Composition and contraction is given by gluing the marked point as described above. This gives us the modular operad controlling open Klein topological conformal\footnote{The use of the word `conformal' here is potentially confusing since dianalytic maps correspond to maps preserving angles but not necessarily oriented angles. A conformal map in the presence of the word `Klein' should therefore be understood in this sense.} field theory (KTCFT). The spaces $D_{g,u,h,n}$ form a suboperad which we denote $D$.

We will think of these two operads as extended modular operads by setting $D((0,2))=\kb((0,2))$ to be the discrete group $\mathbb{Z}_2$ which acts on Klein surfaces by switching the orientation of marked points.

Similarly by gluing admissible symmetric Riemann surfaces without boundary at marked points in the natural way the spaces $\mb_{\tilde{g},n}$ form a modular operad $\mb$ and the spaces $\dr_{\tilde{g},n}$ form a suboperad which we denote $\dr$.

\begin{remark}
The gluings of the operad $\mb$ when thought of as gluings of admissible symmetric Riemann surfaces are `closed string' gluings. In this way the operad $\mb$ is closer in spirit to the Deligne--Mumford operad (see, for example, Getzler and Kapranov \cite{getzlerkapranov}).
\end{remark}

\begin{proposition}\label{prop:celldecomp}
The spaces $D_{g,u,h,n}$ admit a decomposition into orbi-cells labelled by reduced M\"obius graphs.
\end{proposition}

\begin{proof}
This is not hard. It is intuitively obvious how we can label a surface by a M\"obius graph but in order to specify the colouring of the edges we need to understand the dianalytic structure about the nodes. It is easiest (although unenlightening) to do this via the orienting double. Let $\mob\Gamma_{g,u,h,n}$ denote the set of reduced M\"obius graphs of topological type $(g,u,h)$ with $n$ legs. Let $(K,\mathbf{p},\mathbf{q})\in D_{g,u,h,n}$. We associate a graph $\gamma(K)\in\mob\Gamma_{g,u,h,n}$ to $K$ as follows: There is one vertex for each irreducible component of $K$, an edge for each node and a leg for each marked point. This yields a graph. We need to specify a ribbon structure and a colouring of the half edges and verify we can do this in a well defined manner. We consider the orienting double $(K_\mathbf{O},f,\sigma)$ and for each irreducible component $A$ of $K$ we choose an irreducible component $\hat{A}$ of $f^{-1}(A)\subset K_\mathbf{O}$. Since $\hat{A}$ is an oriented disc this gives a natural cyclic order on the half edges of $\gamma(K)$. We colour the leg corresponding to $p_i$ by $0$ if $q_i\in\hat{A}$ and by $1$ otherwise. A node $x$ of $K$ lies on the boundary of either $1$ or $2$ irreducible components. If it lies on only $1$ component, $B$ say, then we colour the edge associated to it by $0$ if a preimage of $x$ in $K_\mathbf{O}$ lies only on $\hat{B}$, else we colour the half edges by different colours (by \autoref{rem:reducedmobgraph} it does not matter how we do this). If $x$ lies on the boundary of both $B$ and $C$, then we colour the edge associated to it by $0$ if $\hat{B}$ and $\hat{C}$ intersect at a preimage of $x$ and by different colours otherwise.

This yields a reduced M\"obius graph. We must show it is well defined since we made a choice of irreducible components in $K_\mathbf{O}$. Given an irreducible component $A$ of $K$, if we had chosen the other preimage of $A$ then the cyclic ordering at the corresponding vertex would be reversed and the colouring also reversed. Thus the resulting M\"obius graphs would be isomorphic.

We now show that for any graph $G\in\mob\Gamma_{g,u,h,n}$ the space of surfaces $K\in D_{g,u,h,n}$ with $\gamma(K)=G$ is an orbi-cell. This follows from the topology of the moduli space of discs: let $D$ be a disc with an analytic structure. Then $D$ is holomorphic to the unit disc in the complex plane which has automorphism group $PSL_2(\mathbb{R})$ and so the space of $n\geq 3$ marked points on the unit disc is the configuration space of marked points on $S^1$ modulo $PSL_2(\mathbb{R})$. Further, automorphisms of the unit disc preserve the cyclic ordering of marked points and so this space decomposes into cells labelled by ribbon corollas. As noted in \autoref{lem:riemklein} the moduli space of marked Klein discs can be identified with the moduli space of coloured marked unit discs modulo the action of the anti-analytic map reversing the cyclic ordering of the marked points. The space of coloured marked unit discs decomposes into cells and the action reversing the cyclic ordering freely maps cells to cells and so we have a cell decomposition of the moduli space of marked Klein discs. Clearly each cell is labelled by a different M\"obius corolla. Therefore to each vertex $v$ of a M\"obius graph we associate a cell $X(v)$. Then we can let $X(G)=\prod_v X(v)$. We can identify the orbispace of surfaces with $\gamma(K)=G$ as $X(G)/\operatorname{Aut}(G)$, which is an orbi-cell.
\end{proof}

\begin{remark}\label{rem:cellattach}
The orbi-cell labelled by a graph $G$ is attached to the the cells labelled by the graphs given by all ways of expanding the vertices of $G$ of valence greater than $3$, since two marked points meeting corresponds to bubbling off a disc. This should of course remind us of the differential of the operad $\moddmass$.
\end{remark}

\begin{lemma}
A stable Klein surface with $n>0$ oriented marked points has no non-trivial automorphisms.
\end{lemma}

\begin{proof}
We must show that the orienting double has no non-trivial automorphisms. If the orienting double is disconnected then an automorphism would necessarily restrict to an automorphism of one of the connected components. The result is true for stable bordered Riemann surfaces \cite[Lemma 3.0.11]{costello2} so the orienting double has no non-trivial automorphisms.
\end{proof}

\begin{corollary}
If $n>0$ then $D_{g,u,h,n}$ is an ordinary space and decomposes into a cell complex.
\qed
\end{corollary}

\begin{proposition}\label{prop:dianalyticgraphs}
The spaces $\dr_{\tilde{g},n}$ admit a decomposition into orbi-cells labelled by a certain type of reduced graph (which we call a dianalytic ribbon graph).
\end{proposition}

\begin{proof}
Each orbi-cell will be labelled by a (reduced) ribbon graph where two ribbon graphs are considered equivalent if there is an isomorphism of the underlying graphs that at each vertex either preserves or reverses the cyclic ordering (note that this definition can be thought of as M\"obius graphs without a colouring). We will call such graphs up to this equivalence \emph{dianalytic ribbon graphs}. Given a surface $K\in \dr_{\tilde{g},n}$ we first choose an orientation for each irreducible disc. We can then associate a ribbon graph to it where there is a vertex for each irreducible component of $K$, an edge for each node and a leg for each marked point. Again this is well defined since choosing a different orientation of an irreducible component reverses the cyclic ordering at the vertex associated to that component.

It remains to show that the space of surfaces corresponding to a given graph $G$ is an orbi-cell. This follows from the fact that the moduli space of stable dianalytic discs with $n$ marked points can be identified with the moduli space of marked oriented discs modulo the action of the map reversing the orientation. This action freely maps cells to cells (since there are at least $3$ marked points on a disc so reversing the orientation gives a different cell) so we get a cell decomposition of the moduli space of stable dianalytic discs with cells labelled by dianalytic ribbon corollas. We can associate to each vertex $v$ of $G$ a cell $X(v)$ and once again the orbi-cell of surfaces corresponding to the graph is $X(G)/\operatorname{Aut}(G)$ where $X(G)=\prod_vX(v)$.
\end{proof}

\begin{remark}\label{rem:notgrpquotient}
Note that although each orbi-cell of $\dr_{\tilde{g},n}$ is the quotient of orbi-cells in $\coprod D_{g,u,h,n}$ (where the disjoint union is as usual taken over surfaces such that $2g+u+h-1=\tilde{g}$) by the action of a finite group switching the colourings of half edges, this does not extend to a global action and so the space $\dr_{\tilde{g},n}$ is not obtained as a quotient of $\coprod D_{g,u,h,n}$ by some group action, unlike the space of smooth surfaces $\mo_{\tilde{g},n}$ which, as mentioned, is obtained as the quotient of $\coprod \ko_{g,u,h,n}$ by an action of $\mathbb{Z}_2^{\times n}$.
\end{remark}

We now come to our main result concerning the moduli space of Klein surface with oriented marked points. We have the following main result by Costello \cite{costello1,costello2}:

\begin{proposition}\label{prop:costello}
The inclusion $D_{g,h,n}\hookrightarrow\nb_{g,h,n}$ is a homotopy equivalence.
\end{proposition}

Our main result concerning the moduli space of Klein surfaces will follow from the Klein analogue of this:

\begin{proposition}\label{prop:main1}
The inclusion $D_{g,u,h,n}\hookrightarrow\kb_{g,u,h,n}$ is a homotopy equivalence.
\end{proposition}

\begin{proof}
By \autoref{lem:riemklein} with \autoref{prop:costello} this is clear if $u=0$. We therefore will restrict our attention to $u\neq 0$. Fortunately for us the proof of \autoref{prop:costello} carries over easily to a proof of this. As in that case, to prove \autoref{prop:main1} we first show the following:

\begin{lemma}\label{lem:incl1}
The inclusion $\partial\kb_{g,u,h,0}\hookrightarrow\kb_{g,u,h,0}$ is a homotopy equivalence of orbispaces.
\end{lemma}

\begin{proof}
The key idea of the proof is to construct a deformation retract of $\kb_{g,u,h,0}$ onto its boundary $\partial\kb_{g,u,h,0}$. This is done by using the hyperbolic metric on a Klein surface $K$ to flow the boundary $\partial K$ inwards until $K$ becomes singular. Some of the work involved can be avoided by passing to the orienting double and using the facts for Riemann surfaces proved in \cite{costello2}.

Let $K\in\ko_{g,u,h,0}$ be a Klein surface. Since we are assuming $u\geq 1$ and $(g,u,h)\neq(0,1,1)$ (since $\kb_{0,1,1,0}$ is empty) then the complex double of $K$ is a hyperbolic surface and so there is a unique hyperbolic metric on $K$ such that the boundary is geodesic. By taking the unit inward pointing normal vector field on $\partial K$ and using the geodesic flow on $K$ we can flow $\partial K$ inwards. Let $K_t$ be the surface with boundary obtained by flowing in $\partial K$ a distance $t$. Note that this process lifts to the orienting double $K_\mathbf{O}$ which is connected and corresponds to the process obtained using the hyperbolic metric on $K_\mathbf{O}$. In \cite[Lemma 3.0.8]{costello2} it is shown that this process applied to a Riemann surface eventually yields a singular surface and further that all the singularities are nodes. So by considering the orienting double we see the same is true for the Klein surface $K$. More precisely, for some $T$ we have $K_T\in\partial\kb_{g,u,h,0}$.

Let $S\in\mathbb{R}_{\geq 0}$ be the smallest number such that $K_S\in\partial\kb_{g,u,h,0}$ so that $K_t$ is in the interior $\ko_{g,u,h,0}$ for all $t<S$. We have a map $\Phi\co \ko_{g,u,h,0}\times[0,1]\rightarrow\kb_{g,u,h,0}$ defined by $\Phi(K,x)=K_{Sx}$. We extend this to a map $\Phi'\co \kb_{g,u,h,0}\times[0,1]\rightarrow\kb_{g,u,h,0}$ by setting $\Phi'(K',t)=K'$ for $K'\in\partial\kb_{g,u,h,0}$. To see this extends $\Phi$ continuously we take a sequence $K_i$ of surfaces converging to $K\in\partial\kb_{g,u,h,0}$ and let $x_i$ be any sequence. We must show $\Phi'(K_i,x_i)\rightarrow K$. Observing that, after forgetting the symmetry, $(K_i)_\mathbf{O}\rightarrow K_\mathbf{O}$ and comparing to the proof of \cite[Lemma 3.0.9]{costello2} this is clear. Then $\Phi'$ is a deformation retract of the inclusion as required.
\end{proof}

\begin{lemma}\label{lem:incl2}
The inclusion $\partial\kb_{g,u,h,n}\hookrightarrow\kb_{g,u,h,n}$ is a homotopy equivalence.
\end{lemma}

\begin{proof}
Since the moduli space of a M\"obius strip with an oriented marked point is the same as that of an annulus with a marked point then if $(g,u,h,n)=(0,1,1,1)$ \autoref{lem:incl2} can be seen directly.

There is a map $\kb_{g,u,h,n+1}\rightarrow\kb_{g,u,h,n}$ forgetting the last marked point and contracting any resulting unstable components which is a locally trivial fibration (in the orbispace sense). This follows by considering for some $K\in\kb_{g,u,h,n}$ the space of ways of adding an oriented marked point to $\partial K$. This is the same as the space of ways of adding a single marked point to $\partial K_\mathbf{O}$. Then by the same argument as \cite[Lemma 3.0.5]{costello2} it is clear that this map is a locally trivial fibration. Therefore if $\partial\kb_{g,u,h,n}\hookrightarrow\kb_{g,u,h,n}$ is a homotopy equivalence then so is $\partial\kb_{g,u,h,n+1}\hookrightarrow\kb_{g,u,h,n+1}$ and \autoref{lem:incl2} follows.
\end{proof}

We can now see that \autoref{prop:main1} follows from \autoref{lem:incl2} by an inductive argument like that in \cite[Lemma 3.0.12]{costello2}.
\end{proof}

We can also obtain such a result for the operad $\mb$. We consider $K\in\mb$ as a na\"ive nodal Klein surface and then consider the space of ways of adding a marked point to $\partial K$ by an identical argument to \cite[Lemma 3.0.9]{costello2} to see the map $\mb_{g,n+1}\rightarrow\mb_{g,n}$ forgetting the last marked point and stabilising is a locally trivial fibration in the orbispace sense.

Since $\mb_{\tilde{g},0}$ can be obtained from $\coprod \kb_{g,u,h,0}$ by identifying only points in $\coprod \partial \kb_{g,u,h,0}$ (by forgetting the dianalytic structure at nodes), it follows immediately from \autoref{lem:incl1} that there is a deformation retract of the map $\mb_{\tilde{g},0}\setminus\mo_{\tilde{g},0}\hookrightarrow \mb_{\tilde{g},0}$. Therefore by the same argument as above we can deduce the following:

\begin{proposition}\label{prop:main2}
The inclusion $\dr_{g,n}\hookrightarrow\mb_{g,n}$ is a homotopy equivalence.
\qed
\end{proposition}

We now obtain our main theorem immediately from \autoref{prop:main1} and \autoref{prop:main2}.

\begin{theorem}
\Needspace*{4\baselineskip}\mbox{}
\begin{itemize}
\item The inclusion $D\hookrightarrow\kb$ is a homotopy equivalence of extended topological modular operads.
\item The inclusion $\dr\hookrightarrow\mb$ is a homotopy equivalence of extended topological modular operads.
\qed
\end{itemize}
\end{theorem}

Given an appropriate chain complex $C_*$ (we take coefficients in $\mathbb{Q}$) with a K\"unneth map $C_*(X)\otimes C_*(Y)\rightarrow C_*(X\times Y)$ then $C_*(\kb)$ is an extended dg modular operad. An algebra over this is called an open KTCFT. Since the space of dianalytic ribbon graphs is obtained from the space of M\"obius graphs by forgetting the colouring we see that $C_*(\dr)=C_*(D)/(a=1)$ where $a\in C_*(D)((0,2))\cong\mathbb{Q}[\mathbb{Z}_2]$ is the involution. The above then translates into:

\begin{theorem}\label{thm:main}
There are quasi-isomorphism of extended dg modular operads over $\mathbb{Q}$
\begin{gather*}
C_*(D)\simeq C_*(\kb)\\
C_*(D)/(a=1)\simeq C_*(\mb)
\end{gather*}
where $a\in C_*(D)((0,2))\cong\mathbb{Q}[\mathbb{Z}_2]$ is the involution.
\qed
\end{theorem}

Since the spaces $D_{g,u,h,n}$ are orbi-cell complexes we can give a simple description for the operad $C_*(D)$ over $\mathbb{Q}$ using the cellular chain complex. We now identify this dg operad $C_*(D)$ and so relate our results to the previous sections.

\begin{proposition}\label{prop:cellchainiso}
There is an isomorphism
\[C_*(D)\cong\moddmass\]
(up to homological/cohomological grading).
\end{proposition}

\begin{remark}
By `up to homological/cohomological grading' we mean that $C_*(D)$ is graded homologically whereas $\moddmass$ is graded cohomologically. Given a cohomologically graded complex $V=\bigoplus V^i$ we set $V_{-i}=V^i$ to obtain an equivalent homologically graded complex. We can swap the grading of operads in this way.
\end{remark}

\begin{proof}
This follows from considering \autoref{prop:celldecomp} and \autoref{rem:cellattach} with \autoref{prop:orientedmobgraph}. The space $C_*(D_{g,u,h,n})$ is generated by oriented orbi-cells so a basis is given by reduced M\"obius graphs $G$ of topological type $(g,u,h,n)$ together with an orientation of the corresponding orbi-cell. An orientation can be given by an ordering of the vertices of $G$ and at each vertex an ordering of the set of half edges attached to it. It is clear by considering \autoref{rem:orientation} that orientations of the corresponding orbi-cell are equivalent to orientations on $G$ as defined earlier. Noting then that $C_*(D)$ is the modular closure of its genus $0$ part, it is not too difficult to check that the operad structures coincide.
\end{proof}

This is the Klein version of the fact by Costello \cite{costello1} that $\moddass$ gives a chain model for the homology of $\nb$ which gives the well known ribbon graph complexes computing the homology of the spaces $\no_{g,h,n}\simeq\nb_{g,h,n}$.

In our case this means we obtain M\"obius graph complexes computing the rational homology of the spaces $\ko_{g,u,h,n}\simeq\kb_{g,u,h,n}$, generated by oriented reduced M\"obius graphs with the differential expanding vertices of valence greater than $3$. When $n>0$ this computes the integral homology since $D_{g,u,h,n}$ is then an ordinary cell complex. When $u=0$ and $n>0$ this complex is a sum of ribbon graph complexes as expected. For $u=n=0$ this complex is the ribbon graph complex quotiented by the action of $\mathbb{Z}_2$ reversing the cyclic ordering at every vertex of a graph. For $u\neq 0$ we obtain combinatorially distinct complexes.

We can also obtain graph complexes for the rational homology of the spaces $\mo_{\tilde{g},n}\cong(\coprod\ko_{g,u,h,n})/\mathbb{Z}_2^{\times n}$ by forgetting the colours of the legs of M\"obius graphs. Additionally by forgetting the colours of all the half edges of M\"obius graphs we obtain graph complexes for the spaces $\mb_{\tilde{g},n}$. We will describe all these graph complexes concretely, without reference to operads, in the next section.

We finish this section with some observations. As already stated we have found two different ways of approaching the problem of allowing nodes on Klein surfaces. In the case of surfaces with oriented marked points we obtain a partial compactification that is homotopy equivalent to the space of smooth surfaces. We should note that $H_0(\kb)\cong\modmass$. In the second case the partial compactification is quite different.

Since the spaces $\mo_{\tilde{g},n}$ (which are in general not connected) are obtained from the disjoint union of the spaces $\ko_{g,u,h,n}$ modulo the action of the finite group $\mathbb{Z}_2^{\times n}$ then the non-zero degree rational homology of $\mo_{0,n}$ is trivial (as $\mass$ is Koszul).

There is a map of operads given by the composition $q\co \moddass\hookrightarrow\moddmass\twoheadrightarrow\moddmass/(a=1)$ and this map is surjective. Geometrically this corresponds to the fact that the spaces of $\dr$ are subspaces of the loci of admissible Riemann surfaces in $\mb$ having an orientable quotient (and so up to homotopy we do not need to worry about unorientable surfaces in $\mb$). It is easy to see that $H_0(\mb)\cong H_0(\dr)\cong\modcom$ and so the spaces $\mb_{\tilde{g},n}$ are connected.

More interestingly the spaces $\mb_{0,n}$ have non-trivial rational homology in higher degrees. To see this we first note that if $T\in\dmass/(a=1)$ is a cycle and $d(T')=T$ for some $T'$ then there is a $G\in\dass$ such that $q(G)=T'$ so $T=q(dG)$ and so the cycle $T$ lifts to a cycle $dG$ in $\dass$. Therefore if we find a cycle in $\dmass/(a=1)$ that does not lift to another cycle we know it gives a non-trivial homology class. It is easy to write down an example, see \autoref{fig:nontrivialclass}. This fact also justifies \autoref{rem:notgrpquotient}.

\begin{figure}[ht!]
\centering{\small
\begin{align*}
T\quad=\quad &+ \quad
{\xygraph{!{<0cm,0cm>;<0.6cm,0cm>:<0cm,0.6cm>::}
&&3="3"&4="4"\\
1="1"&2="2"&="w" \\
&="v"\\
&="0"
"v"-"1" "v"-"2" "v"-"0" "v"-"w" "w"-"3" "w"-"4"
}}
\quad - \quad
{\xygraph{!{<0cm,0cm>;<0.6cm,0cm>:<0cm,0.6cm>::}
&&3="3"&4="4"\\
2="2"&1="1"&="w" \\
&="v"\\
&="0"
"v"-"1" "v"-"2" "v"-"0" "v"-"w" "w"-"3" "w"-"4"
}}\\
\quad &+ \quad
{\xygraph{!{<0cm,0cm>;<0.6cm,0cm>:<0cm,0.6cm>::}
&1="1"&3="3"&4="4"\\
2="2"&&="w" \\
&="v"\\
&="0"
"v"-"2" "v"-"w" "v"-"0" "w"-"1" "w"-"3" "w"-"4"
}}
\quad - \quad
{\xygraph{!{<0cm,0cm>;<0.6cm,0cm>:<0cm,0.6cm>::}
&2="1"&3="3"&4="4"\\
1="2"&&="w" \\
&="v"\\
&="0"
"v"-"2" "v"-"w" "v"-"0" "w"-"1" "w"-"3" "w"-"4"
}}
\quad + \quad
{\xygraph{!{<0cm,0cm>;<0.6cm,0cm>:<0cm,0.6cm>::}
2="2"&&3="3"\\
1="1"&="w"&4="4" \\
&="v"\\
&="0"
"v"-"1" "v"-"4" "v"-"0" "v"-"w" "w"-"3" "w"-"2"
}}\\
\quad &- \quad
{\xygraph{!{<0cm,0cm>;<0.6cm,0cm>:<0cm,0.6cm>::}
1="2"&&3="3"\\
2="1"&="w"&4="4" \\
&="v"\\
&="0"
"v"-"1" "v"-"4" "v"-"0" "v"-"w" "w"-"3" "w"-"2"
}}
\quad + \quad
{\xygraph{!{<0cm,0cm>;<0.6cm,0cm>:<0cm,0.6cm>::}
2="2"&1="1"&3="3"\\
&="w"&&4="4"\\
&&="v"\\
&&="0"
"v"-"4" "v"-"w" "v"-"0" "w"-"1" "w"-"3" "w"-"2"
}}
\quad - \quad
{\xygraph{!{<0cm,0cm>;<0.6cm,0cm>:<0cm,0.6cm>::}
1="2"&2="1"&3="3"\\
&="w"&&4="4"\\
&&="v"\\
&&="0"
"v"-"4" "v"-"w" "v"-"0" "w"-"1" "w"-"3" "w"-"2"
}}
\end{align*}}
\caption{A non-trivial homology class in $H_1(\mb_{0,5})$. When the differential is applied to the first two terms of $T$ two of the resulting trees cancel as elements of $\dmass/(a=1)$ but not as elements of $\dass$ no matter how we lift $T$.}
\label{fig:nontrivialclass}
\end{figure}

\subsection{Summary of the associated graph complexes}
To make our results explicit we will finish by unwrapping \autoref{thm:main} and \autoref{prop:cellchainiso} and defining the graph complexes in a more explicit and straightforward manner, without reference to operads.

Recall an orientation of a graph $G$ is a choice of orientation of the vector space $\mathbb{Q}^{\edges(G)}\oplus H_1(|G|,\mathbb{Q})$. Denote by $\Gamma_{g,h,n}$ the vector space over $\mathbb{Q}$ generated by oriented reduced ribbon graphs of topological type $(g,h,n)$. That is, equivalence classes of pairs $(G,\mathrm{or})$ with $G$ a reduced ribbon graph of genus $g$ with $h$ boundary components and $n$ legs and $\mathrm{or}$ an orientation of $G$, subject to the relations $(G,-\mathrm{or})=-(G,\mathrm{or})$. 

Denote by $\mob\Gamma_{g,u,h,n}$ the vector space over $\mathbb{Q}$ generated by oriented reduced M\"obius graphs of topological type $(g,u,h,n)$. 

Denote by $\Gamma_{\tilde{g},n}$ the space 
\[\Gamma_{\tilde{g},n}=\bigoplus_{2g+h-1=\tilde{g}}\Gamma_{g,h,n}\]
and denote by $\mob\Gamma_{\tilde{g},n}$ the space:
\[\mob\Gamma_{\tilde{g},n}=\bigoplus_{2g+u+h-1=\tilde{g}}\mob\Gamma_{g,u,h,n}\]
The finite group $\mathbb{Z}_2^{\times n}$ acts on $\mob\Gamma_{\tilde{g},n}$ by switching the colours of legs.

Denote by $\Gamma^\mathbb{R}_{\tilde{g},n}$ the space of oriented reduced dianalytic ribbon graphs (see the proof of \autoref{prop:dianalyticgraphs}) of topological type $(\tilde{g},n)$. Observe that $\Gamma^\mathbb{R}_{\tilde{g},n}=\Gamma_{\tilde{g},n}/I=\mob\Gamma_{\tilde{g},n}/J$ where $I$ is the subspace generated by relations of the form $(G,\mathrm{or})=(H,\mathrm{or}')$ whenever $(G,\mathrm{or})$ is isomorphic to $(H,\mathrm{or}')$ after reversing the cyclic ordering at some of the vertices of $G$ if necessary and $J$ is the subspace generated by relations of the form $(G,\mathrm{or})=(H,\mathrm{or}')$ whenever $(G,\mathrm{or})$ is isomorphic to $(H,\mathrm{or}')$ after changing the colours of any of the half edges of $G$ if necessary. What this means is that $\Gamma^\mathbb{R}_{\tilde{g},n}$ is obtained from $\Gamma_{\tilde{g},n}$ by identifying cyclic orderings at vertices of ribbon graphs with the reverse cyclic orderings or that it is obtained from $\mob\Gamma_{\tilde{g},n}$ by forgetting the colourings of M\"obius graphs.

These spaces are finite dimensional and cohomologically graded by the number of internal edges in a graph. We define a differential on these spaces by
\[d(G,\mathrm{or})=\sum (G',\mathrm{or'})
\]
where the sum is taken over classes $(G',\mathrm{or'})$ arising from all ways of expanding one vertex of $G$ into two vertices each of valence at least $3$ so that $G'/e=G$ where $e$ is the new edge joining the two new vertices. The orientation $\mathrm{or'}$ is the product of the natural orientations on $\mathbb{Q}^{\edges(G')}\supset \mathbb{Q}^{\edges(G')\setminus e}$ and $H_1(|G'|)\cong H_1(|G'/e|)$.

We can then unwrap the main theorems.

\begin{theorem}
\Needspace*{4\baselineskip}\mbox{}
\begin{itemize}
\item There are isomorphisms:
\[
H_\bullet(\no_{g,h,n},\mathbb{Q})\cong H_\bullet(\nb_{g,h,n},\mathbb{Q})\cong H^{6g+3h+n-6-\bullet}(\Gamma_{g,h,n})
\]
Further for $n\geq 1$ such isomorphisms also hold for integral homology. This is the well known ribbon graph decomposition.
\item There are isomorphisms:
\[
H_\bullet(\ko_{g,u,h,n},\mathbb{Q})\cong H_\bullet(\kb_{g,u,h,n},\mathbb{Q})\cong H^{6g+3u+3h+n-6-\bullet}(\mob\Gamma_{g,u,h,n})
\]
Further for $n\geq 1$ such isomorphisms also hold for integral homology.
\item There are isomorphisms:
\[
H_\bullet(\mo_{\tilde{g},n},\mathbb{Q})\cong H^{3\tilde{g}+n-3-\bullet}(\mob\Gamma_{\tilde{g},n})/\mathbb{Z}_2^{\times n}
\]
\item There are isomorphisms:
\[
H_\bullet(\mb_{\tilde{g},n},\mathbb{Q})\cong H^{3\tilde{g}+n-3-\bullet}(\Gamma^\mathbb{R}_{\tilde{g},n})
\]
\end{itemize}
\end{theorem}

\bibliography{references}
\bibliographystyle{alphaurl}

\end{document}